\documentclass[11pt,oneside]{amsart}

\usepackage{amsmath}

\usepackage{tikz-cd}

\usepackage{fix-cm}

\usepackage[obeyspaces]{url}
\usepackage{hyperref}

\DeclareMathAlphabet{\mathcal}{OMS}{cmsy}{m}{n}

\usepackage[arrow,curve,matrix,tips,2cell]{xy}

\usepackage[hmargin=3.5cm,vmargin=1.5cm, bottom=50pt]{geometry}

\usepackage{anyfontsize}
\usepackage{mathrsfs}
\usepackage{bm}

\usepackage{bbm}

\usepackage{fancyhdr}

\usepackage{tensor}

\usepackage{mathtools}

\usepackage{accents}

\usepackage{amscd}
\usepackage{amssymb}
\usepackage{latexsym}
\usepackage{url}

\usepackage{graphicx}

\usepackage{enumitem}
\usepackage{quiver}

\theoremstyle{plain}
\newtheorem{theorem}{Theorem}[section]
\newtheorem*{theorem*}{Theorem}
\newtheorem{lemma}[theorem]{Lemma}
\newtheorem*{lemma*}{Lemma}
\newtheorem{corollary}[theorem]{Corollary}
\newtheorem{proposition}[theorem]{Proposition}

\theoremstyle{definition}
\newtheorem{remark}[theorem]{Remark}
\newtheorem{definition}[theorem]{Definition}
\newtheorem*{definition*}{Definition}

\newtheorem{question}[theorem]{Question}
\newtheorem*{question*}{Question}

\newtheorem{example}[theorem]{Example}
\newtheorem{examples}[theorem]{Examples}

\numberwithin{equation}{section}

\makeatletter
\def\revddots{\mathinner{\mkern1mu\raise\p@
\vbox{\kern7\p@\hbox{.}}\mkern2mu
\raise4\p@\hbox{.}\mkern2mu\raise7\p@\hbox{.}\mkern1mu}}
\makeatother % ddots gespiegelt
\newcommand{\bgl}{\begin{equation}} %eine Gleichung mit Ziffer
\newcommand{\egl}{\end{equation}}
\newcommand{\bgloz}{\begin{equation*}} %eine Gleichung ohne Ziffer
\newcommand{\egloz}{\end{equation*}}
\newcommand{\bgln}{\begin{eqnarray}} %mehrere Gleichungen mit Ziffer
\newcommand{\egln}{\end{eqnarray}}
\newcommand{\bglnoz}{\begin{eqnarray*}} %mehrere Gleichungen ohne Ziffer
\newcommand{\eglnoz}{\end{eqnarray*}}
\newcommand{\btheo}{\begin{theorem}}
\newcommand{\etheo}{\end{theorem}}
\newcommand{\btheooz}{\begin{theorem*}}
\newcommand{\etheooz}{\end{theorem*}}

\newcommand{\blemma}{\begin{lemma}}
\newcommand{\elemma}{\end{lemma}}
\newcommand{\blemmaoz}{\begin{lemma*}}
\newcommand{\elemmaoz}{\end{lemma*}}
\newcommand{\bproof}{\begin{proof}}
\newcommand{\eproof}{\end{proof}}
\newcommand{\bbew}{\begin{beweis}}
\newcommand{\ebew}{\end{beweis}}
\newcommand{\bremark}{\begin{remark}}
\newcommand{\eremark}{\end{remark}}
\newcommand{\bdefin}{\begin{definition}}
\newcommand{\edefin}{\end{definition}}
\newcommand{\bdefinoz}{\begin{definition*}}
\newcommand{\edefinoz}{\end{definition*}}
\newcommand{\bex}{\begin{example}}
\newcommand{\eex}{\end{example}}
\newcommand{\bexs}{\begin{examples}}
\newcommand{\eexs}{\end{examples}}

\newcommand{\bprop}{\begin{proposition}}
\newcommand{\eprop}{\end{proposition}}
\newcommand{\bcor}{\begin{corollary}}
\newcommand{\ecor}{\end{corollary}}
\newcommand{\bfa}{\begin{cases}} %Fallunterscheidung
\newcommand{\efa}{\end{cases}}

\newcommand{\bquestion}{\begin{question}}
\newcommand{\equestion}{\end{question}}
\newcommand{\bquestionoz}{\begin{question*}}
\newcommand{\equestionoz}{\end{question*}}
\newtheorem{introtheorem}{Theorem}

\newtheorem{introcor}[introtheorem]{Corollary}

\NewDocumentEnvironment{manual}{O{introtheorem}m}
 {%
  \addtocounter{introtheorem}{-1}%
  \renewcommand\theintrotheorem{#2}%
  \begin{#1}
 }
 {\end{#1}}
\makeatletter
\DeclareRobustCommand{\thmprime}{%
  \begingroup
  \expandafter\in@\expandafter b\expandafter{\f@series}%
  \ifin@ \boldmath \fi
  $\m@th{}^{\prime}$%
  \endgroup
}
\makeatother
\newcommand{\cE}{\mathcal E}
\newcommand{\cF}{\mathcal F}
\newcommand{\cG}{\mathcal G}

\newcommand{\cL}{\mathcal L}

\newcommand{\cO}{\mathcal O}

\newcommand{\cU}{\mathcal U}
\newcommand{\cV}{\mathcal V}

\newcommand{\cZ}{\mathcal Z}

\def\Cz{\mathbb{C}}

\def\Nz{\mathbb{N}}

\def\Lz{\mathbb{L}}

\def\Qz{\mathbb{Q}}
\def\Rz{\mathbb{R}}

\def\Zz{\mathbb{Z}}

\newcommand{\fC}{\mathfrak C}

\newcommand{\fS}{\mathfrak S}

\newcommand{\mfd}{\mathfrak d}

\newcommand{\mft}{\mathfrak t}

\newcommand{\mfv}{\mathfrak v}

\newcommand{\an}[1]{``#1''} % Anfuehrungsstriche
\newcommand{\ti}{\tilde}

\newcommand{\ma}{\mapsto} % wird abgebildet auf
\newcommand{\onto}{\twoheadrightarrow} % surjektiv
\newcommand{\into}{\hookrightarrow} % injektiv
\newcommand{\isom}{\xrightarrow{\raisebox{-1ex}[0ex][0ex]{$\sim$}}} % isom
\newcommand{\Rarr}{\Rightarrow} % Folgerung
\newcommand{\Larr}{\Leftarrow} % Folgerung
\def\SEMI{\mbox{$\times\kern-2pt\vrule height5pt width.6pt \kern3pt $}}

\newcommand{\Spec}{{\rm Spec\,}} % Spektrum

\newcommand{\id}{{\rm id}}

\DeclareMathOperator{\Ind}{Ind}

\DeclareMathOperator{\ab}{Ab}
\renewcommand{\dim}{{\rm dim}\,}
\newcommand{\rk}{{\rm rk}\,}
\newcommand{\img}{{\rm im\,}}
\newcommand{\reg}{^\times} % regulaer
\newcommand{\lspan}{{\rm span}} % linearer Aufspann
\newcommand{\clspan}{\overline{\lspan}} % Abschluss des linearen Auspanns
\newcommand{\ilim}{\varinjlim} % induktiver Limes
\newcommand{\lge}{\left\{} % links geschweift
\newcommand{\rge}{\right\}} % rechts geschweift
\newcommand{\lsp}{\left\langle} % links spitz
\newcommand{\rsp}{\right\rangle} % links spitz
\newcommand{\gekl}[1]{\lge #1 \rge} % geschweifte Klammer
\newcommand{\spkl}[1]{\lsp #1 \rsp} % spitze Klammer
\newcommand{\suchthat}{\;\ifnum\currentgrouptype=16 \middle\fi|\;}
\newcommand{\menge}[2]{\gekl{ #1 \suchthat #2 }} % Menge
\def\bf1{\mathbf{1}}
\DeclareMathOperator{\dom}{dom}
\DeclareMathOperator{\ran}{ran}

\newcommand{\bmx}{\bm{x}}

\newcommand{\bmF}{\bm{F}}

\newcommand{\oset}[2]{%
  \mathop{#2}\limits^{\vbox to -1.66ex{%
  \kern -1.4ex\hbox{$#1$}\vss}}}

\newcommand{\tipreceq}{\oset{\sim}{\preceq}}
\newcommand{\tiprec}{\oset{\sim}{\prec}}

\newcommand{\crL}{\mathscr{L}}

\newcommand{\R}{\mathscr{R}}

\newcommand{\Y}{\mathscr{Y}}

\DeclareMathOperator{\cs}{\mathrm{C}^\ast}
\DeclareMathOperator{\csr}{\mathrm{C}^\ast_r}
\newcommand{\Ktop}{\mathrm{K}^{\mathrm{top}}}
\DeclareMathOperator{\K}{\mathrm{K}}
\DeclareMathOperator{\Hlgy}{\mathrm{H}}
\DeclareMathOperator{\KK}{\mathrm{KK}}

\newcommand{\acts}{\mathbin{\curvearrowright}}
\newcommand{\rightacts}{\mathbin{\curvearrowleft}}
\newcommand{\down}{\mathop{\downarrow}}
\newcommand{\bigdown}{\mathop{\big\downarrow}}

\DeclareMathOperator{\Res}{Res}

\DeclareFontFamily{U}{mathb}{\hyphenchar\font45}
\DeclareFontShape{U}{mathb}{m}{n}{
      <5> <6> <7> <8> <9> <10>
      <10.95> <12> <14.4> <17.28> <20.74> <24.88>
      mathb10
      }{}
\DeclareSymbolFont{mathb}{U}{mathb}{m}{n}

\DeclareMathSymbol{\sqbullet}{1}{mathb}{"0D}
\usepackage{newtxtext}
\usepackage{newtxmath}

\begin{document}

\pretolerance=10000

\title{Discretisation and independent resolutions of ample groupoids}

\thispagestyle{fancy}

\author[X. Li]{Xin Li}
\address[X.~Li]{School of Mathematics and Statistics, University of Glasgow, University Place, Glasgow G12 8QQ, United Kingdom}
\email{Xin.Li@glasgow.ac.uk}

\author[A. Miller]{Alistair Miller}
\address[A.~Miller]{%
    Department of Mathematics\\
    KU Leuven\\
    Celestijnenlaan 200B\\
 box 2400\\
    3001 Leuven\\
    Belgium}
\email{alistair.miller@kuleuven.be}


\thanks{This project has received funding from the European Research Council (ERC) under the European Union's Horizon 2020 research and innovation programme (grant agreement No. 817597). The first author was supported by the Independent Research Fund Denmark through Grant 1054-00094B and by the Research Foundation Flanders (FWO) through Project 1212126N}

\begin{abstract}
We develop a general framework for understanding and computing both the groupoid homology of an ample groupoid and the topological K-theory of its reduced $\mathrm{C}^\ast$-algebra, based on two main ideas: discretisation and independent resolutions. Discretisation shows that a special class of ample groupoids we term \emph{independent} groupoids are homologically and K-theoretically equivalent to discrete groupoids.
We introduce the notion of a \emph{resolution} by independent groupoids and provide a recipe for building a controlled independent resolution of a given ample groupoid of interest, leading to a systematic way of studying its homology and K-theory. 
In order to illustrate our general ideas and methods, we work out several concrete examples and applications. Garside categories provide a wide range of examples, including higher rank graphs, self-similar groups and spherical Artin--Tits groups. We also present an application to the homology of Stein's groups. 
\end{abstract}

\maketitle

\tableofcontents

\setlength{\parindent}{0cm} \setlength{\parskip}{0.5cm}

\section{Introduction}
\label{s:intro}

Topological K-theory is an invariant for $\cs$-algebras which plays an important role in Elliott's classification programme \cite{GLNI,GLNII,GL,Win,Whi,CGSTW} as well as non-commutative geometry and its applications to topology and geometry in the context of the Baum--Connes conjecture \cite{BCH,Val}. The goal of the present paper is to develop a general framework which allows us to study and compute this fundamental invariant for $\cs$-algebras of ample groupoids, which is a rich class of $\cs$-algebras covering prominent examples such as AF algebras, Kirchberg algebras satisfying the UCT, Cuntz--Krieger algebras, graph $\cs$-algebras, higher rank graph $\cs$-algebras, semigroup $\cs$-algebras and $\cs$-algebras attached to self-similar groups. Our methods also provide an approach to understanding and determining groupoid homology as introduced in \cite{CM}, which is a fundamental invariant for ample groupoids that has been used successfully to classify topological dynamical systems \cite{GPS95,GMPS08,GMPS10} and exhibits interesting connections to the group homology of topological full groups \cite{Mat12,Mat15,Mat16,SW,Li25}. The latter is a class of groups of dynamical origin with striking properties \cite{JM,Nek18b,SWZ}.

Our framework is based on two main ingredients: discretisation and independent resolutions. 

The idea of discretisation in K-theory goes back to \cite{CEL1}, where the initial motivation was to compute K-theory for certain semigroup $\cs$-algebras. Based on \cite{CEL1}, a conceptual perspective involving topological dynamical systems was developed in \cite{CEL2}, generalizing K-theory formulas obtained in \cite{CEL1} to crossed products of certain group actions on totally disconnected spaces. An alternative perspective using the language of inverse semigroups and their $\cs$-algebras emerged from \cite{Nor}. Both perspectives have then been further developed in \cite{Li22}, which established discretisation for crossed products of certain partial group actions. More recently, \cite{Mil3} developed discretisation for more general groupoids without relying on the presence of partial group actions, but assuming that the groupoids of interest have torsion-free isotropy groups. In the algebraic context, discretisation has been formulated in \cite{ACM}.

Independent resolutions were first introduced in \cite{LNI,LNII}. Their purpose is to describe general ample groupoids in terms of groupoids which are covered by discretisation, enabling us to study and compute K-theory and groupoid homology. 

\begin{introtheorem}[Discretisation in K-theory, see Theorem~\ref{thm:Disc_K}]\label{t:introKtheory}
Let $G$ be the universal groupoid of a countable inverse semigroup $S$ with idempotent semilattice $E$ such that $G$ is Hausdorff and satisfies the Baum--Connes conjecture. Consider the action of $S$ on $E\reg = E \setminus \{0\}$ as a discrete space and assume further that the transformation groupoid $S \ltimes E\reg$ satisfies the Baum--Connes conjecture. Then 
\[ \K_*(\csr(G)) \cong \K_*(\csr(S \ltimes E\reg)). \]
\end{introtheorem}
This reduces the K-theory computation for $G$ to that of the discrete groupoid $S \ltimes E\reg$. The computation breaks down further by Morita equivalence invariance to a K-theory computation for the isotropy groups of the discrete groupoid. Explicitly, in this case we obtain 
\[ \K_*(\csr(G)) \cong  \bigoplus_{[e] \in S \backslash E\reg} \K_*(\csr(S_e)), \]
where $S_e = \{ s \in S \mid s^* s = s s^* = e \}$ is the maximal subgroup at a non-zero idempotent $e \in E\reg$. Note that $S \ltimes E\reg$ satisfies the Baum--Connes conjecture if and only if $S_e$ satisfies the Baum--Connes conjecture for each $e \in E\reg$.

It is worth pointing out that the K-theoretic isomorphism in Theorem~\ref{t:introKtheory} is induced by a concrete \'etale groupoid correspondence (see Definitions~\ref{d:correspondence} and \ref{d:discretisation}).

Previous work only established this discretisation formula for special inverse semigroups (strongly $0$-$E$-unitary ones, see \cite{Li22}) or needed the extra assumption that the isotropy groups of $G$ are torsion-free (see \cite{Mil3}). In contrast, our discretisation result (Theorem~\ref{thm:Disc_K}) is close to the optimal level of generality. The Baum--Connes conjecture is used because in the actual proof, an isomorphism is constructed for the left-hand side of the Baum--Connes conjecture (the reader may consult \cite{Tu,Tu2,PY22,BP24,PY25} and the references therein for more information about the Baum--Connes conjecture for groupoids). This explains the assumptions \an{countable} and \an{Hausdorff}, because there is currently no suitable formulation of the Baum--Connes conjecture available without these hypotheses. For example, if a second countable Hausdorff \'etale groupoid $G$ is amenable or has the Haagerup property, then the Baum--Connes conjecture is valid for $G$ by \cite{Tu}. We note however that discretisation is still natural and desirable in the non-Hausdorff setting, and is known to hold for some examples (such as in the setting of Example \ref{ex:ssg} as used in \cite{MS}).

To apply Theorem \ref{t:introKtheory} in practical circumstances it is useful to be able to recognise when an ample groupoid $G$ can be realised as the universal groupoid of an inverse semigroup, and what data one should have to perform this realisation. Suppose $G = S \ltimes X$ for some inverse semigroup action $S \acts X$. A regular basis \cite[\S 2]{CEL2} is a set $\mathcal U$ of non-empty compact open sets in $X$ such that $\mathcal U \cup \{\emptyset\}$ is closed under intersection, generates the topology of $X$ in a suitable sense and satisfies the \emph{independence} condition: if $U, U_1, \dots, U_n \in \mathcal U$ satisfy $U \subseteq \bigcup_{i=1}^n U_i$, then $U \subseteq U_i$ for some $i$. 
If $X$ admits an $S$-invariant regular basis, then $G = S \ltimes X$ may be identified with the universal groupoid of an inverse semigroup. Thus Theorem \ref{t:introKtheory} admits the following dynamical reformulation.

\begin{manual}[introtheorem]{A\thmprime}[Discretisation in K-theory, a dynamical perspective]\label{t:introKtheorydynamical}
Suppose $S \acts X$ is an action of a countable inverse semigroup $S$ on a second countable totally disconnected locally compact Hausdorff space $X$ such that $G = S \ltimes X$ is Hausdorff and satisfies the Baum--Connes conjecture. Suppose that $X$ admits an $S$-invariant regular basis $\mathcal U$ of compact open sets. Consider the action of $S$ on $\mathcal U$ as a discrete space and suppose that for each $U \in \mathcal U$ the isotropy group $S_U = (S \ltimes \mathcal U)^U_U$ satisfies the Baum--Connes conjecture. Then 
\[ \K_*(\csr(G)) \cong \K_*(\csr(S \ltimes \mathcal U)) \cong \bigoplus_{[U] \in S \backslash \mathcal U} \K_*(\csr(S_U)). \]
\end{manual}

Groupoids which admit an $S$-invariant regular basis are examples of a class we term \emph{independent} groupoids (see Definition \ref{d:independent}). However, this is a very special class of groupoids, and an ample groupoid $G$ of interest will likely not be independent. Instead, we look for a \emph{resolution} of $G$ by independent groupoids. This is an ``exact'' sequence of independent groupoids $(G_k = S_k \ltimes X_k)_{k \geq 0}$ 
\[ \cdots \to G_k \to \cdots \to G_0 \to G \to \emptyset, \] 
i.e. we have closed invariant subsets $C_k \subseteq X_k$ together with isomorphisms $G \cong G_0 |_{C_0}$, $G_k|_{X_k \setminus C_k} \cong G_{k+1}|_{C_{k+1}}$ (see Definition \ref{d:resolution}). Under additional exactness assumptions, an independent resolution produces an exact sequence at the level of reduced groupoid $\cs$-algebras. When the resolution has finite length (i.e. eventually $G_k = \emptyset$), a series of six-term exact sequences relates the K-theory of $G$ to that of the independent groupoids $G_k$, which can be computed through Theorem \ref{t:introKtheory}. Independent resolutions always exist for abstract reasons, but this does not guarantee finite length resolutions.

Dynamical systems $S \acts X$ may often be studied through a naturally occurring $S$-invariant collection of compact open subsets in $X$. A \emph{representation} (see Definition \ref{def:GenIndLFWS}) of a semilattice $E$ on $X$ is a homomorphism 
\[ U \colon E \to \cO_c(X) \]
sending $0$ to $\emptyset$ which realises each $e \in E$ as a compact open set $U(e)$ in $X$. We consider representations $U \colon E \to \cO_c(X)$ of semilattices $E$ equipped with an action of $S$ by partial order automorphisms (see Definition \ref{d:S acts on L}) such that $U$ is $S$-equivariant; if $U$ is faithful and its image is $S$-invariant, $E$ will inherit such an action. It is reasonable to hope that one's favourite dynamical system $S \acts X$ admits a naturally occurring $S$-equivariant representation $U \colon E \to \cO_c(X)$ of a semilattice $E$ which generates the set of compact open sets in $X$ (we say $U$ is \emph{generating}), but in many cases $U$ will fail the following independence condition:
\begin{equation}\label{eq:independenceintro}
\text{if $e, e_1, \dots, e_n \in E$ satisfy $U(e) \subseteq \bigcup_{i=1}^n U(e_i)$, then $e \leq e_i$ for some $i$.}
\end{equation}
The key idea that we exploit to build independent resolutions is that in many natural examples, the independence condition \eqref{eq:independenceintro} will fail in a systematic way. A \emph{concrete finite cover} of $e \in E$ is a finite set $F \subseteq E$ below $e$ with $U(e) = \bigcup_{f \in F} U(f)$. Each failure of the independence condition is witnessed by a (non-trivial) concrete finite cover. We consider \emph{systems}
\[ \R = \{ \R(e) \mid e \in E \}, \]
of finite covers, where each $\R(e)$ is a set of concrete finite covers of $e$ (see Definition \ref{d:system}). It turns out that an $S$-invariant system $\R$ which is large enough to explain every failure of independence (we say $\R$ is \emph{thorough}, see Definition \ref{d:thorough}) contains enough data to explicitly construct an independent resolution of $G = S \ltimes X$. This gives us control over the independent resolution in terms of the system $\R$ of finite covers, and in particular enables us to construct finite length independent resolutions. 

\begin{introtheorem}[Independent resolutions, see Theorem \ref{t:resolution}]\label{t:introresolutions}
Let $G = S \ltimes X$ be an ample groupoid and let $U \colon E \to \mathcal O_c(X)$ be an $S$-equivariant generating representation of a semilattice $E$ by compact open subsets in $X$. Let $\R$ be a thorough $S$-invariant system of concrete finite covers in $E$ with respect to $U$. Then there is an independent resolution 
\[ \cdots \to G_k \to \cdots \to G_0 \to G \to \emptyset \]
of $G$ such that for each $k \geq 0$, $G_k = S \ltimes X_k$ for an $S$-space $X_k$ admitting an $S$-invariant regular basis, and $G_{k+1} = \emptyset$ for $k \geq \sup_{e \in E} \# \{F \in \R(e) \mid F \text{ non-trivial} \} $.
\end{introtheorem}

For groupoids of source-free row-finite higher rank graphs (see \cite{KuPa} and also \cite{RSY03,RSY04}), we always obtain finite length independent resolutions, see Example \ref{ex:HRG5}. Combined with discretisation, our work provides an alternative perspective (see Example \ref{ex:HRG6}) on K-theory computations carried out in \cite{FKPS}. For groupoids attached to self-similar groups, K-theory computations have been worked out in \cite{Nek,MS}. In this case, we obtain an independent resolution of length $1$ (see Example \ref{ex:ssg}), which combined with discretisation leads to a conceptual explanation of the results in \cite{Nek,MS}. We note that higher rank graphs only fit into the precise setup of Theorem \ref{t:introresolutions} if they are singly aligned. Following \cite[\S~8]{Mil}, to treat a more general finitely aligned setup, we replace semilattices with a more general notion of \emph{locally finite weak semilattices} (see \S~\ref{ss:LFWS}). Theorems \ref{t:introKtheory}, \ref{t:introKtheorydynamical} and \ref{t:introresolutions} are all established in the context of locally finite weak semilattices in the main text.

We highlight that Theorem \ref{t:introresolutions} not only helps us to understand ample groupoids, but also discrete groups. An \emph{Artin--Tits group} $\Gamma$ is a group given by generators $S$ and relations $R$ of the form \[\underbrace{stst\cdots}_{\text{$m_{s,t}$ elements}} = \underbrace{tsts\cdots}_{\text{$m_{s,t}$ elements}}\] for parameters $m_{s,t} = m_{t,s} \in \{2,3,\dots,\infty\}$ for each $s,t \in S$, with the convention that $m_{s,t} = \infty$ imposes no relation. The monoid $M_\Gamma = \langle S \mid R \rangle^+$ with the same presentation is the corresponding \emph{Artin--Tits monoid}. The Artin--Tits group $\Gamma$ is said to be \emph{spherical} if the quotient group (the \emph{Coxeter group}) determined by the relations $s^2 = 1$ is finite. In this case $M_\Gamma$ embeds in $\Gamma$. This gives rise to the semilattice $E_0 \coloneq \menge{\gamma M_{\Gamma}}{\gamma \in \Gamma} \cup \{\emptyset\}$ of subsets of $\Gamma$, where the partial order is given by inclusion. $\Gamma$ acts on $E_0$ by order automorphisms via left multiplication. Theorem \ref{t:introresolutions} applies to the trivial representation of $E_0$ on the trivial $\Gamma$-space $\{*\}$ together with a naturally constructed system $\R$ of finite covers, yielding a finite length independent resolution of $\Gamma = \Gamma \ltimes \{*\}$ with trivial isotropy groups. We obtain the following corollary.

We say the \emph{discretisation formula with arbitrary coefficients} is valid for a countable discrete group $\Gamma$ if for every countable semilattice $E$, every $\Gamma$-action $\Gamma \acts E$ by order automorphisms and every separable $\Gamma$-$\cs$-algebra $A$, there is an isomorphism
\[
 \K_*(\Gamma \ltimes_r (A \otimes C_0(E\reg)) ) \cong \K_*( \Gamma \ltimes_r (A \otimes C_0(\widehat{E})) ).
\]
Here $\widehat E$ is the character space of $E$ and $\Gamma$ acts diagonally on $A \otimes C_0(E\reg)$ and $A \otimes C_0(\widehat{E})$.

\begin{introcor}[see Corollary~\ref{cor:Garside_BC} and Example~\ref{ex:Artin}]\label{cor:Gars_BC_intro}
Let $\Gamma'$ be a subgroup of a spherical Artin--Tits group. Then $\Gamma'$ satisfies the Baum--Connes conjecture with arbitrary coefficients if and only if the discretisation formula with arbitrary coefficients is valid for $\Gamma'$.
\end{introcor}
\setlength{\parindent}{0cm} \setlength{\parskip}{0cm}
Actually, we prove this result in the more general case of exact subgroups of enveloping groups of Garside monoids (Corollary~\ref{cor:Garside_BC}). 
Note that Braid groups --- a particular case of spherical Artin--Tits groups --- are shown to satisfy the Baum--Connes conjecture with arbitrary coefficients in \cite{Schick}. However, the Baum--Connes conjecture with coefficients is open for general spherical Artin--Tits groups. (For related results on the Baum--Connes conjecture for Artin--Tits groups, see \cite{HO} and the references therein.) 
\setlength{\parindent}{0cm} \setlength{\parskip}{0.5cm}

All of the examples mentioned so far are particular cases of groupoids attached to Garside categories (see Definitions~\ref{d:Garside1} and \ref{d:Garside2}) as in \cite{Li21a,Li21b}. For these groupoids, we obtain finite length independent resolutions as soon as the underlying Garside category satisfies appropriate finiteness conditions. In this way, our work provides a unifying perspective on all the examples above. 

Let us turn to discretisation in groupoid homology. In this setting, unlike for K-theory, no extra hypotheses (such as satisfaction of the Baum--Connes conjecture, Hausdorffness, or finiteness requirements) are needed, and no exactness is required in applications of independent resolutions. The following discretisation formula is proved in \cite[Example 3.10]{Mil2}.

\begin{introtheorem}[Discretisation in homology]\label{t:introHomology}
Let $G$ be the universal groupoid of a countable inverse semigroup $S$ with idempotent semilattice $E$. Then 
\[ \Hlgy_*(G) \cong \Hlgy_*(S \ltimes E\reg) \cong \bigoplus_{[e] \in S \backslash E\reg} \Hlgy_*(S_e), \]
where $S_e = \{ s \in S \mid s^* s = s s^* = e \}$ is the maximal subgroup at a non-zero idempotent $e \in E\reg$.
\end{introtheorem}

We also extend this to the more general setting of locally finite weak semilattices in \S~\ref{ss:Disc_H}, following the same proof strategy. As with discretisation in K-theory, discretisation in homology admits a dynamical reformulation.

\begin{manual}[introtheorem]{E\thmprime}[Discretisation in homology, a dynamical perspective]
Let $G = S \ltimes X$ be an ample groupoid which admits an $S$-invariant regular basis $\mathcal U$ of compact open sets in $X$. Then 
\[ \Hlgy_*(G) \cong \Hlgy_*(S \ltimes \mathcal U) \cong \bigoplus_{[U] \in S \backslash \mathcal U} \Hlgy_*(S_U), \]
where $S_U = (S \ltimes \mathcal U)^U_U$ is the isotropy group at $U \in \mathcal U$.
\end{manual}

The independent resolutions we construct for higher rank graphs and for self-similar groups (see Examples \ref{ex:HRG5}, \ref{ex:HRG6} and \ref{ex:ssg}) also provide an alternative perspective on the groupoid homology computations undertaken in \cite{FKPS} and \cite{Nek,MS} respectively. 

One motivation for studying groupoid homology is that it leads to applications to group homology of topological full groups via \cite{Li25}. Stein's groups \cite{Ste92} are generalizations of Thompson groups and Higman--Thompson groups, where we are allowed to vary the conditions on slopes and breakpoints as well as the length of the underlying interval. Stein's groups were realised as the topological full groups $\bmF(G_{\Pi,A})$ of ample groupoids $G_{\Pi,A}$ by Tanner \cite{Tan} (see also \cite{Tanarxiv}), based on ideas in \cite{Li15}, and groupoid homology and K-theory computations have been carried out in particular cases in \cite{Tan,Li15}. The groupoid perspective was used by Matui to classify Stein's groups in \cite{Mat25}. Interestingly, in this case the groupoids $G_{\Pi,A}$ fit into an independent co-resolution of length $1$, in the sense that they can be described as complements of the form $G_0 \setminus G_{-1}$, where $G_0$ and $G_{-1}$ are independent ample groupoids, and $G_{-1}$ is the restriction of $G_0$ to a closed invariant subset of $G_0^{(0)}$. This, combined with discretisation, leads to concrete spectral sequences computing groupoid homology and K-theory for underlying groupoids of Stein's groups. In particular, we obtain the following concrete, sufficient criteria for Stein's groups to be rationally acyclic. 
The parameters $(\Pi,A)$ for the underlying groupoids $G_{\Pi,A}$ of Stein's groups are a multiplicative subgroup $\Pi$ of $\Rz_{>0}$ and a subgroup $A \subseteq (\Rz,+)$ such that $\Pi$ acts on the monoid $A_+ := A \cap \Rz_{\geq 0}$ by multiplication.
\begin{introcor}[see Corollary~\ref{cor:ConcCompStein}]
If one of the generators of $\Pi$ can be chosen as a number in $\Qz_{>0} \setminus \gekl{1} \subseteq \Rz_{>0} \setminus \gekl{1}$, then $\Hlgy_n(G_{\Pi,A}) \otimes \Qz = 0$ for all $n$, and thus the corresponding Stein's group $\bmF(G_{\Pi,A})$ is rationally acyclic.
\end{introcor}
\setlength{\parindent}{0cm} \setlength{\parskip}{0cm}

We also obtain sufficient criteria for Stein's groups to be integrally acyclic.
\setlength{\parindent}{0cm} \setlength{\parskip}{0.5cm}

Let us now present the main ideas behind the proofs of discretisation and independent resolutions.

The proof strategy for Theorem~\ref{t:introKtheory} is similar to the one developed in \cite{Mil3}, which is based on the ABC spectral sequence and its functoriality with respect to groupoid correspondences. In \cite{Mil3}, a spectral sequence argument then reduces Theorem~\ref{t:introKtheory} to discretisation in groupoid homology (this reduction uses the torsion-free isotropy assumption), which is then established by analysing chain complexes. In general, i.e., without the torsion-free isotropy assumption, we have to work with spectral sequences involving more complicated invariants, and the challenge we overcome in this paper is to establish the desired isomorphisms at the level of these more sophisticated analogues of groupoid homology.

In the context of Theorem \ref{t:introresolutions}, the main novelty of our paper is the development of a systematic way of producing systems of finite covers which lead to independent resolutions recursively following a concrete algorithm. At the technical level, the core exactness result is Theorem~\ref{thm:IndRes}, which makes sure that our algorithm does indeed produce an independent resolution, and which is the key improvement compared to previous results in \cite{LNI, LNII}, where such an exactness result was only established under strong assumptions. We have also developed a systematic approach to finding the data needed to construct an independent resolution, in the language of equivariant representations of semilattices and more generally locally finite weak semilattices. In particular, the notion of a thorough system of finite covers (Definition \ref{d:thorough}) provides a practical way to check that a system of finite covers produces a resolution of the desired groupoid (see Lemma \ref{l:thorough}).

The present paper is structured as follows: In Section~\ref{s:LFWS}, we recall the notion of locally finite weak semilattices from \cite[\S~8]{Mil}, study an equivariant version and develop a new, alternative perspective in terms of topological dynamical systems. Independent resolutions are discussed in Section~\ref{s:IndRep}, where we also develop our general construction based on systems of finite covers. Again, we work out an interpretation from the point of view of topological dynamics. Throughout Sections~\ref{s:LFWS} and \ref{s:IndRep}, we use the concrete example of higher rank graphs and their groupoids to illustrate our ideas. Section~\ref{s:Disc} treats the crucial idea of discretisation. First, we establish discretisation in groupoid homology in Section~\ref{ss:Disc_H}. Secondly, we formulate discretisation in K-theory in Section~\ref{ss:Disc_K}. Since the proof of K-theoretic discretisation is long and technical, it occupies a separate section (Section~\ref{proof section}). Finally, applications and examples are presented in Section~\ref{s:AppEx}. In Appendix~\ref{correspondence functor}, we explain how groupoid correspondences lead to $\cs$-algebraic correspondences in an equivariant setting. This is needed in Section~\ref{proof section}.

\section{Inverse semigroup actions on locally finite weak semilattices}
\label{s:LFWS}

In this section, we recall the notion of locally finite weak semilattices (following \cite[\S~8]{Mil}) and explain how to construct $\cs$-algebras from these structures. We then develop an alternative, dual picture, which emphasizes the topological perspective. Moreover, we consider the equivariant setting by adding an inverse semigroup action.

First we fix notation and conventions for inverse semigroups and \'etale groupoids. An \'etale groupoid $G$ is a topological groupoid whose range and source maps are local homeomorphisms; in this article the unit space $G^0$ is assumed to be locally compact and Hausdorff, but $G$ may be non-Hausdorff. Inverse semigroups and semilattices are assumed to have a $0$ element, and we denote the semilattice of idempotents in an inverse semigroup $S$ by $E(S)$. There is a natural partial order $\leq$ on $S$ given by $s \leq t$ if $s \in tE(S)$, or equivalently $s \in E(S)t$. Inverse semigroups $S$ act on sets $X$ by partial bijections; each $s \in S$ acts as a bijection written $x \mapsto s . x$ from its domain $\dom_X(s) \subseteq X$ to its range $\ran_x(s) \subseteq X$, and $\dom_X(0) = \ran_X(0) = \emptyset$. An action of $S$ on a locally compact Hausdorff space $X$ is by partial homeomorphisms; in this case $\dom_X(s)$ and $\ran_X(s)$ are assumed open and $s \colon \dom_X(s) \to \ran_X(s)$ a homeomorphism. The transformation groupoid $S \ltimes X$ is an \'etale groupoid with unit space $X$ given by
\[
 S \ltimes X \coloneq \menge{(s,x) \in S \times X}{x \in \dom_{X}(s)} / {\sim},
\]
where $(s_1,x_1) \sim (s_2,x_2)$ if and only if $x_1 = x_2$ and there exists $s \leq s_1, s_2$ with $x_1 = x_2 \in \dom_{X}(s)$. We denote the class of $(s,x)$ by $[s,x]$, which has range $s . x$ and source $x$. Multiplication in $S \ltimes X$, if defined, is given by $[t,s.x][s,x] = [ts,x]$.

\subsection{Locally finite weak semilattices}
\label{ss:LFWS}

\bdefin
A \emph{weak semilattice} $L$ is a poset with a zero element $0$ (a necessarily unique element with the property that $0 \leq l$ for all $l \in L$) such that for all $l_1, l_2 \in L$ there exists a finite subset $F \subseteq \menge{l \in L}{l \leq l_1, \, l \leq l_2}$ such that for all $l \in L$ with $l \leq l_1$, $l \leq l_2$, there exists $f \in F$ with $l \leq f$. Define $l_1 \down l_2$ as the set of maximal elements in $\menge{l \in L}{l \leq l_1, \, l \leq l_2}$. 
\setlength{\parindent}{0.5cm} \setlength{\parskip}{0cm}

A subset $\bar{F}$ of a weak semilattice is called $\down$-closed if for all $x, y \in \bar{F}$, $x \down y \subseteq \bar{F}$. We call a weak semilattice $L$ \emph{locally finite} if for any finite subset $F \subseteq L$ there exists a finite subset $\bar{F} \subseteq L$ which is $\down$-closed such that $F \subseteq \bar{F}$.
\edefin
\setlength{\parindent}{0cm} \setlength{\parskip}{0cm}

Note that in a weak semilattice $L$, $l_1 \down l_2$ is always a finite set for all $l_1, l_2 \in L$. 
\setlength{\parindent}{0cm} \setlength{\parskip}{0.5cm}

In the following, we will use the following notation: Given a poset $L$ with zero element $0$, and a subset $F \subseteq L$, we write $F\reg \coloneq F \setminus \gekl{0}$. A subset of $L$ is called a down-set if whenever $l$ is in the subset and $k \in L$ satisfies $k \leq l$, then $k$ has to be in the subset. Moreover, for $x$ and $y$ in $L$, we write $x \perp y$ if the only element below $x$ and $y$ is $0$. In that case, we say that $x$ and $y$ are disjoint or orthogonal. This happens in a weak semilattice if and only if $x \down y = \{ 0 \}$. Given a finite subset $F$ of a weak semilattice $L$, we write $\bigdown_{f \in F} \, f$ for the set of maximal elements in $\menge{k \in L}{k \leq f \text{ for all } f \in F}$.

\bremark
A semilattice $E$ is always a locally finite weak semilattice because for any $e,f \in E$ we have $e \down f = \gekl{ef}$, and for any finite subset $F \subseteq E$ the set $\bar{F}$ of finite products of elements in $F$ is $\down$-closed.
\eremark

In \S~\ref{s:LFWS} and \S~\ref{s:IndRep}, we will illustrate our concepts and results using higher rank graphs as examples, which we will develop step by step. 
\bex
\label{ex:HRG1}
Let $\Zz_{\geq 0} = \gekl{0, 1, 2, \dotsc}$ be the additive monoid of non-negative integers, and let $k$ be an integer with $k \geq 1$. A \emph{higher rank graph} $\Lambda$ of rank $k$ (or \emph{$k$-graph}) is a small category with a $\Zz_{\geq 0}^k$-valued degree map, i.e., a functor $d \colon \Lambda \to \Zz_{\geq 0}^k$ such that the following \emph{unique factorisation property} holds: For all $\nu \in \Lambda$ with $d(\nu) = pq$, there exist unique elements $\lambda, \mu \in \Lambda$ with $\nu = \lambda \mu$, $d(\lambda) = p$ and $d(\mu) = q$.
\setlength{\parindent}{0.5cm} \setlength{\parskip}{0cm}

$\cs$-algebras attached to higher rank graphs have been introduced in \cite{KuPa} (see also \cite{RSY03,RSY04}).

Set $L_{\Lambda} \coloneq \Lambda \cup \{ 0 \}$, ordered by $\lambda \leq \mu$ if $\lambda \Lambda \subseteq \mu \Lambda$, i.e. there is $\lambda' \in \Lambda$ with $\lambda \lambda' = \mu$.\footnote{This has the slightly awkward consequence that the degree map $d$ reverses inequalities, i.e. $\lambda \leq \mu$ implies $d(\lambda) \geq d(\mu)$.} The higher rank graph $\Lambda$ is called finitely aligned if $L_{\Lambda}$ is a weak semilattice. In this case it is automatically locally finite; let us explain why. For $\lambda, \mu \in \Lambda$, the set $\lambda \down \mu $ consists of minimal common extensions of $\lambda$ and $\mu$, which are all orthogonal. That is, for distinct $\nu, \xi \in \lambda \down \mu $, we have $\nu \perp \xi$. Minimal common extensions $\nu$ and $\xi$ of $\lambda$ and $\mu$ must satisfy $d(\nu) = d(\xi) = d(\lambda) \vee d(\mu)$ (where $\vee$ denotes the least common upper bound in the monoid $\Zz_{\geq 0}^k$), and if $\nu$ and $\xi$ have a common extension then $\nu = \xi$ by the unique factorisation property. Now, given a finite set $F \subseteq \Lambda$, consider the finite set 
\[\bar{F} = \bigcup_{A \subseteq F}  \bigdown_{f \in A} \, f. \]  
This contains $F$ and is $\down$-closed because, given $\lambda \in \bigdown_{f \in A} \, f$ and $\lambda' \in \bigdown_{f' \in A'} \, f'$, we have $\lambda \down \lambda' \subseteq \bigdown_{g \in A \cup A'} \, g$.

Given a $k$-graph $\Lambda$ and an element $\lambda \in \Lambda$, let $\mfd(\lambda)$ denote the domain of $\lambda$ and $\mft(\lambda)$ the codomain of $\lambda$. Given an object $v$ in $\Lambda$, note that $v \Lambda = \mft^{-1}(v)$. Moreover, let $e_i \in \Zz_{\geq 0}^k$ be the $i$-th generator (for $1 \leq i \leq k$) and define 
\[ \Lambda^{(i)} := d^{-1}(e_i). \]
The $k$-graph $\Lambda$ is \emph{row-finite} if for every object $v$ in $\Lambda$ and $1 \leq i \leq k$, we have $\# \, v \Lambda^{(i)} < \infty$. In this case, $\Lambda$ is finitely aligned. It is \emph{source-free} if $0 < \# \, v \Lambda^{(i)}$ for each $i$. 
\eex
\setlength{\parindent}{0cm} \setlength{\parskip}{0.5cm}

Next, let us define the $\cs$-algebra of a locally finite weak semilattice. We need the following notation. Let $D$ be a commutative algebra and $F$ a finite set of idempotents in $D$. Define
\[
 \bigvee_{p \in F} p \coloneq \sum_{\emptyset \neq F' \subseteq F} (-1)^{\# \, F' - 1} \prod_{p \in F'} p.
\]
In other words, $\bigvee_{p \in F} p$ is the smallest idempotent $q$ in $D$ such that $p \leq q$ for all $p \in F$.

\bdefin
Let $L$ be a locally finite weak semilattice. The $\cs$-algebra $\cs(L)$ of $L$ is the universal $\cs$-algebra which is commutative and generated by projections $\menge{p_l}{l \in L}$ subject to the relations $p_0 = 0$ and 
\[
 p_k p_l = \bigvee_{m \in k \down l} p_m \quad \quad \quad \text{ for all } k, l \in L.
\]
\edefin

To see that $\cs(L)$ is non-trivial, let us construct a concrete representation. Consider the Hilbert space $\ell^2(L\reg)$ with canonical orthonormal basis $\menge{\delta_x}{x \in L\reg}$. Given $l \in L$, let $P_l$ be the orthogonal projection onto the closed subspace $\clspan(\menge{\delta_k}{k \leq l})$. Let $\csr(L) \coloneq \cs(\menge{P_l}{l \in L}) \subseteq \cL(\ell^2(L\reg))$. Clearly, $P_0 = 0$ and $P_k P_l = \bigvee_{m \in k \down l} P_m$ for all $k, l \in L$. Hence, by universal property of $\cs(L)$, there exists a homomorphism $\pi_r \colon \cs(L) \to \csr(L)$ sending $p_l$ to $P_l$ for all $l \in L$.

Let us now study $\cs(L)$ in more detail.
\blemma
\label{lem:ZF}
Let $L$ be a locally finite weak semilattice and $\bar{F}$ a $\down$-closed finite subset of $L$. For every subset $F \subseteq \bar{F}$, we have $\bigvee_{k \in F} \, p_k \in \Zz \text{-} \lspan(\menge{p_l}{l \in \bar{F}})$.
\elemma
\setlength{\parindent}{0cm} \setlength{\parskip}{0cm}

\bproof
Given $l \in \bar{F}$, define $\ell(l) \coloneq \#( \menge{k \in \bar{F}\reg}{k \leq l} )$, and set $\ell(F) \coloneq \max \menge{\ell(l)}{l \in F}$. We proceed by induction on $\min(\# \, F, \ell(F))$. If $\ell(F) = 1$, then all $p_k$ (for $k \in F$) are pairwise disjoint, so that $\bigvee_{k \in F} \, p_k = \sum_{k \in F} p_k$. If $\# \, F = 1$, then there is nothing to prove. For the induction step, write $F = \gekl{k} \cup F'$. Then $q \coloneq \bigvee_{k' \in F'} p_{k'}$ lies in $\Zz \text{-} \lspan(\menge{p_l}{l \in \bar{F}})$. Moreover,
\[
 \bigvee_{l \in F} p_l = p_k \vee q = p_k + q - p_k q = p_k + q - \bigvee_{k' \in F'} p_k p_{k'} = p_k + q - \bigvee_{k' \in F'} \bigvee_{j \in k \down k'} \, p_j.
\]
As every $j \in \bigcup_{k' \in F'} k \down k'$ satisfies $j \leq k'$ for some $k' \in F'$, we have $\ell(\bigcup_{k' \in F'} k \down k') \leq \ell(F')$. It hence follows that $\bigvee_{k' \in F'} \bigvee_{j \in k \down k'} \, p_j \in \Zz \text{-} \lspan(\menge{p_l}{l \in \bar{F}})$ by induction hypothesis.
\eproof
\setlength{\parindent}{0cm} \setlength{\parskip}{0.5cm}

\blemma
\label{lem:C*L}
Let $L$ be a locally finite weak semilattice.
\setlength{\parindent}{0cm} \setlength{\parskip}{0cm}

\begin{enumerate}[label=(\roman*)]
    \item If $F \subseteq L$ is a finite $\down$-closed subset, then $\cs(\menge{p_f}{f \in F}) = \lspan((\menge{p_f}{f \in F})$ is finite dimensional.
    \item We have $\cs(L) \cong \ilim_F \cs(\menge{p_f}{f \in F})$, where the inductive limit is taken over all finite $\down$-closed subsets $F$ of $L$, ordered by inclusion.
    \item $\pi_r \colon \cs(L) \to \csr(L)$ is an isomorphism.
    \item The set $\menge{p_l}{l \in L\reg}$ is linearly independent (in $\cs(L)$).
    \item Let $D$ be a commutative $\cs$-algebra generated by projections $\menge{q_l}{l \in L}$ such that $q_0 = 0$ and $q_k q_l = \bigvee_{m \in k \down l} q_m$ for all $k, l \in L$. Then the canonical homomorphism $\pi \colon \cs(L) \to D$ sending $p_l$ to $q_l$ is an isomorphism if and only if the set $\menge{q_l}{l \in L\reg}$ is linearly independent (in $D$).
\end{enumerate}
\elemma
\setlength{\parindent}{0cm} \setlength{\parskip}{0cm}

\bproof
(i) follows from Lemma~\ref{lem:ZF}. (ii) is clear. For (iii), take a finite $\down$-closed subset $F \subseteq L$. For every $l \in F\reg$, consider the projection
\[
 P'_l \coloneq P_l - \bigvee_{k \in F, \, k \lneq l} P_k \in \cs(\menge{P_f}{f \in F\reg}).
\]
The set $\menge{P'_l}{l \in F\reg}$ is linearly independent because $P'_l(\delta_m) = 1$ if $l=m$ and $P'_l(\delta_m) = 0$ if $l \neq m$. Hence we have
\begin{align*}
 \# \, F\reg &\leq \dim( \cs(\menge{P_f}{f \in F\reg}) ) \\
 &\leq \dim( \cs(\menge{p_f}{f \in F\reg}) ) = \dim( \lspan( \menge{P_f}{f \in F\reg}) ) \leq \# \, F\reg.
\end{align*}
It follows that the restriction of $\pi_r$ to $\cs(\menge{p_f}{f \in F\reg})$ is isometric, and thus $\pi_r$ itself is isometric by (ii). This also shows that $\menge{P_l}{l \in L\reg}$ is linearly independent (in $\csr(L)$), which implies (iv) because of (iii). Finally, (v) follows by a similar argument: If $\menge{q_l}{l \in L\reg}$ is linearly independent (in $D$), then $\dim( \cs(\menge{q_f}{f \in F\reg}) = \# \, F\reg$ for every finite $\down$-closed subset $F$ of $L$, so that the restriction of $\pi$ to $\cs(\menge{p_f}{f \in F\reg})$ is isometric. Again (ii) implies that $\pi$ itself is isometric.
\eproof
\setlength{\parindent}{0cm} \setlength{\parskip}{0.5cm}

In the following, we develop a topological perspective. 

\bdefin
Let $L$ be a locally finite weak semilattice. A \emph{character} on $L$ is a map $\chi: \: L \to \gekl{0,1}$ with the property that $\chi(0) = 0$, $\chi$ is non-zero (i.e., there exists $l \in L$ such that $\chi(l) \neq 0$), and, for all $l_1, l_2 \in L$, we have that $\chi(l_1) = 1$ and $\chi(l_2) = 1$ holds if and only if there exists $l \in l_1 \down l_2$ with $\chi(l) = 1$. Let $\widehat{L}$ be the set of all characters on $L$. We equip $\widehat{L}$ with the topology of pointwise convergence. We introduce the notation $\widehat{L}(l) \coloneq \menge{\chi \in \widehat{L}}{\chi(l) = 1}$ for $l \in L\reg$. For each non-zero element $l \in L\reg$, the \emph{principal character at $l$} is the character $\chi_l$ defined by $\chi_l^{-1}(1) = \{ k \in L \mid k \geq l \}$.
\edefin

The following are immediate consequences of the definitions and the universal property of $\cs(L)$.
\blemma
We have a homeomorphism $\widehat{L} \isom \Spec(\cs(L))$ sending a character $\chi \in \widehat{L}$ to the element $\cs(L) \to \Cz, \, p_l \ma \chi(l)$ in $\Spec(\cs(L))$. The inverse sends the element $\omega \colon \cs(L) \to \Cz$ in $\Spec(\cs(L))$ to the character $\chi \colon L \to \gekl{0,1}, \, l \ma \omega(p_l)$.
\elemma
\bcor
\label{cor:C*L=C0whL}
We have an isomorphism $\cs(L) \cong C_0(\widehat{L})$ sending $p_l$ to the characteristic function $1_{\widehat{L}(l)}$.
\ecor

As a next step, let us explain how our setup can be viewed as a generalisation of the setup considered in \cite[\S~2]{CEL2}, with locally finite weak semilattices replacing semilattices. Consider $\cV \coloneq \{ \widehat{L}(l) \mid l \in L\reg \} \cup \gekl{\emptyset}$. Then $\cV$ is a set of compact open subsets of $\widehat{L}$ which is a poset with respect to $\subseteq$, and the map $L \to \cV$ sending $l \in L\reg$ to $\widehat{L}(l)$ and $0$ to $\emptyset$ is a poset isomorphism. In particular, it follows that $\cV$ is a locally finite weak semilattice. However, $\cV$ has the special property that $V \cap W = \bigcup_{U \in V \down W} U$ for all $V, W \in \cV$. Moreover, $\cV$ has two additional properties; it generates the set of all compact open subsets of $\widehat L$ under finite unions and relative complements, and it satisfies the independence condition: given $V, V_1, \dotsc, V_i \in \cV\reg$ with $V \subseteq \bigcup_{h=1}^i V_h$, then there exists $1 \leq h \leq i$ such that $V \subseteq V_h$. 

This motivates the following definition.

\bdefin\label{def:GenIndLFWS}
A \emph{representation} $U$ of a weak semilattice $L$ on a totally disconnected locally compact Hausdorff space $X$ is an order-preserving function $U \colon L \to \cO_c(X)$ with $U(0) = \emptyset$ and $U(l_1) \cap U(l_2) = \bigcup_{l \in l_1 \down l_2} U(l)$ for any $l_1,l_2 \in L$. The representation $U$ is \emph{generating} if 
\setlength{\parindent}{0cm} \setlength{\parskip}{0cm}

\begin{enumerate}
\item[(GEN)]\label{generating} $\{U(l) \mid l \in L\}$ generates $\cO_c(X)$ under finite unions and relative complements $V \setminus W$ for $V \subseteq W$,
\end{enumerate}
and \emph{independent} if 
\begin{enumerate}
\item[(IND)]\label{independent} whenever $l,l_1,\dots,l_n \in L$ satisfy $U(l) \subseteq \bigcup_{i=1}^n U(l_i)$, there is $1 \leq i \leq n$ such that $l = l_i$.
\end{enumerate}
Independence implies that the representation is faithful in the sense that $U$ is injective. A faithful representation given by a subset $L \subseteq \cO_c(X)$ will also be called a weak semilattice \emph{in} $X$.
\edefin

\blemma
\label{lem:X=whV}
Let $U \colon L \to \cO_c(X)$ be a generating representation of a locally finite weak semilattice $L$ on $X$. Then there is a homeomorphism 
\begin{align*}
 X \to & \widehat L_U :=  \menge{\chi \in \widehat{L}}{\substack{\text{for all } f, l_1, \dotsc, l_n \in L \text{ with } U(f) \subseteq \bigcup_i U(l_i),  \\ \chi(f) = 1 \text{ implies } \chi(l_i) = 1 \text{ for some } i}}, \\
 x \ma & \chi_x,
\end{align*}
from $X$ to the closed subset $\widehat L_U \subseteq \widehat L$, sending $x \in X$ to the character $\chi_x$ determined by $\chi_x(l) = 1$ if and only if $x \in U(l)$. The inverse map sends a character $\chi \in \widehat{L}_U$ to the unique point $x \in X$ such that 
\[
 \gekl{x} = \Big( \bigcap_{l \in L, \, \chi(l) = 1} U(l) \Big) \setminus \Big( \bigcup_{k \in L, \, \chi(k) = 0} U(k) \Big).
\]
This homeomorphism sends $U(l)$ to $\menge{\chi \in \widehat{L}_U}{\chi(l) = 1}$.
If $U$ is moreover independent, then $\widehat L_U = \widehat L$ and this determines a homeomorphism $X \cong \widehat L$.
\elemma
\setlength{\parindent}{0cm} \setlength{\parskip}{0cm}

\bproof
We first note that if $U$ is independent it is clear that $\widehat L_U = \widehat L$.
\setlength{\parindent}{0cm} \setlength{\parskip}{0.5cm}

We must check that given $\chi \in \widehat{L}_U$, the set $\Big( \bigcap_{l \in L, \, \chi(l) = 1} U(l) \Big) \setminus \Big( \bigcup_{k \in L, \, \chi(k) = 0} U(k) \Big)$ is a singleton. 

This is the intersection of the compact sets $U(l) \setminus U(k)$ for $k,l \in L$ with $\chi(k) = 0$ and $\chi(l) = 1$. We will argue this is nonempty by the finite intersection property, so consider $(k_1,l_1),\dots,(k_n,l_n)$ as above. Then there exists $l \in \bigdown_{i=1}^n l_i$ with $\chi(l) = 1$, and so $\bigcap_{i=1}^n U(l_i) \setminus \bigcup_{i=1}^n U(k_i) \supseteq U(l) \setminus \bigcup_{i=1}^n U(k_i)$ is nonempty because $\chi \in \widehat L_U$. 

Suppose now that it contains distinct elements $x \neq y$. Because $U$ is generating, there is $l \in L$ such that $U(l)$ either contains $x$ but not $y$ or $y$ but not $x$. Without loss of generality we have $x \in U(l)$, which implies that $\chi(l) = 1$. But this in turn implies $y \in U(l)$, a contradiction.

It is clear that this determines an inverse to the continuous map $x \mapsto \chi_x$. 
If $x \in U(l)$, then $\chi_x(l) = 1$ and conversely if $\chi = \chi_x$ satisfies $\chi(l) = 1$, then $x \in U(l)$. Thus the compact open set $U(l)$ is sent to the compact open set $\{\chi \in \widehat L_U \mid \chi(l) = 1\}$. Because $U$ is generating and $X$ is totally disconnected, we conclude that $x \mapsto \chi_x$ is open and thus a homeomorphism.
\eproof
\setlength{\parindent}{0cm} \setlength{\parskip}{0.5cm}

\bremark
When $L$ is a semilattice, an independent generating representation of $L$ on $X$ is precisely the data of a regular basis in the sense of \cite[\S~2]{CEL2}. A general locally finite weak semilattice $L$ with an independent generating representation $U$ on $X$ is a semilattice if and only if $\{ U(l) \mid l \in L\reg\}$ is a regular basis. Moreover, \cite[Proposition~2.12]{CEL2} tells us that if $X$ is a totally disconnected, locally compact Hausdorff space which is second countable, then there always exists a regular basis, and hence an independent generating representation of a semilattice on $X$.
\eremark

Later on, we will need the following definition.
\bdefin
Given a locally finite weak semilattice $L$, let $\Zz L$ be the free abelian group generated by $L\reg$, i.e., $\Zz L = \bigoplus_{l \in L\reg} \Zz$. We write elements of $\Zz L$ as finite sums of the form $\sum_{l \in L\reg} a_l \, l$, for $a_l \in \Zz$.
\edefin

The following is an immediate consequence of Lemma~\ref{lem:C*L}.
\bcor
\label{cor:ZL=span}
We have an isomorphism of abelian groups
\[
 \Zz L \isom \Zz \text{-} \lspan(\menge{p_l}{l \in L}), \, l \ma p_l.
\]
\ecor

Moreover, Lemma~\ref{lem:ZF} implies that $\Zz \text{-} \lspan(\menge{p_l}{l \in L})$ is a $\Zz$-algebra, where the multiplication is induced by the one in $\cs(L)$. Therefore, this enables us to equip $\Zz L$ with a $\Zz$-algebra structure such that the isomorphism provided by Corollary~\ref{cor:ZL=span} becomes an isomorphism of $\Zz$-algebras. In particular, we can introduce the following.
\bdefin
\label{def:join_ZL}
Given a finite subset $F \subseteq L$, define $\bigvee_{f \in F} \, f \in \Zz L$ as the element of $\Zz L$ that is mapped to $\bigvee_{f \in F} \, p_f \in \Zz \text{-} \lspan(\menge{p_l}{l \in L})$ under the isomorphism from Corollary~\ref{cor:ZL=span}.
\edefin

Our topological perspective gives another picture for $\Zz L$. In the following, $C_c(\widehat L, \Zz)$ denotes the $\Zz$-algebra of compactly supported, continuous functions $\widehat{L} \to \Zz$. Since $\Zz$ is viewed as a discrete space, a function $\widehat{L} \to \Zz$ is continuous if and only if it is locally constant.
\blemma\label{lemma on identifying Z algebras}
There is an isomorphism of $\Zz$-algebras
\[ 
 \Zz L \isom C_c(\widehat L, \Zz), l \ma 1_{\widehat{L}(l)}.
\]
\elemma
\setlength{\parindent}{0cm} \setlength{\parskip}{0cm}

\bproof
Our map is given by the composition
\[
 \Zz L \isom \Zz \text{-} \lspan(\menge{p_l}{l \in L}) \into \cs(L) \isom C_0(\widehat{L}),
\]
where the first arrow is provided by Corollary~\ref{cor:ZL=span}, the second arrow is the canonical inclusion, and the third arrow is the isomorphism from Corollary~\ref{cor:C*L=C0whL}. It hence suffices to show that the image of our map is all of $C_c(\widehat L, \Zz)$. This follows from the observation that $C_c(\widehat L, \Zz)$ is generated as an abelian group by characteristic functions $1_V$ of basic compact open sets of the form $V = \widehat{L}(l) \setminus (\bigcup_{k \in F} \widehat{L}(k))$, where $F$ is a finite subset of $\menge{k \in L}{k \leq l}$, and such elements $1_V$ are in the image of our map $\Zz L \to C_c(\widehat L, \Zz)$.
\eproof

\subsection{The equivariant setting}\label{ss:eqLFWS}

\bdefin\label{d:S acts on L}
Let $L$ be a poset with $0$. A partial order automorphism of $L$ is an order isomorphism between down-sets of $L$.
An action of an inverse semigroup $S$ on a poset $L$ with $0$ is an action on the set $L$ by partial order automorphisms.
\edefin

An action $S \acts L$ on a locally finite weak semilattice $L$ induces an action $S \acts \widehat{L}$. The dual action is given by the partial homeomorphisms $\dom_{\widehat L}(s) \to \ran_{\widehat L}(s), \, \chi \ma s.\chi$, where 
\[
 \dom_{\widehat L}(s) \coloneq \menge{\chi \in \widehat{L}}{\chi(l) = 1 \text{ for some } l \in \dom_L(s)},
\]
$\ran_{\widehat L}(s)$ is defined analogously, and $s.\chi(l) = 1$ if and only if there exists $k \in \ran_L(s)$ with $k \leq l$ and $\chi(s^{-1}(k)) = 1$.

\bremark
The most canonical action of an inverse semigroup $S$ on a locally finite weak semilattice is the action $S \acts E$ on the semilattice $E$ of idempotents in $S$. The action is given at $s \in S$ by $\dom_E(s) = \menge{e \in E}{e \leq s^{-1}s}$ and $s.e \coloneq s e s^{-1}$. In this case, $S \ltimes \widehat{E}$ is Paterson's universal groupoid of $S$ (see \cite[\S~4.3]{Pat}, but with a small difference arising from the treatment of $0 \in S$ as pointed out in \cite[\S~5.5.4]{CELY}).
\eremark

Now let us develop a topological perspective in the equivariant setting, following \cite[Definition~2.12]{Li22}. 
\bdefin
\label{def:InvGILFWS}
Let $S$ be an inverse semigroup, and let $S$ act on a totally disconnected, locally compact Hausdorff space $X$ and a locally finite weak semilattice $L$. A representation $U \colon L \to \cO_c(X)$ is \emph{$S$-equivariant} if 
\setlength{\parindent}{0cm} \setlength{\parskip}{0cm}

\begin{itemize}
    \item for each $s \in S$ we have $\dom_X (s) = \bigcup_{l \in \dom_L (s)} U(l)$, and
    \item for each $s \in S$ and $l \in \dom_L (s)$, we have $s.(U(l)) = U(s.l)$.
\end{itemize}
\edefin
\setlength{\parindent}{0cm} \setlength{\parskip}{0cm}

Note that, if $U$ is $S$-equivariant, then the following are automatically satisfied:
\begin{itemize}
    \item $\ran_X (s) = \bigcup_{l \in \ran_L (s)} U(l)$ for every $s \in S$,
    \item $U|_{\dom_L (s)} \colon \dom_L (s) \to \cO_c(\dom_X (s))$ is a representation of the locally finite weak semilattice $\dom_L (s)$ on $\dom_X (s)$, and similarly for $U|_{\ran_L (s)}$.
\end{itemize}
\setlength{\parindent}{0cm} \setlength{\parskip}{0.5cm}

The following is now a straightforward consequence.
\blemma\label{lem:eqX=whV}
Let $S$ be an inverse semigroup, let $S$ act on a totally disconnected, locally compact Hausdorff space $X$ and a locally finite weak semilattice $L$, and suppose that $U$ is an $S$-equivariant generating representation of $L$ on $X$. Then we have an isomorphism of $S$-dynamical systems $(S \acts X) \cong (S \acts \widehat{L}_U)$. More precisely, the homeomorphism $X \cong \widehat{L}_U$ from Lemma~\ref{lem:X=whV} is $S$-equivariant. If $U$ is moreover independent, this yields an isomorphism $(S \acts X) \cong (S \acts \widehat{L})$.
\elemma

\bdefin\label{d:independent}
Let $S \acts X$ be an action of an inverse semigroup $S$ on a totally disconnected locally compact Hausdorff space $X$. We say $S \acts X$ is \emph{independent} if there is an $S$-equivariant independent generating representation of a locally finite weak semilattice $L$ on $X$. An ample groupoid $G$ is called \emph{independent} if there is an independent action $S \acts X$ with $G \cong S \ltimes X$.
\edefin

\bex\label{ex:HRG2}
Let $\Lambda$ be a higher rank graph as in Example~\ref{ex:HRG1}. The \emph{graph inverse semigroup} $S_\Lambda$ is the left inverse hull of $\Lambda$, i.e., the smallest inverse semigroup of partial bijections of $\Lambda$ generated by the partial bijections $\mfd(\lambda) \Lambda \isom \lambda \Lambda, \, \mu \ma \lambda \mu$ (for $\lambda \in \Lambda$). The action of $S_{\Lambda}$ on $\Lambda$ extends to an action on the poset $L_{\Lambda} = \Lambda \cup \{0\}$ by setting $s. 0 = 0$ for each $s \in S_\Lambda$.

A \emph{path} of (possibly infinite) shape $m \in (\mathbb Z_{\geq 0} \cup \{\infty\})^k$ in $\Lambda$ is a function $z \colon \{ p \in \mathbb Z_{\geq 0}^k \mid p \leq m \} \to \Lambda$ such that $z(0) \in \Lambda^0$ and $z(q) \in z(p) \Lambda$ whenever $p \leq q$. We write $d(z) := m$. An element $\lambda$ of $\Lambda$ is considered as a finite path of shape $d(\lambda)$ via the function $z_\lambda$ which sends $p \leq d(\lambda)$ to the unique $\mu \in \Lambda$ with $d(\mu) = p$ and $\lambda \in \mu \Lambda$. We say a path $z \in X_\Lambda$ \emph{extends} $\lambda \in \Lambda$ if $z(d(\lambda)) = \lambda$. Let $X_\Lambda$ be the set of paths (of any shape) in $\Lambda$; this is straightforward to identify with the space $X_\Lambda$ constructed in \cite[\S~5]{FMY}. Moreover, under the topology considered in \cite{FMY}, $\cO_c(X_\Lambda)$ is generated by the compact open sets $V(\lambda) := \{ z \in X_\Lambda \mid z(d(\lambda)) = \lambda \}$ for $\lambda \in \Lambda$ under finite intersections, finite unions and relative complements. The inverse semigroup $S_\Lambda$ acts on $X_\Lambda$ with domains $\dom_{X_\Lambda}(s) := \{ z \in X_\Lambda \mid \text{there is } \lambda \in \dom_{\Lambda} (s) \text{ with } z(d(\lambda)) = \lambda \}$ as follows. For $z \in \dom_{X_\Lambda}(s)$, the path $s . z \in X_\Lambda$ has shape $d(s.z) := \sup \{ d(s.\lambda) \in \mathbb Z_{\geq 0}^k \mid \lambda \in \dom_\Lambda (s), \, z(d(\lambda)) = \lambda \}$ and for $p \in \mathbb Z_{\geq 0}^k$ with $p \leq d(s.z)$, the element $(s.z)(p)$ is the unique $\mu \in \Lambda$ with $d(\mu) = p$ such that there exists $\lambda \in \dom_\Lambda(s)$ with $z(d(\lambda)) = \lambda$ and $s . \lambda \in \mu \Lambda$. The groupoid $S_\Lambda \ltimes X_\Lambda$ agrees with the groupoid $G_\Lambda$ from \cite{FMY} (although our $S_{\Lambda}$ is an inverse subsemigroup of what is denoted by $S_{\Lambda}$ in \cite[\S~4]{FMY}).

Suppose that $\Lambda$ is finitely aligned so that $L_\Lambda$ is a locally finite weak semilattice. Then $V \colon L_\Lambda \to \cO_c(X_\Lambda)$ determined by $V(\lambda)$ for $\lambda \in \Lambda$ is an $S_\Lambda$-equivariant generating representation of $L_\Lambda$. By consideration of the finite paths $z_\lambda \in X_\Lambda$ associated to $\lambda \in \Lambda$, this representation is moreover independent, and indeed one can identify $X_\Lambda$ with the character space $\widehat L_\Lambda$.

Consider also the boundary path space $\partial \Lambda \subseteq X_\Lambda$ considered in \cite{FMY}, which in the finitely aligned setting consists of all ``non-extendible'' paths, i.e. $z \in X_\Lambda$ such that if $z' \in X_\Lambda$ satisfies $d(z) \leq d(z')$ and $z(p) = z'(p)$ for all $p \in \mathbb Z^k_{\geq 0}$ with $p \leq d(z)$, then $z=z'$. The space $\partial \Lambda$ is closed, $S_\Lambda$-invariant and inherits an $S_\Lambda$-equivariant generating representation $U \colon L_\Lambda \to \cO_c(\partial \Lambda)$ of $L_\Lambda$ given by $U(\lambda) = V(\lambda) \cap \partial \Lambda$. We note that $U(\lambda)$ is nonempty for each $\lambda \in \Lambda$, as any finite path may be extended to some non-extendible path. If we assume that $\Lambda$ is row-finite and source-free, i.e. $0 <  \# v\Lambda^{(i)} < \infty$ for each $v \in \Lambda^0$ and $1 \leq i \leq k$, then $\partial \Lambda = \Lambda^\infty$ is the space of paths with shape $(\infty,\infty,\dots,\infty)$.
\eex

\section{Independent resolutions}
\label{s:IndRep}

\subsection{Definitions and general observations}
\label{ss:IndResDef}

The idea of independent resolutions is to find \an{resolutions} of ample groupoids by independent groupoids, i.e. those of the form $S \ltimes \widehat{L}$ as in \S~\ref{ss:eqLFWS}. 

\bdefin\label{d:resolution}
Let $G$ be an ample groupoid. An \emph{independent resolution} of $G$ is given by a sequence of ample groupoids $G_i$, $i = -1, 0, 1, 2, \dotsc$, with the following extra data and subject to the following conditions:
\setlength{\parindent}{0cm} \setlength{\parskip}{0cm}

\begin{itemize}
\item $G_{-1} = G$,
\item For every $i$, a decomposition $G_i^{0} = C_i \amalg O_i$, where $C_i$ is a closed $G_i$-invariant subspace (and hence $O_i$ is open and invariant), and $C_{-1} = G^0$,
\item For every $i \geq -1$, a groupoid isomorphism $\psi_i  \colon G_i \vert_{O_i} \isom G_{i+1} \vert_{C_{i+1}}$, where $G_i \vert_{O_i} \coloneq \menge{g \in G_i}{r(g), s(g) \in O_i}$ is the restriction of $G_i$ to $O_i$, and similarly for $G_{i+1} \vert_{C_{i+1}}$,
\item For each $i \geq 0$, an action $S_i \acts L_i$ of an inverse semigroup $S_i$ on a locally finite weak semilattice $L_i$ with an isomorphism $G_i \cong S_i \ltimes \widehat{L_i}$. 
\end{itemize}
Note that $\psi_i$ restricts to a homeomorphism $O_i \cong C_{i+1}$. 
\setlength{\parindent}{0cm} \setlength{\parskip}{0.5cm}

We call $(G_i)_i$ an independent resolution in the strong sense if for each $i$ we have $L_i = E(S_i)$ (the semilattice of idempotents in $S_i$) with the canonical action $S_i \acts E(S_i)$.
\edefin
\setlength{\parindent}{0cm} \setlength{\parskip}{0.5cm}

Let us now explain why independent resolutions always exist (even in the strong sense). First of all, given an ample groupoid $G$, we can always find an inverse semigroup $S$ of compact open bisections in $G$ such that $G \cong S \ltimes X$ (see \cite[Proposition~2.2]{KM}). Furthermore, we can choose a family $\cU$ of compact open subsets of $X$ such that $\cU$ generates the topology of $X$ in the sense of \cite[\S~2]{CEL2} (see also \cite{CEL1}), i.e. $\cU$ generates the entire family of compact open subsets through finite unions, finite intersections and relative complements. Let $E$ be the smallest semilattice of compact open subsets of $X$ containing $\cU$ which is $S$-invariant. By construction, $E$ is an $S$-invariant generating semilattice in $X$. Thus by Lemma \ref{lem:eqX=whV} we have the following homeomorphism, identifying $X$ with a closed $S$-invariant subspace of $\widehat{E}$:
\begin{align*}
 X \to & \menge{\chi \in \widehat{E}}{\substack{\text{for all } f, e_1, \dotsc, e_n \in E \text{ with } f = \bigcup_i e_i \text{ in } X,  \\ \chi(f) = 1 \text{ implies } \chi(e_i) = 1 \text{ for some } i}}, \\
 x \ma & \chi_x,
\end{align*}
where $\chi_x(e) = 1$ if and only if $x \in e$. Hence, we can set $G_0 \coloneq S \ltimes \widehat{E}$, define $C_0$ as the subspace of $\widehat{E}$ homeomorphic to $X$ under the map above, and $O_0 \coloneq \widehat{E} \setminus C_0$. Now proceed inductively, with $G_0 \vert_{O_0}$ in place of $G$. This produces an independent resolution (even in the strong sense).

\subsection{Systems of finite covers}
\label{ss:IndResCovers}

The above-mentioned existence statement is quite abstract. Let us now present a more systematic way of producing independent resolutions. Suppose $G \cong S \ltimes X$ as before and suppose $U$ is an $S$-equivariant, generating representation of a locally finite weak semilattice $L$ in $X$. As before, we may by Lemma \ref{lem:eqX=whV} identify $X$ with the closed $S$-invariant subspace 
\[ C_{-1} = \left\{ \chi \in \widehat L \suchthat  \substack{\text{for all } l, f_1, \dots, f_n \in L \text{ with } U(l) = \bigcup_i U(f_i), \\ \chi(l) = 1 \text{ implies } \exists i \text{ with } \chi(f_i) = 1 } \right\} \]
of $\widehat L$. We could as before have arranged $L$ to be a semilattice, but the extra flexibility allows us to more easily handle situations naturally endowed with a locally finite weak semilattice.

The inclusion $X \subseteq \widehat{L}$ induces a map $C_c(\widehat{L},\Zz) \onto C_c(X,\Zz)$. The kernel $I$ of this map is generated by the indicator functions on relative complements of the form 
\begin{equation}\label{kernel description one}
 \widehat L(l) \setminus \bigcup_{f \in F} \widehat L(f),
\end{equation}
where $l \in L$ and $F \subseteq L$ is a finite set with $l = \bigcup_{f \in F} f$ in $X$. This satisfies the property that if $k \in L$ is orthogonal to each $f \in F$, it is also orthogonal to $l$. We say that $F$ is a \textit{finite cover} of $l$. This motivates the following general definition.
\bdefin
Let $L$ be a locally finite weak semilattice and let $l \in L$. A finite set $F \subseteq L$ is a \textit{finite cover} of $l$ if
\setlength{\parindent}{0cm} \setlength{\parskip}{0cm}

\begin{itemize}
\item $f \leq l$ for each $f \in F$, and
\item if $k \in L$ with $k \perp f$ for each $f \in F$, then $k \perp l$.
\end{itemize}
A finite cover $F$ of $l$ is \emph{trivial} if $l \in F$. {The set of maximal elements in a finite cover is again a finite cover; it is natural to only consider finite covers whose elements are all maximal.} 
\setlength{\parindent}{0cm} \setlength{\parskip}{0.5cm}

Given a representation $U \colon L \to \cO_c(X)$ of $L$ on a totally disconnected locally compact Hausdorff space $X$, we say a finite cover $F$ of $l \in L$ is \emph{concrete} if $U(l) = \bigcup_{f \in F} U(f)$.
\edefin
\setlength{\parindent}{0cm} \setlength{\parskip}{0.5cm}

\bex\label{ex:HRG3}
Let $\Lambda$ be a finitely aligned higher rank graph and let $L_\Lambda = \Lambda \cup \{0\}$ be its locally finite weak semilattice (see Examples \ref{ex:HRG1} and \ref{ex:HRG2}). A finite cover $E$ of a vertex $v \in \Lambda^0$ is precisely a finite exhaustive subset in the sense of \cite[Definition 3.6]{FMY}. Suppose $v \in \Lambda^0$ is a vertex and $1 \leq i \leq k$ is an index such that $0 < \# \, v\Lambda^{(i)} < \infty$. Then for each $\lambda \in v\Lambda$ an explicit finite cover of $\lambda$ is given by \[ \lambda \Lambda^{(i)}.\]
This is concrete with respect to the representation $U \colon L_\Lambda \to \cO_c(\partial \Lambda)$ given by $U(\lambda) = \{ z \in \partial \Lambda \mid z(d(\lambda)) = \lambda \}$ because any infinite path $z \in \partial \Lambda$ starting with $\lambda$ must start with $\lambda \epsilon$ for some $\epsilon \in \Lambda^{(i)}$.
\eex

We introduce the following conventions. 
\bdefin
Let $L$ be a locally finite weak semilattice and let $F$ be a finite cover of an element $l \in L$. We write
\[ (l \mid F) \coloneq l - \bigvee_{f \in F} f \in \Zz L \]
for the relative complement of $F$ in $l$ (the join $\bigvee$ was introduced in Definition~\ref{def:join_ZL}). Given further an element $k \leq l$, the finite cover
\begin{equation}
    k \down F := {\max} \big( \bigcup_{f \in F} k \down f \big)
\end{equation}
is given by the set {of maximal elements in} $\bigcup_{f \in F} k \down f$, which is itself a finite cover of $k$. For a finite set $\cF$ of finite covers of $l$, we set 
\begin{equation}
(l \mid \cF) := l - \bigvee_{f \in F \in \cF} f = \prod_{F \in \cF} (l \mid F)
\end{equation}
We write $k \down \cF$ for the finite set $\{ k \down F \mid F \in \cF \}$ of finite covers of $k$. 
\edefin
The following computation is straightforward.
\blemma\label{l:cover mult}
Let $L$ be a locally finite weak semilattice, and let $\cF$ be a finite set of finite covers of $l \in L$. Then for each $h \in L$ we have $h(l \mid \cF) = \bigvee_{k \in h \down l} (k \mid k \down \cF) \in \Zz L$.
\elemma

Let us come back to the motivating case where $G \cong S \ltimes X$ and $U \colon L \to \cO_c(X)$ is an $S$-equivariant generating representation of a locally finite weak semilattice $L$ on $X$. Identifying $C_c(\widehat L, \Zz)$ with $\Zz L$, our description of the kernel $I$ of the surjection $C_c(\widehat L, \Zz) \onto C_c(X, \Zz)$ using \eqref{kernel description one} is now equivalent to the statement that $I$ is generated by the relative complements $(l \mid F)$ of certain finite covers in $L$. Consider, for each $l \in L$,
\[
 \R(l) = \menge{F \subseteq L}{F \text{ is a finite cover of } l, \, (l \mid F) \in I}.
\]
Because $I$ is an ideal in $\Zz L$, the system $\R = (\R(l))_{l \in L}$ of finite covers is closed under $\down$ in the following sense:
\setlength{\parindent}{0cm} \setlength{\parskip}{0cm}

\begin{enumerate}
\item[(M)] For all $l \in L$, $F \in \R(l)$ and $k \in L$ with $k \leq l$, $k \down F \in \R(k)$.
\end{enumerate}
\setlength{\parindent}{0cm} \setlength{\parskip}{0.5cm}

Moreover, $\R$ is $S$-invariant in the sense that for all $l \in L$, $s \in S$ with $l \in \dom_L(s)$, and for all $F \in \R(l)$, the finite cover \[s.F \coloneq \menge{s.f}{f \in F}\] is an element of $\R(s.l)$. This $S$-invariance holds because $I$ is $S$-invariant. Let us now turn this into an abstract definition.
\bdefin\label{d:system}
Let $L$ be a locally finite weak semilattice. A \emph{system of finite covers in $L$} is a tuple 
\[ \R = (\R(l))_{l \in L}\]
where each $\R(l)$ is a set of finite covers of $l \in L$, such that
\setlength{\parindent}{0cm} \setlength{\parskip}{0cm}

\begin{itemize}
\item for each $l \in L$, the trivial cover $\{l\}$ is an element of $\R(l)$, and
\item $\R$ satisfies condition (M): if for all $l \in L$, $F \in \R(l)$ and $k \in L$ with $k \leq l$, we have that $k \down F \in \R(k)$.
\end{itemize}
Given an action $S \acts L$ of an inverse semigroup $S$, we say that $\R$ is $S$-invariant if for all $l \in L$, $s \in S$ with $l \in \dom(s)$, and for all $F \in \R(l)$, we have that $s.F \coloneq \menge{s.f}{f \in F}$ lies in $\R(s.l)$. 
\edefin
\setlength{\parindent}{0cm} \setlength{\parskip}{0cm}

The set $\R(l)$ of all finite covers $F$ of $l$ for which $(l \mid F) \in I$, i.e. $l = \bigcup_{f \in F} f$ in $X$, may be impractically large. In concrete applications there is often a natural system of finite covers which is large enough to generate $I$ but small enough to be manageable. 
\setlength{\parindent}{0cm} \setlength{\parskip}{0.5cm}

\bdefin\label{d:thorough}
Let $U \colon L \to \cO_c(X)$ be a generating representation of a locally finite weak semilattice $L$ on a totally disconnected locally compact Hausdorff space $X$. Let $\R$ be a system of concrete finite covers in $L$. We say that $\R$ is \emph{thorough} if whenever $l \in L$ satisfies $U(l) \subseteq \bigcup_{i=1}^n U(l_i)$, with $l_i \in L$, there exist a positive integer $N$ and covers
\setlength{\parindent}{0cm} \setlength{\parskip}{0cm}

\begin{itemize}
    \item $F(l) \in \R(l)$,
    \item $F(l, f_1) \in \R(f_1)$ for all $f_1 \in F(l)$,
    \item $F(l, f_1, f_2) \in \R(f_2)$ for all $f_1 \in F(l)$ and $f_2 \in F(l,f_1)$,
    \item $\dotso$,
    \item $F(l, f_1, f_2, \dotsc, f_{N-1}) \in \R(f_{N-1})$    
    for all $f_1 \in F(l), f_2 \in F(l,f_1), \dots,$ and $f_{N-1} \in F(l, f_1, f_2, \dotsc, f_{N-2})$,
\end{itemize}
such that, for all $f_1 \in F(l)$, $f_2 \in F(l,f_1), \dots, f_{N-1} \in F(l, f_1, f_2, \dotsc, f_{N-1})$ and $f_N \in F(l, f_1, f_2, \dotsc, f_{N-1})$, there exists an index $i$ such that $f_N \leq l_i$.
\edefin
\setlength{\parindent}{0cm} \setlength{\parskip}{0.5cm}

\bex
\label{ex:HRG4}
Let $\Lambda$ be a higher rank graph of rank $k$ as in Examples~\ref{ex:HRG1}, \ref{ex:HRG2} and \ref{ex:HRG3}, and let us use the same notation as before. As in Example~\ref{ex:HRG3}, we assume that for every vertex $v$ in $\Lambda$ and $1 \leq i \leq k$, we have $0 < \# v \Lambda^{(i)} < \infty$. For $\lambda \in \Lambda = {L_\Lambda}\reg$ set 
\[ \R(\lambda) = \{\{\lambda\}\} \cup \{ \lambda \Lambda^{(i)} \mid 1 \leq i \leq k \} \]
and set $\R(0) = \{\{0\}\}$. 

Let us check the family $\R$ of finite covers satisfies condition (M). Let $\mu, \lambda \in \Lambda$ with $\mu \leq \lambda$ (i.e. $\mu \in \lambda \Lambda$) and let $1 \leq i \leq k$. If in the $i$th component we have $d(\mu)_i > d(\lambda)_i$, then the unique factorisation property implies that $\mu \down \lambda \Lambda^{(i)} = \{\mu\} \in \R(\mu)$. If instead $d(\mu)_i = d(\lambda)_i$, unique factorisation forces $\mu \down \lambda \Lambda^{(i)} = \mu \Lambda^{(i)} \in \R(\mu)$.

Moreover, the system $\R$ is thorough with respect to the generating representation $U \colon L_\Lambda \to \cO_c(\partial \Lambda)$ (see Examples~\ref{ex:HRG2} and \ref{ex:HRG3}). Indeed, suppose $\lambda, \lambda_1,\dots,\lambda_n \in \Lambda$ satisfy $U(\lambda) \subseteq \bigcup_{j=1}^n U(\lambda_j)$. The intuition for the following is that to deduce $f_N \leq \lambda_i$ for some $i$, we need only force $d(f_N) \geq d(\lambda_i)$ for each $i$ and apply the unique factorisation property. Pick $d = (d_1,\dots,d_k) \in \mathbb N^k$ with $d(\lambda) + d \geq d(\lambda_i)$ for each $1 \leq i \leq n$, and set $N = \sum_{i=1}^k d_i$. Pick $i_0,\dots,i_{N-1} \in \{1,\dots,k\}$ such that exactly $d_i$ are equal to $i$ for each $1 \leq i \leq k$. Finally, set $F(\lambda) = \lambda \Lambda^{(i_0)}$ and, by recursion on $1 \leq m \leq N-1$, given $f_1 \in F(\lambda), f_2 \in F(\lambda,f_1), \dots, f_m \in F(\lambda,f_1,\dots,f_{m-1})$, set $F(\lambda,f_1,\dots,f_m) = f_m \Lambda^{(i_m)}$. By construction, at $m = N-1$, each $f_N \in F(\lambda,f_1,\dots,f_{N-1})$ satisfies $d(f_N) = d(\lambda) + d$, and $U(f_N) \subseteq U(\lambda) \subseteq \bigcup_{j=1}^n U(\lambda_i)$. Pick $z \in U(f_N)$, which lies in $U(\lambda_i)$ for some $i$. Then $z = f_N z'$ and $z = \lambda_i z''$ for some $z'',z'' \in \partial \Lambda$, and since $d(\lambda_i) \leq d(f_N)$ it follows by the unique factorisation property that $f_N \in \lambda_i \Lambda$, i.e. $f_N \leq \lambda_i$.

Note that $\bigvee_{\varepsilon \in \lambda \Lambda^{(i)}} \lambda \varepsilon = \sum_{\varepsilon \in \lambda \Lambda^{(i)}} \lambda \varepsilon$ because elements in $\lambda \Lambda^{(i)}$ are pairwise orthogonal.
\eex

\blemma\label{l:thorough}
Let $U \colon L \to \cO_c(X)$ be a generating representation of a locally finite weak semilattice $L$ on a totally disconnected locally compact Hausdorff space $X$. 
Then the following are equivalent:
\setlength{\parindent}{0cm} \setlength{\parskip}{0cm}

\begin{enumerate}
\item\label{eq:thorough1} $\R$ is thorough.
\item\label{eq:thorough2} The complement $\widehat L \setminus X$ of the inclusion $X \subseteq \widehat L$ from Lemma \ref{lem:X=whV} is covered by relative complements of the form $\widehat L(l) \setminus \bigcup_{f \in F} \widehat L(f)$, where $l \in L$ and $F \in \R(l)$.
\item\label{eq:thorough3} The kernel $I$ of the map $C_c(\widehat{L},\Zz) \onto C_c(X,\Zz)$ induced by the canonical inclusion $X \subseteq \widehat{L}$ is the algebra generated by the indicator functions on relative complements of the form $\widehat L(l) \setminus \bigcup_{f \in F} \widehat L(f)$, where $l \in L$ and $F \in \R(l)$.
\item\label{eq:thorough4} The kernel of the map $\Zz L \to C_c(X,\Zz)$ which sends $l$ to $1_{U(l)}$ is the algebra generated by 
\[
 \menge{(l \mid F)}{l \in L, \, F \in \R(l)}.
\]
\end{enumerate}
\elemma
\setlength{\parindent}{0cm} \setlength{\parskip}{0cm}

\bproof
The equivalence of \eqref{eq:thorough3} and \eqref{eq:thorough4} is immediate via the isomorphism $\Zz L \isom C_c(\widehat{L},\Zz)$ from Lemma~\ref{lemma on identifying Z algebras}. The subalgebra of $\Zz L$ generated by $\menge{(l \mid F)}{l \in L, \, F \in \R(l)}$ is already an ideal by Lemma~\ref{l:cover mult}. This implies the equivalence of \eqref{eq:thorough2} with \eqref{eq:thorough3} and \eqref{eq:thorough4}.  
\setlength{\parindent}{0cm} \setlength{\parskip}{0.5cm}

Suppose that $\R$ is thorough. 
The set $\widehat L \setminus X$ is covered by sets of the form $\widehat L(l) \setminus \bigcup_{i=1}^n \widehat L(l_i)$ for elements $l, l_1,\dots,l_n \in L$ with $U(l) \subseteq \bigcup_{i = 1}^n U(l_i)$ by Lemma \ref{lem:X=whV}. Fix $l, l_1,\dots,l_n \in L$ with $U(l) \subseteq \bigcup_{i = 1}^n U(l_i)$. Applying thoroughness, we find $N \geq 1$ and a network of finite covers 
$F(l, f_1, f_2, \dotsc, f_{j-1}) \in \R(f_{j-1})$    
    for all $f_1 \in F(l), f_2 \in F(l,f_1), \dots,$ and $f_{j-1} \in F(l, f_1, f_2, \dotsc, f_{j-2})$, for each $1 \leq j \leq N$.
    For $1 \leq j \leq N$ recursively define 
    \[\mathscr F^{(j)} = \{(f_1,f_2,\dots,f_j) \mid (f_1,\dots,f_{j-1}) \in \mathscr F^{(j-1)}, f_j \in F(l,f_1,\dots,f_{j-1}) \}. \] 
    The definition of thoroughness tells us that for each $(f_1,\dots,f_N) \in \mathscr F^{(N)}$, there exists an index $i$ such that $f_N \leq l_i$, and therefore the following inclusion holds.
\begin{equation}\label{eq:thorough}
\widehat L(l) \mathbin{\Big\backslash} \bigcup_{i=1}^n \widehat L(l_i)  \subseteq \widehat L(l) \mathbin{\Big\backslash} \bigcup_{(f_1,\dots,f_N) \in \mathscr F^{(N)}} \widehat L(f_N)
\end{equation}
By construction we have the following inclusion
\begin{equation}\label{eq:telescope}
\widehat L(l) \mathbin{\Big\backslash} \bigcup_{(f_1,\dots,f_N) \in \mathscr F^{(N)}} \widehat L(f_N) \subseteq \bigcup_{\substack{0 \leq j \leq N-1, \\ (f_1,\dots,f_j) \in \mathscr F^{(j)}}} \left( \widehat L(f_j) \mathbin{\Big\backslash} \bigcup_{f_{j+1} \in F(l,f_1,\dots,f_j)} \widehat L(f_{j+1})  \right), 
\end{equation}
which demonstrates that \eqref{eq:thorough1} implies \eqref{eq:thorough2}.

Now suppose that $\R$ is not thorough. Then there are elements $l,l_1,\dots,l_n \in L$ with $U(l) \subseteq \bigcup_{i=1}^n U(l_i)$, such that for each $N \geq 1$ and each $F_1 \in \R(l)$, there exists $f_1 \in F_1$ such that for each $F_2 \in \R(f_1)$, there exists $f_2 \in F_2$, $\dots$, such that for each $F_{N} \in \R(f_{N-1})$, there is $f_N \in F_{N}$ such that for all $1 \leq i \leq n$, we have $f_N \not \leq l_i$.

Let $\mathscr S := \{ (k,R) \mid k \in L, \, R \in \R(k) \}$. We claim that there is a character $\chi \in \widehat L(l) \setminus \bigcup_{i=1}^n \widehat L(l_i)$ such that $\chi \notin \widehat L(k) \setminus \bigcup_{f \in R} \widehat L(f)$ for any $(k,R) \in \mathscr S$. As $U(l) \subseteq \bigcup_{i=1}^n U(l_i)$, this means that $\chi \in \widehat L \setminus X$, hence this is enough to prove that \eqref{eq:thorough2} implies \eqref{eq:thorough1}.

Let $A \subseteq \mathscr S$ be a finite subset of $\mathscr S$ of size $N \geq 1$. Set $f_0 := l$ and $A_0 := A$.
\setlength{\parindent}{0cm} \setlength{\parskip}{0cm}

\begin{itemize}
\item If there is $(k_1,R_1) \in A$ with $f_0 \leq k_1$, set $F_1 := f_0 \down R_1 \in \R(l)$ and $A_1 := A \setminus \{(k_1,R_1)\}$. Otherwise, set $A_1 := A$ and $F_1 := \{f_0\} \in \R(f_0)$. 
\item Pick $f_1 \in F_1$ such that for all $F_2 \in \R(f_1)$, there is $f_2 \in F_2$, such that, $\dots$, such that for each $F_N \in \R(f_{N-1})$, there is $f_N \in F_N$ such that for all $1 \leq i \leq n$, we have $f_N \not \leq l_i$.
\item If there is $(k_2,R_2) \in A_1$ with $f_1 \leq k_2$, set $F_2 := f_1 \down R_2 \in \R(f_1)$ and $A_2 := A_1 \setminus \{(k_2,R_2)\}$. Otherwise, set $A_2 := A_1$ and $F_2 := \{f_1\} \in \R(f_1)$. 
\item Pick $f_2 \in F_2$ such that for all $F_3 \in \R(f_2)$, there is $f_3 \in F_3$, such that, $\dots$, such that for each $F_N \in \R(f_{N-1})$, there is $f_N \in F_N$ such that for all $1 \leq i \leq n$, we have $f_N \not \leq l_i$.
\item $\dots$
\item If there is $(k_N,R_N) \in A_{N-1}$ with $f_{N-1} \leq k_N$, set $F_{N} := f_{N-1} \down R_N \in \R(f_{N-1})$ and $A_N := A_{N-1} \setminus \{(k_N,R_N)\}$. Otherwise, set $F_{N} := \{f_{N-1}\} \in \R(f_{N-1})$.
\item Pick $f_N \in F_N$ such that for all $1 \leq i \leq n$, we have $f_N \not \leq l_i$.
\end{itemize}
\setlength{\parindent}{0cm} \setlength{\parskip}{0cm}

We note that either $A_N = \emptyset$ or there is some $1 \leq j \leq N$ such that $f_{j-1} \not \leq k_j$ for each $(k_j,R_j) \in A_{j-1}$, in which case $f_N = f_{N-1} = \cdots = f_{j-1}$ and $A_N = A_{N-1} = \cdots = A_{j-1}$. Thus $f_N \not \leq k$ for each $(k,R) \in A_N$. Moreover, by construction we have that for each $(k,R) \in A \setminus A_N$ there is some $f \in R$ with $f_N \leq f$. Since  $f_N \not \leq l_i$ for each $1 \leq i \leq n$ by construction, we have shown that the principal character $\chi_{f_N}$ is an element of 
\begin{equation}
 K_A := \widehat L(l)  \mathbin{\Big\backslash}  \left[ \bigcup_{i=1}^n \widehat L(l_i) \cup \bigcup_{(k,R) \in A} \left(\widehat L(k) \mathbin{\Big\backslash} \bigcup_{f \in R} \widehat L(f)\right) \right].
\end{equation}
As $K_A \cap K_B = K_{A \cup B}$, the family $\{K_A \mid A \subseteq \mathscr S \text{ finite} \}$ has the finite intersection property. As each $K_A$ is compact, the intersection
\begin{equation}
 \bigcap_{A \subseteq \mathscr S \text{ finite}} K_A = \widehat L(l)  \mathbin{\Big\backslash}  \left[ \bigcup_{i=1}^n \widehat L(l_i) \cup \bigcup_{(k,R) \in \mathscr S} \left(\widehat L(k) \mathbin{\Big\backslash} \bigcup_{f \in R} \widehat L(f)\right) \right].
\end{equation}
is nonempty. This completes the proof that \eqref{eq:thorough2} implies \eqref{eq:thorough1}.
\eproof
\setlength{\parindent}{0cm} \setlength{\parskip}{0.5cm}

\subsection{A recursive construction}
\label{ss:IndResConst}

Now let us present a systematic way of constructing independent resolutions. Suppose $L_0$ is a locally finite weak semilattice and that $\R_0$ is a system of finite covers of $L_0$. 
\bdefin\label{d:recursive construction}
The locally finite weak semilattice of relative complements $L_1 = L_{L_0, \R_0}$ associated to $(L_0,\R_0)$ is the following set of idempotents in $\Zz L_0$:
\[
 L_1 \coloneq \menge{(l \mid \cF)}{l \in L_0, \, \emptyset \neq \cF \subseteq \R_0(l) \text{ finite}} \cup \gekl{0}.
\]
Given $(l\mid \cF) \in L_1$ and $F' \in \R_0(l)$, set
\begin{align*}
 R(l \mid \cF \mid F') &\coloneq \{(l \mid \cF \cup \gekl{F'}) \} \cup \menge{(f \mid f \down \cF)}{f \in F'},\\
 \R_1(l \mid \cF) &\coloneq \R_{L_0,\R_0}(l \mid \cF) \coloneq \menge{R(l \mid \cF \mid F')}{F' \in \R_0(l)}.
\end{align*}
\edefin

\bprop
$L_1$ is again a locally finite weak semilattice with the order coming from the idempotents in $\Zz L_0$, with the $\down$ operation given for $(k \mid \cE), (l \mid \cF) \in L_1$ by
\[ (k \mid \cE) \down (l \mid \cF) = \menge{(f \mid f \down (\cE \cup \cF))}{f \in k \down l}. \]
For each $(l\mid \cF) \in L_1$ and $F' \in \R_0(l)$ the set $R(l \mid \cF \mid F')$ is a finite cover of $(l \mid \cF)$ whose join in $\Zz L_0$ is
\[ (l \mid \cF \cup \{ F'\}) \vee \bigvee_{f \in F'} (f \mid f \down \cF) = (l \mid \cF). \]
Hence $\R_1$ is a system of finite covers in $L_1$. 
\eprop
\bproof
For $f \in k \down l$, we compute 
\[ f(k \mid \cE)(l \mid \cF) = (f \mid f \down \cE) (f \mid f \down \cF) = (f \mid f \down (\cE \cup \cF)), \]
so in particular $(f \mid f \down (\cE \cup \cF)) \leq (k \mid \cE), (l \mid \cF)$. Now suppose $0 \ne (g \mid \cG) \leq (k \mid \cE), (l \mid \cF)$. Because the coefficient of $g$ survives in the expansion of $(g \mid \cG) = (g \mid \cG)(k \mid \cE)$, we must have $g \leq k$, and similarly $g \leq l$. Thus $g \leq f$ for some $f \in k \down l$ and so 
\[ (g \mid \cG) \leq f(k \mid \cE)(l \mid \cF) = (f \mid f \down (\cE \cup \cF)).  \]
This establishes that $L_1$ is a weak semilattice with the claimed $\down$ structure. Local finiteness also follows, because if $X = \menge{(l_i \mid \cF_i)}{i \in I}$ is a finite subset of $L_1$, consider a finite sub-weak semilattice $T \subseteq L_0$ containing $\menge{l_i}{i \in I}$. Then $X$ is contained in the finite sub-weak semilattice
\[ \menge{(l \mid l \down \bigcup_{i \in J} \cF_i)}{l \in X, J \subseteq I}. \]
Now let $(l \mid \cF) \in L_1$ and $F' \in \R_0(l)$. For each $f \in F'$ the element $(f \mid f \down \cF)$ is orthogonal to $(l \mid \cF \cup \{F'\})$, so in $\Zz L_0$ we calculate
\begin{align*}
(l \mid \cF \cup \{ F'\}) \vee \bigvee_{f \in F'} (f \mid f \down \cF) & =  (l \mid \cF \cup \{ F'\}) + \bigvee_{f \in F'} (f \mid f \down \cF)  \\
& = (l \mid \cF)\left((l \mid F') + \bigvee_{f \in F'} f \right) \\
& = (l \mid \cF).
\end{align*}
It follows in particular that $R(l \mid \cF \mid F')$ is a finite cover of $(l \mid \cF)$. Finally, condition (M) follows from the computation that for $(l \mid \cF) \in L_1$, $F' \in \R_0(l)$ and $(k \mid \cE) \leq (l \mid \cF)$, we have 
\[ (k \mid \cE) \down R(l \mid \cF  \mid  F') = R(k \mid \cE  \mid  k \down F'). \qedhere \]
\eproof

\bremark
Let us collect the following immediate observations.
\setlength{\parindent}{0cm} \setlength{\parskip}{0cm}

\begin{enumerate}
\item[(i)] If $L_0$ is a semilattice, then so is $L_1$.
\item[(ii)] Given $S \acts L_0$ such that $\R_0$ is $S$-invariant, there is a canonical action $S \acts L_1$ with $s . (l \mid \cF) = (s.l \mid s. \cF)$ for $l \in \dom_L (s)$ by order isomorphisms of down-sets. Under this action, $\R_1$ is $S$-invariant and for all $l_1 = (l \mid \cF) \in L_1$, the stabiliser $S_{l_1}$ is a subgroup of $S_l$.
\item[(iii)] Recall that a cover $F$ of $l$ is trivial if $l \in F$. Then for all $0 \neq (l \mid \cF) \in L_1$, the set of non-trivial covers in $\R_1(l \mid \cF)$ embeds as a proper subset into the set of non-trivial covers in $\R_0(l)$. The reason is that $R(l \mid \cF \mid F')$ is trivial if $F' \in \cF$ or $F'$ is trivial.
\end{enumerate}
\eremark
\setlength{\parindent}{0cm} \setlength{\parskip}{0cm}

Let us now iterate this procedure to obtain a sequence $(L_0,\R_0)$, $(L_1,\R_1)$, $\dots$. We introduce the notation $(l \mid \cF  \mid  F')$ for $((l \mid \cF) \mid R(l \mid \cF  \mid  F'))$ in $\Zz L_1$. The canonical inclusion $L_{i+1} \into \Zz L_i$ induces an algebra homomorphism $\Zz L_{i+1} \xrightarrow{d_{i+1}} \Zz L_i$. Moreover, the algebra 
\[
 I_0 \coloneq \img d_1 = \spkl{\menge{l - \bigvee_{f \in F} f}{l \in L_0, \, F \in \R_0(l)}} \subseteq \Zz L_0
\]
generated by $L_1$ in $\Zz L_0$ is an ideal. The following is the key observation:
\btheo
\label{thm:IndRes}
The sequence
\[
 \dotso \to \Zz L_2 \xrightarrow{d_2} \Zz L_1 \xrightarrow{d_1} \Zz L_0 \to \Zz L_0 / I_0 \to 0
\]
is exact.
\etheo
\bproof
Exactness at $\Zz L_0$ and $\Zz L_0 / I_0$ is by construction. As the rest of the sequence is constructed iteratively it suffices to show exactness at $\Zz L_1$. The inclusion $\img d_2 \subseteq \ker d_1$ follows from the fact that for any $(l \mid \cF  \mid  F') \in L_2$, the join of the elements of $R(l \mid \cF \mid F')$ within $\Zz L_0$ is $(l \mid \cF)$. For the reverse inclusion $\ker d_1 \subseteq \img d_2$, let $f \in \ker d_1 \subseteq \Zz L_1$. Suppose $f$ can be expressed as
\[f = \sum_{i \in A} a_i (l_i \mid \cF_i),\]
for some finite index set $A$ and integers $a_i$. Our strategy is to show that $f$ can be translated by an element of $\img d_2$ to an element 
\[ g = \sum_{j \in B} b_j (k_j \mid \cE_j) \in f + \img d_2 \]
so that the finite weak semilattice generated within $L_1$ by all the terms present in the expression strictly reduces. Iterating this reduces $f$ to $0$ mod $\img d_2$. 
\setlength{\parindent}{0cm} \setlength{\parskip}{0.5cm}

Thus, let $L_f$ denote the weak semilattice generated within $L_1$ by $\menge{l_i}{i \in A} \cup \bigcup_{i \in A, \, F \in \cF_i} F$. It suffices to build $g$ as above with $k_j$ and $\cE_j$ all taken from $L_f$ but avoiding some maximal element. We can assume that each $a_i \in \Zz$ is non-zero and that $(l_i \mid \cF_i)$ is non-zero and distinct for each $i \in A$. Pick $i_0 \in A$ with $l_{i_0}$ maximal in $\{l_i \mid i \in A \}$ (and thus in $L_f$) and set $A_0 = \menge{i \in A}{l_i = l_{i_0}}$. 
Since $l_{i_0}$ is maximal, the $l_{i_0}$-coefficient of $d_1(f) = 0 \in \Zz L_0$ is $\sum_{i \in A_0} a_i$, which must therefore vanish.
 Set $\cF := \bigcup_{i \in A_0} \cF_i$ and set
\[ g := \sum_{i \in A_0} a_i \left( (l_{i_0} \mid \cF_i) - (l_{i_0} \mid \cF) - \bigvee_{F \in \cF \setminus \cF_i} (l_{i_0} \mid \cF_i  \mid  F) \right) + \sum_{i \in A \setminus A_0} a_i(l_i \mid \cF_i) \in \Zz L_1. \]
Note that the $(l_{i_0} \mid \cF_i  \mid  F)$ above are understood to be in $\Zz L_1$, and are in $\img d_2$ by construction. Since $\sum_{i \in A_0} a_i (l_{i_0} \mid \cF) = 0$, we have $g \in f + \img d_2$. For each $i \in A_0$ we compute in $\Zz L_1$:
\begin{align*}
\bigvee_{F \in \mathcal F \setminus \mathcal F_i} (l_{i_0} \mid \mathcal F_i  \mid  F) & = \bigvee_{F \in \mathcal F \setminus \mathcal F_i} \left( (l_{i_0} \mid \mathcal F_i) - (l_{i_0} \mid \mathcal F_i \cup \{ F \}) - \bigvee_{k \in F} (k \mid k \down \mathcal F_i) \right) \\
& = (l_{i_0} \mid \mathcal F_i) - \prod_{F \in \mathcal F \setminus \mathcal F_i} \left( (l_{i_0} \mid \mathcal F_i \cup \{ F \}) + \bigvee_{k \in F} (k \mid k \down \mathcal F_i) \right) 
\end{align*}
Each $k \in F \in \mathcal F$ is strictly below $l_{i_0}$ because each $(l_i \mid \mathcal F_i)$ is non-zero. The product term above expands out in $\Zz L_1$ as
\begin{align*}
\prod_{F \in \mathcal F \setminus \mathcal F_i} \left( (l_{i_0} \mid \mathcal F_i \cup \{ F \}) + \bigvee_{k \in F} (k \mid k \down \mathcal F_i) \right)  & =  \prod_{F \in \mathcal F \setminus \mathcal F_i} (l_{i_0} \mid \mathcal F_i \cup \{ F \}) + \sum_{j \in B_i} c_{i,j} (k_{i,j} \mid \mathcal E_{i,j}) \\
& = (l_{i_0} \mid \mathcal F) + \sum_{j \in B_i} c_{i,j} (k_{i,j} \mid \mathcal E_{i,j})
\end{align*}
for some finite index set $B_i$ and for each $j \in B_i$ some $c_{i,j} \in \Zz$, some $k_{i,j} \in L_f$ strictly below $l_{i_0}$ and some finite sets $\mathcal E_{i,j}$ of finite covers of $k_{i,j}$ contained within $L_f$. We have thus constructed an element
\[ g = \sum_{i \in A_0}  \sum_{j \in B_i} a_i c_{i,j} (k_{i,j} \mid \cE_{i,j})   + \sum_{i \in A \setminus A_0} a_i(l_i \mid \cF_i) \in f + \img d_2 \]
which is entirely expressed within $L_f \setminus \{ l_{i_0} \}$, a proper sub-weak semilattice of $L_f$.
\eproof
\setlength{\parindent}{0cm} \setlength{\parskip}{0.5cm}

\bremark
We again collect some observations:
\setlength{\parindent}{0cm} \setlength{\parskip}{0cm}

\begin{enumerate}
\item[(i)] If each $F \in \R_{N}(l)$ is trivial for all $l \in L_{N}$, then $L_{N + 1} = \gekl{0}$. Since 
\[\sup_{l \in L_i} \# \menge{F \in \R_i(l)}{F \text{ non-trivial}}\]
strictly decreases as $i$ increases, then $L_{N+1} = 0$ if $N \coloneq \sup_{l \in L_0} \# \{F \in \R_0(l) \mid F \text{ non-trivial}\}$ is finite. In this case we obtain an exact sequence of finite length 
\[
 0 \to \Zz L_{N} \to \dotso \to \Zz L_2 \to \Zz L_1 \to \Zz L_0 \to \Zz L_0 / I_0 \to 0.
\]
\item[(ii)] Let us dualize. Let $I_i \subseteq \Zz L_i$ be the kernel of $\Zz L_i \to \Zz L_{i-1}$. Then $I_i$ gives rise to an open subspace $U_i$ of $\widehat{L_i}$. Set $C_i \coloneq \widehat{L_i} \setminus U_i$. Exactness gives us an isomorphism $\Zz L_i / I_i \cong I_{i-1}$ and thus a homeomorphism $C_i \cong U_{i-1}$. 
\item[(iii)] Now assume that $S$ acts on $L_0$ such that $\R_0$ is $S$-invariant. We then obtain $S$-actions on $L_i$ and all $\R_i$ are $S$-invariant. Also, the exact sequence from Theorem~\ref{thm:IndRes} is $S$-invariant. This gives us an independent resolution $G_i \coloneq S \ltimes \widehat{L_i}$ of $G \coloneq S \ltimes X$, where $X \coloneq C_0 \cong U_0$. In particular, if the finiteness condition in item (i) holds, then we obtain a finite length independent resolution of $G$.
\end{enumerate}
\eremark
\setlength{\parindent}{0cm} \setlength{\parskip}{0.5cm}

Combining the above remarks with Lemma \ref{l:thorough}, we obtain the following.

\btheo\label{t:resolution}
Let $S$ be an inverse semigroup with actions on a totally disconnected locally compact Hausdorff space $X$ and a locally finite weak semilattice $L$, and let $U \colon L \to \cO_c(X)$ be an $S$-equivariant generating representation of $L$ on $X$. Suppose that $\R$ is a thorough $S$-invariant system of concrete finite covers in $L$. Set $(L_0,\R_0) = (L,\R)$  and consider the recursively constructed $(L_k,\R_k)_{k \geq 0}$ (Definition \ref{d:recursive construction}). Then the sequence $(G_i := S \ltimes \widehat{L_i})_{i \geq 0}$ of groupoids is an independent resolution 
\[\cdots \to G_i \to \cdots \to G_0 \to G \to \emptyset \]
of $G := S \ltimes X$. Moreover, if $N \coloneq \sup_{l \in L} \# \{F \in \R(l) \mid F \text{ non-trivial}\}$ is finite, then $G_i = \emptyset$ for $i \geq N+1$. If $L$ is a semilattice, then so is $L_k$ for $k \geq 0$.
\etheo

\bex
\label{ex:HRG5}
Let $\Lambda$ be a source-free row-finite higher rank graph of rank $k$ as in Examples \ref{ex:HRG1}, \ref{ex:HRG2}, \ref{ex:HRG3} and \ref{ex:HRG4}. Our construction yields an independent resolution of length $k$. More precisely, the sequence of weak semilattices constructed above looks as follows:
\begin{align*}
 L_0 &= L_{\Lambda} = \Lambda \cup \gekl{0},\\
 L_1 &= \menge{(\lambda \mid I_1)}{\lambda \in \Lambda; \, \emptyset \neq I_1 \subseteq \gekl{1, \dotsc, k}} \cup \gekl{0},\\
 L_2 &= \menge{(\lambda \mid I_1 \mid I_2)}{\lambda \in \Lambda; \, \emptyset \neq I_1, I_2 \subseteq \gekl{1, \dotsc, k} \text{ disjoint}} \cup \gekl{0},\\
 &\dotso,\\
 L_j &= \menge{(\lambda \mid I_1 \mid \dotso \mid I_j)}{\lambda \in \Lambda; \, \emptyset \neq I_1, \dotsc, I_j \subseteq \gekl{1, \dotsc, k} \text{ disjoint}} \cup \gekl{0},\\
 &\dotso,\\
 L_k &= \menge{(\lambda \mid I_1 \mid \dotso \mid I_k)}{\lambda \in \Lambda; \, \emptyset \neq I_1, \dotsc, I_k \subseteq \gekl{1, \dotsc, k} \text{ disjoint}} \cup \gekl{0},
\end{align*}
where for $L_k$, we must have $\# \, I_1 = \cdots = \# \, I_k = 1$. Here we define recursively 
\[
 (\lambda \mid I_1) \coloneq (\lambda \mid \{\lambda \Lambda^{(i)} \mid i \in I_1 \}), 
\]
and 
\[
 (\lambda \mid I_1 \mid \dotso \mid I_j) \coloneq \prod_{i \in I_j} ((\lambda \mid I_1 \mid \dotso \mid I_{j-1}) \mid R((\lambda \mid I_1 \mid \dotso \mid I_{j-1}) \mid \lambda \Lambda^{(i)}))
\]
in $\Zz L_{j-1}$, where we define $R((\lambda \mid I_1 \mid \dotso \mid I_{j-1}) \mid \lambda \Lambda^{(i)})$ as the union of the subsets 
\[
 \gekl{(\lambda \mid I_1 \mid \dotso \mid I_{j-1} \cup \gekl{\lambda \Lambda^{(i)}})}
\]
and 
\[
 \menge{ \left(f \suchthat f \down \menge{R((\lambda \mid I_1 \mid \dotso \mid I_{j-2}) \mid \lambda \Lambda^{(i')})}{i' \in I_{j-1}} \right)}{f \in R((\lambda \mid I_1 \mid \dotso \mid I_{j-2}) \mid \lambda \Lambda^{(i)})}
\]
of $L_{j-1}$.
\eex

\section{Discretisation}
\label{s:Disc}

Now we come to the crucial idea of discretisation. Let $L$ be a locally finite weak semilattice equipped with an action of an inverse semigroup $S$. Write $L\reg \coloneq L \setminus \gekl{0}$, and consider the restriction $S \acts L\reg$ of our action to $L\reg$. Viewing $L\reg$ as a discrete topological space, we view $S \acts L\reg$ as an action of $S$ by partial homeomorphisms on $L\reg$ and form the partial transformation groupoid $S \ltimes L\reg$.

The idea of discretisation is to find a suitable weak equivalence between $S \ltimes L\reg$ and $S \ltimes \widehat{L}$.

Note that we always have an embedding $L\reg \into \widehat{L}$ sending $k \in L\reg$ to the character $\chi_k$ given by $\chi_k(l) = 1$ if and only if $k \leq l$. Hence we obtain an induced embedding $S \ltimes L\reg \into S \ltimes \widehat{L}$. But this is not the desired weak equivalence. Instead, we construct a proper étale correspondence $\Omega \colon S \ltimes L\reg \to S \ltimes \widehat{L}$.

To define \'etale correspondences, we need the notion of an action of an \'etale groupoid $G$ on a topological space $X$. Such an action consists of an \emph{anchor map} $\tau \colon X \to G^0$ assumed continuous, and an \emph{action map} $(g,x) \mapsto g . x \colon G \times_{G^0} X = \{ (g,x) \in G \times X \mid s(g) = \tau(x) \} \to X$ satisfying $\tau(g.x) = r(g)$, $\tau(x).x = x$ and $(gh).x=g.(h.x)$ whenever the expressions make sense. The action is said to be \emph{free} if $g . x = x$ implies $g \in G^0$, \emph{proper} if $G \times_{G^0} X \to X \times X$, $(g,x) \mapsto (g.x,x)$ is proper, and \emph{\'etale} if the anchor map is a local homeomorphism.  
The action groupoid $G \ltimes X$ is an \'etale groupoid with arrow space $G \times_{G^0} X$ and unit space $X$ with multiplication given by $(g,h.x)(h,x) = (gh,x)$.

\begin{definition}[\'Etale correspondence, see \cite{AKM22}]\label{d:correspondence}
Let $G$ and $H$ be \'etale groupoids. A $G$-$H$ \emph{bispace} is a topological space $\Omega$ equipped with a left action of $G$ and a right action of $H$ that commute; that is, the left anchor $\rho \colon \Omega \to G^0$ is $H$-invariant, the right anchor $\sigma \colon \Omega \to H^0$ is $G$-invariant and $g . (\omega . h) = (g . \omega) . h$ for compatible $g \in G$, $\omega \in \Omega$ and $h \in H$. An \emph{\'etale correspondence} $\Omega \colon G \to H$ is a $G$-$H$ bispace $\Omega$ such that the right action $\Omega \rightacts H$ is free, proper and \'etale.
The correspondence $\Omega$ is \emph{proper} if the map $\Omega/H \to G^0$ induced by the left anchor is proper. If $G = G^0$ and $H = H^0$ are locally compact Hausdorff spaces, the data of an \'etale correspondence reduces to a locally compact Hausdorff space $Z$ equipped with a continuous map $\rho \colon Z \to X$ and a local homeomorphism $\sigma \colon Z \to Y$; we refer to $Z \colon X \to Y$ as a \emph{topological correspondence}. Note that $Z \colon X \to Y$ is proper if and only if $\rho \colon Z \to X$ is proper.
\end{definition}

Let us introduce the following notation.

\bdefin
Let $G$ be an étale groupoid, let $X_1$ and $X_2$ be right and left $G$-spaces respectively with anchors $\sigma \colon X_1 \to G^0$ and $\rho \colon X_2 \to G^0$ and let $K \subseteq G$ be an open subgroupoid. The \emph{fibre product} of $X_1$ and $X_2$ over $K$
\[ X_1 \times_K X_2 \coloneq \{ (x_1,x_2) \in X_1 \times X_2 \mid \sigma(x_1) = \rho(x_2) \in K^0 \} / K \]
is the quotient defined by $(x_1 . g, x_2) \sim (x_1, g . x_2)$ for $x_1 \in X_1$, $x_2 \in X_2$ and $g \in K$ with $\sigma(x_1) = r(g)$ and $\rho(x_2) = s(g)$. We equip $X_1 \times_K X_2$ with the quotient topology and write $[x_1,x_2]$ for the class of a pair $(x_1,x_2)$. If $K_1, K_2 \subseteq G$ are open subgroupoids with actions $X_1 \rightacts K_1 \acts X_2 \rightacts K_2 \acts X_3$ such that the left and right actions on $X_2$ commute, we identify $(X_1 \times_{K_1} X_2) \times_{K_2} X_3$ with $X_1 \times_{K_1} (X_2 \times_{K_2} X_3)$ via $[[x_1,x_2],x_3] \mapsto [x_1,[x_2,x_3]]$ and write $[x_1,x_2,x_3] \in X_1 \times_{K_1} X_2 \times_{K_2} X_3$.
\edefin

\bremark
In the above, $X_1$ and $X_2$ may be viewed as partial $K$-spaces in that the anchor maps to $K^0$ are only defined on some open subset.
\eremark

The composition $\Lambda \circ \Omega \colon G \to K$ of \'etale correspondences $\Omega \colon G \to H$ and $\Lambda \colon H \to K$ is given by the $G$-$H$ bispace $\Omega \times_H \Lambda$. \'Etale correspondences between \'etale groupoids form a $(2,1)$-category, with isomorphisms of bispaces as $2$-morphisms. In this paper we shall only apply properties of the ordinary category whose morphisms are isomorphism classes of \'etale correspondences, thus we often work with \'etale correspondences (sometimes implicitly) only up to isomorphism.

\begin{definition}\label{d:discretisation}
Let $L$ be a locally finite weak semilattice. The (topological) \emph{discretisation correspondence} $\Omega^0 \colon L\reg \to \widehat L$ is the topological correspondence 
\[ \Omega^0 = \coprod_{l \in L\reg} \widehat L(l) = \left\{(l,\chi) \in L\reg \times \widehat L \suchthat \chi(l) = 1 \right\}, \]
with left anchor map $\rho(l,\chi) = l$ and right anchor map $\sigma(l,\chi) = \chi$. The discretisation correspondence is proper because $\widehat L(l)$ is compact for each $l \in L\reg$.
\end{definition}

When $L$ is equipped with an action of an inverse semigroup $S$, the discretisation correspondence $\Omega^0$ also inherits an action of $S$ with $\dom_{\Omega^0} (s) = \coprod_{l \in \dom_{L\reg} (s)} L(l)$ for each $s \in S$ and $s . (l,\chi) = (s . l, s . \chi)$. The discretisation correspondence $\Omega^0 \colon L\reg \to \widehat L$ becomes an $S$-equivariant topological correspondence in the following sense.

\begin{definition}
Let $S$ be an inverse semigroup. An \emph{$S$-equivariant topological correspondence} is a topological correspondence $Z \colon X \to Y$ equipped with actions of $S$ on $X$, $Y$ and $Z$, such that the anchor maps $\rho \colon Z \to X$ and $\sigma \colon Z \to Y$ are $S$-equivariant with $\rho^{-1}(\dom_X (s)) = \dom_Z (s)$ for each $s \in S$. 
\end{definition}
The composition of two $S$-equivariant topological correspondences is again an $S$-equivariant topological correspondence.
The condition that $\rho^{-1}(\dom_X (s)) = \dom_Z (s)$ for each $s \in S$ ensures well-definition of the na\"ive action of $S \ltimes X$ on $Z$. 
\begin{lemma}
Let $S$ be an inverse semigroup and let $Z \colon X \to Y$ be an $S$-equivariant topological correspondence. Then $S \ltimes X$ acts on $Z$ with anchor $\rho \colon Z \to X$ via $[s,x] . z = s . z$ so that $(S \ltimes X) \ltimes Z = S \ltimes Z$.
\end{lemma}

This means that $S \ltimes X \acts S \ltimes Z$ given by $[s,x] . [t,z] = [st,z]$ defines an action correspondence $S \ltimes X \to S \ltimes Z$. Composing this with the \'etale correspondence associated to the \'etale homomorphism $[s,z] \mapsto [s,\sigma(z)] \colon S \ltimes Z \to S \ltimes Y$, we obtain the following \'etale correspondence.

\bdefin
Let $S$ be an inverse semigroup and let $Z \colon X \to Y$ be an $S$-equivariant topological correspondence. The \emph{associated \'etale correspondence} $Z(S) \colon S \ltimes X \to S \ltimes Y$ has underlying space $Z(S) = Z \times_Y (S \ltimes Y)$, with left action given by $[s,x] . (z,[t,y]) = (s . z, [st,y])$ and right action by $(z,[s,y]) . [t,y'] = (z,[st,y'])$. For $g = [s,x]$ and $z \in Z$ with $\rho(z) = x$, we write $g|_z = [s,\sigma(z)]$, which is the unique element of $S \ltimes Y$ satisfying $g . (z,\sigma(z)) = (g . z, g|_z) = (g . z,\sigma(z)) . g|_z \in Z(S)$. 
\edefin

Note that $Z(S) \colon S \ltimes X \to S \ltimes Y$ is proper if and only if $Z \colon X \to Y$ is proper. The construction of the \'etale correspondence associated to an $S$-equivariant topological correspondence is functorial with respect to composition of correspondences.

\bdefin
\label{def:DiscCorr}
Let $L$ be a locally finite weak semilattice equipped with the action of an inverse semigroup $S$. The \emph{discretisation correspondence} $\Omega \colon S \ltimes L\reg \to S \ltimes \widehat{L}$ is the \'etale correspondence $\Omega = \Omega^0(S)$ associated to the $S$-equivariant topological discretisation correspondence $\Omega^0 \colon L\reg \to \widehat L$. As a space, we have
\begin{equation*}
\Omega = \coprod_{l \in L\reg} \menge{[s,\chi] \in S \ltimes \widehat{L}}{\chi \in \dom_{\widehat L}(s), \, s.\chi \in \widehat{L}(l)}.    
\end{equation*}
It is proper because $\Omega^0$ is proper.
\edefin

\bremark
We have $\Omega \cong \Omega^0 \times_{\widehat{L}} (S \ltimes \widehat{L})$. However, $\Omega$ cannot be described as $(S \ltimes L\reg) \times_{L\reg} \Omega^0$. The reason is that for $[l,s,\chi] \in \Omega$, $l$ might not lie in $\ran_{L\reg} (s)$. For the same reason, the canonical action $S \acts \Omega^0$ does not describe $\Omega$ as $S \ltimes \Omega^0$.
\eremark

\subsection{Discretisation in groupoid homology}
\label{ss:Disc_H}

Let $L$ be a locally finite weak semilattice equipped with an action of an inverse semigroup $S$. As explained in \cite[Corollary~3.6]{Mil2}, $\Omega$ induces the map $\Hlgy_*(\Omega) \colon \Hlgy_*(S \ltimes L\reg) \to \Hlgy_*(S \ltimes \widehat{L})$ in groupoid homology. The argument for the following result is identical to that in \cite[Example~3.10]{Mil2}. We reproduce it here for convenience.
\btheo
\label{thm:Disc_H}
For every inverse semigroup action $S \acts L$ on a locally finite weak semilattice $L$, the discretisation correspondence $\Omega$ from Definition~\ref{def:DiscCorr} induces an isomorphism in groupoid homology
\[
 \Hlgy_*(\Omega) \colon \Hlgy_*(S \ltimes L\reg) \isom \Hlgy_*(S \ltimes \widehat{L}).
\]
\etheo
\setlength{\parindent}{0cm} \setlength{\parskip}{0cm}

\bproof
The \'etale correspondence $\Omega \colon S \ltimes L\reg \to S \ltimes \widehat L$ decomposes into the correspondences associated to the action $S \ltimes L\reg \acts \Omega^0$ and the \'etale homomorphism $S \ltimes \Omega^0 \to S \ltimes \widehat{L}$. Set $G = S \ltimes L\reg$ and $H = S \ltimes \widehat{L}$ and $K = S \ltimes \Omega^0$. Following \cite[Examples~3.8 and 3.9]{Mil2}, the induced map $\Hlgy_*(\Omega)$ is induced by a chain map $\Phi \colon \mathbb Z[G^\bullet] \to \mathbb Z[H^\bullet]$ of Matui's chain complexes. The chain map is induced at $n \geq 0$ (contravariantly) by the map $K^n \to G^n, [s_{n-1},\dots,s_0,l,\chi] \mapsto [s_{n-1},\dots,s_0,l]$ and (covariantly) by the local homeomorphism $K^n \to H^n, [s_{n-1},\dots,s_0,l,\chi] \mapsto [s_{n-1},\dots,s_0,\chi]$. This boils down to the map
\begin{align*}
\Phi_n \colon \mathbb Z[G^n] & \to \mathbb Z[H^n] \\
\delta_x & \mapsto 1_{V_x},
\end{align*}
where for each $x = [s_{n-1},\dots,s_0,l] \in G^n$, the compact open set $V_x \subseteq H^n$ is given by 
\[ V_x = \{[s_{n-1},\dots,s_0,\chi] \mid \chi \in \widehat L(l) \}. \]
We will show that $\Phi_n$ is in fact an isomorphism for each $n \geq 0$.
\setlength{\parindent}{0cm} \setlength{\parskip}{0.5cm}

For $x,y \in G^n$ we set $x \leq y$ if $V_x \subseteq V_y$, this happens exactly if $x = [s_{n-1},\dots,s_0,k]$ and $y = [s_{n-1},\dots,s_0,l]$ with $k \leq l$. It follows that $\{V_x \mid x \leq y \}$ is a generating independent concrete locally finite weak semilattice in $V_y$. By Lemma \ref{lemma on identifying Z algebras}, the restriction $\Zz[\{x \mid x \leq y \}] \to \Zz[V_y]$ of $\Phi_n$ is an isomorphism. Surjectivity of $\Phi_n$ follows because $\{V_y \mid y \in G^n\}$ covers $H^n$ and hence $\{\Zz[V_y] \mid y \in G^n \}$ spans $\Zz[H^n]$. 

For injectivity of $\Phi_n$, let $F \subseteq G^n$ be finite and let $(a_x)_{x \in F}$ be integers with $\sum_{x \in F} a_x 1_{V_x} = 0$. Pick a maximal element $y = [s_{n-1},\dots,s_0,l] \in F$. Then $0 = \sum_{x \in F} a_x 1_{V_x}$ evaluated at the element $[s_{n-1},\dots,s_0,\chi_l] \in H^n$ recovers $a_y$ because $[s_{n-1},\dots,s_0,\chi_l] \in V_x$ if and only if $y \leq x$. Thus $a_y = 0$ and by induction we conclude $\sum_{x \in F} a_x \delta_x = 0$.
\eproof
\setlength{\parindent}{0cm} \setlength{\parskip}{0.5cm}

Since $S \ltimes L\reg$ is a discrete groupoid, it is Morita equivalent to a disjoint union of isotropy groups. Therefore, we arrive at the following consequence.
\bcor
\label{cor:OSF_H}
For every inverse semigroup action $S \acts L$ on a locally finite weak semilattice $L$, we have the following isomorphism in groupoid homology:
\[
 \bigoplus_{[l] \in S \backslash L\reg} \Hlgy_*((S \ltimes L\reg)_l^l) \cong \Hlgy_*(S \ltimes \widehat{L}).
\]
Here $S \backslash L\reg$ is the set of orbits for the action $S \acts L\reg$, and $(S \ltimes L\reg)_l^l$ denotes the isotropy at the unit $l \in L\reg = (S \ltimes L\reg)^{0}$.
\ecor

\subsection{Discretisation in K-theory}
\label{ss:Disc_K}

The discretisation correspondence $\Omega \colon S \ltimes L\reg \to S \ltimes \widehat L$ induces a proper $\cs$-correspondence $\cs(\Omega) \colon \cs(S \ltimes L\reg) \to \cs(S \ltimes \widehat L)$. In fact, because $\Omega$ is built from the $S$-equivariant topological correspondence $\Omega^0 \colon L\reg \to \widehat L$, this descends to a proper $\cs$-correspondence $\csr(\Omega) \colon \csr(S \ltimes L\reg) \to \csr(S \ltimes \widehat L)$ at the reduced level (see Proposition \ref{free and reduced}).
Our goal is to prove that this induces an isomorphism in K-theory $\K_*(\csr(S \ltimes L\reg) ) \isom \K_*(\csr(S \ltimes \widehat L))$.

\btheo\label{thm:Disc_K}
Let $L$ be a countable locally finite weak semilattice equipped with the action of a countable inverse semigroup $S$ and suppose that $S \ltimes \widehat L$ is Hausdorff. Then the discretisation correspondence $\Omega$ induces an isomorphism
\begin{equation*}
 \Ktop_*(\Omega) \colon \Ktop_*(S \ltimes L\reg) \isom \Ktop_*(S \ltimes \widehat{L})
\end{equation*}
of the left-hand side of the Baum--Connes conjecture for the groupoids $S \ltimes L\reg$ and $S \ltimes \widehat{L}$. If $S \ltimes \widehat L$ and $S \ltimes L\reg$ satisfy the Baum--Connes conjecture, then $\csr(\Omega) \colon \csr(S \ltimes L\reg) \to \csr(S \ltimes \widehat L)$ induces an isomorphism
\[ \K_*(\csr(\Omega)) \colon \K_*(\csr(S \ltimes L\reg) ) \isom \K_*(\csr(S \ltimes \widehat L)).\]
\etheo

The following is the analogue of Corollary~\ref{cor:OSF_H} (and the argument is the same).
\bcor
\label{cor:OSF_K}
Let $L$ and $S$ be as in Theorem~\ref{thm:Disc_K}. Under the same hypotheses as in Theorem~\ref{thm:Disc_K}, we have the following isomorphism in K-theory:
\[
 \bigoplus_{[l] \in S \backslash L\reg} \K_*(\csr((S \ltimes L\reg)_l^l)) \cong \K_*(S \ltimes \widehat{L}).
\]
Here $S \backslash L\reg$ is the set of orbits for the action $S \acts L\reg$, and $(S \ltimes L\reg)_l^l$ denotes the isotropy at the unit $l \in L\reg = (S \ltimes L\reg)^{0}$.
\ecor

\bremark
\label{rem:BC_discGpd}
$S \ltimes L\reg$ satisfies the Baum--Connes conjecture if and only if, for every unit $l \in L\reg = (S \ltimes L\reg)^{0}$, the isotropy group $(S \ltimes L\reg)_l^l$ satisfies the Baum--Connes conjecture. The reason is that $S \ltimes L\reg$ is Morita equivalent to the groupoid $\coprod_{[l] \in S \backslash L\reg} (S \ltimes L\reg)_l^l$, and satisfying the Baum--Connes conjecture is preserved under Morita equivalence. The latter fact is well known, but we could not find an explicit reference. It follows for example from the functoriality results in \cite{Mil3}. It was stated before \cite[Corollary 7.4]{Mil3} that because each isotropy group in $S \ltimes E\reg$ embeds in an isotropy group of $S \ltimes \widehat E$, the former groupoid inherits satisfaction of the Baum--Connes conjecture from the latter. The second-named author would like to take the opportunity to correct this mistake; \cite[Corollary 7.4 and Example 7.6]{Mil3} should include the extra hypothesis that the relevant subgroups satisfy the Baum--Connes conjecture.
\eremark

The proof of Theorem \ref{thm:Disc_K} is long and technical, so let us first summarise the main ideas. Section \ref{proof section} is dedicated to the proof. 
Satisfaction of the Baum--Connes conjecture reduces the problem from operator K-theory to topological K-theory. The main tool we will apply is \cite[Corollary 7.1]{Mil3}, which gives a sufficient condition for the map (see \cite[Example 5.11]{Mil3}) 
\[ \Ktop_*(\Omega) \colon \Ktop_*(S \ltimes L\reg) \to \Ktop_*(S \ltimes \widehat L) \]
induced by $\Omega$ to be an isomorphism. Setting $X = L^\times$, $Y = \widehat L$, $G = S \ltimes X = S \ltimes L\reg$ and $H = S \ltimes Y = S \ltimes \widehat L$, this condition asks that for each $n \geq 0$, an induced map of derived functors
\begin{equation}\label{eq:derived functor map}
\Lz_n(\Ind_\Omega,\alpha_\Omega,f_\Omega) \colon \Lz_n^{\mathcal F_G}\K_*(G \ltimes C_0(X)) \to \Lz_n^{\mathcal F_H}\K_*(H \ltimes C_0(Y)) 
\end{equation}
with respect to certain families $\mathcal F_G$ and $\mathcal F_H$ of proper groupoids is an isomorphism.

The derived functors and the functoriality appearing in \eqref{eq:derived functor map} require a considerable amount of KK-theoretic machinery from \cite{Mil3} (and references therein \cite{BP24, Mil1, MN06, MN10, Meyer08, PY22}) to formulate, which we cover in Section \ref{KK prelim}. Morally speaking, the group $\Lz_n^{\mathcal F_G}\K_*(G \ltimes C_0(X))$ is the $n$th homology group of $G$ for some non-standard homology theory of groupoids. Indeed, for Hausdorff ample groupoids $G$ with torsion-free isotropy, this will coincide with ordinary groupoid homology $\Hlgy_n(G)$ for $* = 0$ and vanish for $* = 1$ (this is the case already covered in \cite{Mil3}). For now, it suffices to say that $\Lz_n(\Ind_\Omega,\alpha_\Omega,f_\Omega)$ is the map induced in homology by a particular chain map 
\begin{equation}\label{eq:summary chain map}
\Phi \colon \K_*(G \ltimes P_\bullet) \to \K_*(H \ltimes Q_\bullet)
\end{equation}
involving certain chain complexes $P_\bullet$ and $Q_\bullet$ of $G$-$\cs$-algebras and $H$-$\cs$-algebras. 

In Section \ref{groupoid models for maps}, we provide convenient groupoid models for the $\cs$-algebras $G \ltimes P_n$ and $H \ltimes Q_n$ appearing in \eqref{eq:summary chain map}, and provide concrete \'etale groupoid correspondences which induce the chain map $\Phi$.

In Section \ref{isomorphism subsection}, we show that these \'etale groupoid correspondences modelling $\Phi$ can themselves be expressed as discretisation correspondences where the acting inverse semigroup is finite. In this situation, we can prove an isomorphism in K-theory in an elementary fashion, completing the proof of Theorem \ref{thm:Disc_K}.

\section{Proof of K-theoretic discretisation}\label{proof section}

\subsection{KK-theoretic preliminaries}\label{KK prelim}

Here we provide a summary of the construction from \cite{Mil3} of the derived functors appearing in \eqref{eq:derived functor map}. The construction takes place within the Meyer--Nest framework of homological algebra in triangulated categories \cite{MN06, MN10, Meyer08}, utilising the triangulated structure on the Kasparov category $\KK^G$ of a second countable Hausdorff étale groupoid $G$ \cite[Section 1.1]{BP24}. The derived functors are the left derived functors of the K-theory functor $\K_*(G \ltimes -) \colon \KK^G \to \ab$ with respect to a countable family $\mathcal F$ of proper open subgroupoids of $G$ satisfying condition (P) \cite[\S 3]{Mil3}; for every $x \in G^0$ and every finite subgroup $\Gamma \leq G^x_x$ of the isotropy group $G^x_x$, there is a member $K \in \mathcal F$ with $\Gamma \subseteq K$. Before spelling out exactly what being a derived functor with respect to $\mathcal F$ means, let us describe a canonical choice of the family $\mathcal F$ from ~\cite[Example 3.6]{Mil3} for the transformation groupoid $G = S \ltimes X$ of the action of an inverse semigroup $S$ on a locally compact Hausdorff space $X$.
For each $e \in E\reg = \{ e \in S \mid e^2 = e, e \ne 0 \}$ and finite subgroup $F \subseteq S_e \coloneq \menge{s \in S}{ses^{-1} = e}$, set
\begin{align*}
G(F) & \coloneq \menge{[s,l] \in G}{s \in F, \, l \in \dom_X(e)}. 
\end{align*}
Setting $F_{\down} = \menge{sd}{s \in F, \, d \in E} = \menge{sd}{s \in F, \, d \in E, \, d \leq e}$, which is the inverse semigroup generated by $F$ and $\{d \in E \mid d \leq e \}$, the action $S \acts X$ restricts to an action $F_{\down} \acts \dom_X(e) $, and we have a canonical isomorphism $G(F) \cong F_{\down} \ltimes \dom_X(e)$. Note that $G(F)$ is a compact open subgroupoid of $G$. Now set
\begin{align*}
    & \cF_{\rm fin}^X \coloneq \menge{G(F)}{e \in E\reg, \, F \subseteq S_e \text{ finite subgroup}}.
\end{align*}
Then the family $\cF_{\rm fin}^X$ satisfies condition (P) in the sense of \cite[\S~3]{Mil3}. 

Now suppose $Z \colon X \to Y$ is an $S$-equivariant topological correspondence, and write $G = S \ltimes X$, $H = S \ltimes Y$ and $Z(S) \colon G \to H$ for the resulting \'etale correspondence. Then $\cF_{\rm fin}^X$ and $\cF_{\rm fin}^Y$ are compatible under $Z(S) \colon G \to H$ in the sense of \cite[Definition 3.20]{Mil3} (see \cite[Example 3.21]{Mil3}). Assume further that $G$ and $H$ are second countable and Hausdorff. Following \cite[\S~3.6]{MN10} and \cite[\S~2.1]{PY22} (see also \cite[\S~2]{Mil3}), consider the induction and restriction functors
\begin{align*}
I_G \colon \prod_{K \in \cF_{\rm fin}^X} \KK^K & \to \KK^G  & \dashv & &R_G \colon \KK^G & \to \prod_{K \in \cF_{\rm fin}^X} \KK^K \\
(C_K)_{K \in \cF_{\rm fin}^X} & \mapsto \bigoplus_{K \in \cF_{\rm fin}^X} \Ind^G_K C_K && & B & \mapsto (\Res^K_G(B))_{K \in \cF_{\rm fin}^X}.
\end{align*}
which satisfy an adjunction $I_G \dashv R_G$ \cite[Theorem 2.3]{BP24}. Setting $P_n \coloneq (I_G R_G)^{n+1} C_0(X)$, the counit $\epsilon \colon I_G R_G \Rightarrow \id$ of the adjunction gives rise to a chain complex
\begin{equation}\label{eq:projective resolution}
\cdots \to P_n \xrightarrow{\delta_n} \cdots \xrightarrow{\delta_1} P_0 \xrightarrow{\delta_0} C_0(X) \to 0 
\end{equation}
in $\KK^G$ with 
\begin{equation*}
\delta_n := \sum_{i=0}^n (-I_G R_G)^i(\epsilon_{(I_G R_G)^{n-i} C_0(X)}) \colon (I_G R_G)^{n+1} C_0(X) \to (I_G R_G)^n C_0(X).
\end{equation*} This is a $\ker R_G$-projective resolution in the sense of \cite{MN10}. The derived functor $\Lz_n^{\cF_{\rm fin}^X}\K_*(G \ltimes C_0(X))$ is realised as the $n$th homology group of the chain complex
\[ \cdots \to \K_*(G \ltimes P_n) \to \cdots \to \K_*(G \ltimes P_0) \to 0. \]
The derived functor $\Lz_n^{\cF_{\rm fin}^Y}\K_*(H \ltimes C_0(Y))$ for $H = S \ltimes Y$ is obtained by performing the analogous construction for $H$ and $\cF_{\rm fin}^Y$ on $C_0(Y)$, setting $Q_n \coloneq (I_H R_H)^{n+1} C_0(Y)$ to obtain a $\ker R_H$-projective resolution $Q_\bullet \to C_0(Y)$.

Functoriality of these derived functors is developed in \cite[\S 5]{Mil3}; let us explain how the $S$-equivariant topological correspondence $Z \colon X \to Y$ induces this functoriality to produce a map 
\begin{equation}\label{eq:derived functor map two}
\Lz_n(\Ind_{Z(S)},\alpha_{Z(S)},f_{Z(S)}) \colon \Lz_n^{\cF_{\rm fin}^X}\K_*(G \ltimes C_0(X)) \to \Lz_n^{\cF_{\rm fin}^Y}\K_*(H \ltimes C_0(Y)). 
\end{equation}
via \cite[Proposition 5.12]{Mil3}. The three components $\Ind_{Z(S)}$, $\alpha_{Z(S)}$ and $f_{Z(S)}$ are constructions we utilise from \cite{Mil1}. Let us describe what they do in our setting.

Let $\Omega \colon G \to H$ be a second countable Hausdorff étale correspondence of second countable Hausdorff étale groupoids with $G^0 = X$, $H^0 = Y$. We will utilise the \emph{induction functor} 
\[ \Ind_\Omega \colon  \KK^H \to \KK^G\]
constructed in \cite{Mil1}, which is a triangulated functor, but we shall need its precise form only for commutative $H$-$\cs$-algebras. Given a locally compact Hausdorff $H$-space $W$, the induced algebra $\Ind_\Omega C_0(W)$ is given by
\[ \Ind_\Omega C_0(W) = C_0(\Omega \times_H W). \]
If $Z \colon V \to W$ is a proper $H$-equivariant topological correspondence, the induced morphism $\Ind_\Omega \cs(Z) \colon \Ind_\Omega C_0(V) \to \Ind_\Omega C_0(W)$ is modelled by the proper $G$-equivariant topological correspondence $\Omega \times_H Z \colon \Omega \times_H V \to \Omega \times_H W$.

We will chiefly consider the case of $\Omega \colon G \to H$ given by $Z(S) \colon S \ltimes X \to S \ltimes Y$ for an $S$-equivariant topological correspondence $Z \colon X \to Y$. In this case, the $G$-space $\Omega \times_H W$ is canonically isomorphic to $Z \times_Y W$, with the action of $G = S \ltimes X$ given by $[s,x] . (z,w) = (s . z, [s,\sigma(z)] . w)$, or, in alternate notation, $g  . (z,w) = (g . z, g|_z . w)$. We thus make the identification
\begin{equation}
\Ind_{Z(S)} C_0(Y) = C_0(Z \times_Y W). 
\end{equation}

The induction functor is based on subgroupoid induction functors, and indeed for a proper open subgroupoid $K \subseteq G$, the bispace $G \acts G_{K^0} \rightacts K$ is an étale correspondence $G \to K$, and for any $K$-$\cs$-algebra $B$ the induced algebra $\Ind_{G_{K^0}} B$ is the generalised fixed point algebra $(s^* B)^K$. We refer to this induced algebra as $\Ind^G_K B$, and we note that it is naturally $G$-equivariantly Morita equivalent to the induced algebra $\Ind^G_K B$ constructed in \cite{BP24} (see \cite[Remark 2.1]{BP24}).\footnote{The point of the induced algebra $\Ind^G_K B$ constructed in \cite{BP24} is that it works even for open subgroupoids $K \subseteq G$ for which the action of $K$ on $G_{K^0}$ is not proper, but we do not require this extra generality.} Thus for a locally compact Hausdorff $K$-space $W$, the induced algebra $\Ind^G_K C_0(W)$ is given by $C_0(G \times_K W)$.

From \cite{Mil1} we shall also utilise the \emph{induction crossed product}
\[ \Omega \ltimes B \colon G \ltimes \Ind_\Omega B \to H \ltimes B \]
associated to a Hausdorff étale correspondence $\Omega \colon G \to H$ and an $H$-$\cs$-algebra $B$. This is a proper $\cs$-correspondence natural in $B$, and we write 
\[\alpha_\Omega(B) \colon \K_*(G \ltimes \Ind_\Omega B) \to \K_*(H \ltimes B) \]
for the induced map in K-theory.
We shall only need the precise form of $\Omega \ltimes B$ for the \'etale correspondence $\Omega = Z(S) \colon S \ltimes X \to S \ltimes Y$ of an $S$-equivariant topological correspondence $Z \colon X \to Y$ and a commutative $H$-$\cs$-algebra $B = C_0(W)$. In this setting, 
the algebra $G \ltimes \Ind_{Z(S)} C_0(W)$ is modelled by $S \ltimes (Z \times_Y W)$ and $Z(S) \ltimes C_0(W)$ is modelled by the proper étale correspondence
\begin{equation}\label{crossed product model}
(Z \times_Y W)(S) \colon S \ltimes (Z \times_Y W) \to S \ltimes W
\end{equation}
induced by the $S$-equivariant local homeomorphism $(z,w) \mapsto w \colon Z \times_Y W \to W$.

The final ingredient from \cite{Mil1} we need is the $G$-equivariant $*$-homomorphism $f_{Z(S)} \colon C_0(X) \to \Ind_\Omega C_0(Y)$ in the case that $\Omega \colon G \to H$ is proper. This is simply induced by the $G$-equivariant proper map $\rho \colon Z \to X$, using the identification $\Ind_{Z(S)} C_0(Y) = C_0(Z)$. The composition $\alpha_{Z(S)}(C_0(Y)) \circ \K_*(G \ltimes f_{Z(S)}) \colon \K_*(\cs(G)) \to \K_*(\cs(H))$ recovers $\K_*(\cs(\Omega))$.

Let us return to the setting of an $S$-equivariant topological correspondence $Z \colon X \to Y$, which we moreover assume to be proper. Once again assume that $G = S \ltimes X$ and $H = S \ltimes Y$ are second countable and Hausdorff and assume that $Z$ is second countable. We obtain a second countable Hausdorff proper \'etale correspondence $Z(S) \colon G \to H$. 
We now describe the construction of the induced map $\Lz_n(\Ind_{Z(S)},\alpha_{Z(S)},f_{Z(S)})$ \eqref{eq:derived functor map two} via \cite[Proposition 5.12]{Mil3}. As the families $\cF_{\rm fin}^X$ and $\cF_{\rm fin}^Y$ are compatible under ${Z(S)}$ in the sense of \cite[Definition 3.20]{Mil3}, it follows by \cite[Example 5.8]{Mil3} that $(\Ind_{Z(S)}, \alpha_{Z(S)}, f_{Z(S)})$ is an ABC cycle in the sense of \cite[Definition 5.7]{Mil3}.
Following \cite[Proposition 5.12]{Mil3}, we consider from \eqref{eq:projective resolution} the projective resolutions $P_\bullet \to C_0(X)$ and $Q_\bullet \to C_0(Y)$ with $P_n = (I_G R_G)^{n+1} C_0(X)$ and $Q_n = (I_H R_H)^{n+1} C_0(Y)$ and the induced resolution $\Ind_{Z(S)} Q_\bullet \to \Ind_{Z(S)} C_0(Y)$ in $KK^G$. A chain map $\tilde f \colon P_\bullet \to \Ind_{Z(S)} Q_\bullet$ in $\KK^G$ which sits over $f_{Z(S)} \colon C_0(X) \to \Ind_{Z(S)} C_0(Y)$ exists and is unique up to chain homotopy for abstract reasons, but as we shall see there is a natural construction in our setting. The map $\Lz_n(\Ind_{Z(S)},\alpha_{Z(S)},f_{Z(S)})$ is induced in homology by the chain map $\Phi \colon \K_*(G \ltimes P_\bullet) \to \K_*(H \ltimes Q_\bullet)$ given by the composition
\begin{equation}\label{chain map crucial}
\Phi \colon \K_*(G \ltimes P_\bullet) \xrightarrow{\K_*(G \ltimes \tilde f)} \K_*(G \ltimes \Ind_{Z(S)} Q_\bullet) \xrightarrow{\alpha_{Z(S)}(Q_\bullet)} \K_*(H \ltimes Q_\bullet). 
\end{equation}
If $\Phi$ is an isomorphism, then so too is $\Lz_n(\Ind_{Z(S)},\alpha_{Z(S)},f_{Z(S)})$, and through \cite[Corollary 7.1]{Mil3} we deduce that $\Ktop_*({Z(S)}) \colon \Ktop_*(G) \to \Ktop_*(H)$ is an isomorphism. Let us summarise what we have reduced the problem to.

\bprop\label{isomorphism proposition before concrete models}
Let $S$ be a countable inverse semigroup and let $Z \colon X \to Y$ be a proper $S$-equivariant topological correspondence with $X$, $Y$ and $Z$ second countable and $G = S \ltimes X$ and $H = S \ltimes Y$ Hausdorff.
Consider the projective resolutions $P_\bullet \to C_0(X)$, $Q_\bullet \to C_0(Y)$ as in \eqref{eq:projective resolution}. 
Suppose $\tilde f \colon P_\bullet \to \Ind_{Z(S)} Q_\bullet$ is a chain map in $\KK^G$ which sits over $f_{Z(S)} \colon C_0(X) \to \Ind_{Z(S)} C_0(Y)$ such that the composition
\[ \Phi \colon \K_*(G \ltimes P_\bullet) \xrightarrow{\K_*(G \ltimes \tilde f)} \K_*(G \ltimes \Ind_{Z(S)} Q_\bullet) \xrightarrow{\alpha_{Z(S)}(Q_\bullet)} \K_*(H \ltimes Q_\bullet) \]
is an isomorphism of chain complexes. Then $\Ktop_*({Z(S)}) \colon \Ktop_*(G) \to \Ktop_*(H)$ is an isomorphism.
\eprop

\subsection{Concrete groupoid models for maps}\label{groupoid models for maps}

Let us identify a concrete groupoid model for the chain map \eqref{chain map crucial}. 

We first consider the chain map $P_\bullet \to \Ind_{Z(S)} Q_\bullet$, whose components are elements in $\KK^G$. We will model these elements using proper topological correspondences that are equivariant with respect to $G$.

\bdefin
Let $G$ be an \'etale groupoid. A \emph{$G$-equivariant topological correspondence} is a topological correspondence $Z \colon X \to Y$ equipped with actions of $G$ on $X$, $Y$ and $Z$ such that the anchor maps are $G$-equivariant. The transformation groupoid $G \ltimes Z$ inherits the structure of an \'etale correspondence $G \ltimes Z \colon G \ltimes X \to G \ltimes Y$ with anchors $\rho \circ r$ and $\sigma \circ s$ and actions given by $(g_1,x) . (g_2,z) = (g_1 g_2, z)$ and $(g_1,z) . (g_2,y) = (g_1 g_2, g_2^{-1} . z)$.
\edefin

We describe the construction of a $G$-equivariant $\cs$-correspondence from a $G$-equivariant topological correspondence in Appendix \ref{correspondence functor}, where we moreover consider $G$-equivariant \'etale correspondences and establish compatibility with crossed products and composition. In particular, proper $G$-equivariant topological correspondences $Z \colon X \to Y$ functorially induce morphisms $C_0(X) \to C_0(Y)$ in $\KK^G$ by Proposition \ref{prop:G equivariant correspondence functor}.

The algebras $P_n = (I_G R_G)^{n+1} C_0(X)$ and $Q_n = (I_H R_H)^{n+1} C_0(Y)$ admit the direct sum decompositions
\begin{align*}
P_n & = \bigoplus_{F_n,\dotsc, F_0} \Ind^G_{G(F_n)} \Res^{G(F_n)}_G \cdots \Ind^G_{G(F_0)} \Res^{G(F_0)}_G C_0(X) \\
Q_n & = \bigoplus_{F_n,\dotsc, F_0} \Ind^H_{H(F_n)} \Res^{H(F_n)}_H \cdots \Ind^H_{H(F_0)} \Res^{H(F_0)}_H C_0(Y)
\end{align*}
with summands indexed by tuples $(F_n, \dotsc, F_0)$, where $F_m$ is a finite subgroup of the maximal subgroup $S_{e_m} = \{ s \in S \mid s^* s = e_m = s s^* \}$ at some $e_m \in E\reg$. We shall see that the chain map \eqref{chain map crucial} inducing $\Lz_n(\Ind_{{Z(S)}}, \alpha_{{Z(S)}}, f_{{Z(S)}})$ respects these direct sum decompositions, so that we may work with individual summands. Moreover, $P_n$, $\Ind_{{Z(S)}}  Q_n$ and $Q_n$ as well as their direct summands are commutative $\cs$-algebras, and we may therefore work at the space level.

Given $n \geq 0$, $e_0, \dots, e_n \in E\reg$ and finite subgroups $F_i \leq S_{e_i}$ for $0 \leq i \leq n$ as above, let $G_m \coloneq G(F_m)$ and $H_m \coloneq H(F_m)$. 
The spectra of the direct summands of $P_n$ and $Q_n$ associated to $(F_n,\dotsc,F_0)$ are given respectively by the spaces
\begin{align}\label{definition of X(F_n...F_0) Y(F_n...F_0)}
 X(F_n,\dotsc,F_0) & \coloneq G \times_{G_n} \cdots \times_{G_1} G \times_{G_0} X, \\
 Y(F_n, \dotsc, F_0) & \coloneq  H \times_{H_n} \cdots \times_{H_1} H \times_{H_0} Y. \nonumber
\end{align}
Formally, a general element of $X(F_n, \dotsc, F_0)$ is given by the class $[g_n,\dotsc,g_0,x]$ of an $(n+2)$-tuple, where $g_i = [s_i,x_i] \in G$ for some $s_i \in S$ and $x_i \in X$ and $x = x_0 \in \dom_X (e_0)$. But since the elements $x_{i+1} \in X$ for $0 \leq i < n$ are determined by $x_{i+1} = s_i . x_i$, we will often instead denote this element by $[s_n,\dotsc,s_0,x]$.\footnote{We hope that no confusion is caused by the competing notations $[g_n,\dotsc,g_0,x]$ and $[s_n,\dotsc,s_0,x]$.} Similarly, we write $[s_n, \dotsc, s_0, y]$ for the element $[[s_n,y_n], [s_{n-1},y_{n-1}], \dotsc, [s_0,y_0], y] \in Y(F_n, \dotsc, F_0)$, so that 
\begin{align*}
 X(F_n,\dotsc,F_0) & = \{ [s_n,\dotsc,s_0,x] \mid s_0,\dotsc,s_n \in S, x \in \dom_X (s_n \cdots s_0) \}, \\
 Y(F_n,\dotsc,F_0) & = \{ [s_n,\dotsc,s_0,y] \mid s_0,\dotsc,s_n \in S, y \in \dom_Y (s_n \cdots s_0) \}. 
\end{align*}
For each $(F_n,\dotsc,F_0)$ we construct a proper $G$-equivariant topological correspondence 
\begin{equation}\label{eq: chain map correspondences}
Z(F_n,\dotsc,F_0) \colon X(F_n, \dotsc, F_0) \to Z \times_Y Y(F_n,\dotsc,F_0)
\end{equation}
as follows, noting that $Z \times_Y Y(F_n,\dotsc,F_0)$ is the spectrum of the associated direct summand of $\Ind_{Z(S)} Q_n$. The correspondence has underlying space
\begin{align*}
    Z(F_n, \dotsc, F_0) & \coloneq G \times_{G_n} \cdots \times_{G_1} G \times_{G_0} Z \\
    & = \{ [g_n,\dotsc,g_0,z] \mid (g_n,\dots,g_0) \in G^{n+1}, z \in \rho^{-1}(s(g_0)) \} \\
    & = \{ [s_n,\dotsc,s_0,z] \mid s_0,\dotsc,s_n \in S, z \in \dom_Z (s_n \cdots s_0) \},
\end{align*}
which inherits a left $G$-action from its leftmost factor of $G$. We define $G$-equivariant maps (under the two pictures of elements described above)
\begin{align}\label{anchor map definitions}
\rho^{(n)} \colon  Z(F_n,\dots,F_0) & \to X(F_n,\dots,F_0) \nonumber \\
 [g_n,\dots,g_0,z]  & \mapsto [g_n,\dots,g_0,\rho(z)] \nonumber\\
 [s_n,\dots,s_0,z]  & \mapsto [s_n,\dots,s_0,\rho(z)] \nonumber\\
 \sigma^{(n)} \colon  Z(F_n,\dots,F_0) & \to Y(F_n,\dots,F_0) \nonumber\\
 [g_n,\dots,g_0,z]  & \mapsto [g_n |_{g_{n-1} \cdots g_0 . z}, \dotsc, g_0 |_z,\sigma(z)] \\
 [s_n,\dots,s_0,z]  & \mapsto [s_n,\dots,s_0,\sigma(z)] \nonumber\\
 p^{(n)} \colon  Z(F_n,\dots,F_0) & \to Z \nonumber\\
 [g_n,\dots,g_0,z]  & \mapsto g_n \cdots g_0 . z \nonumber\\
 [s_n,\dots,s_0,z]  & \mapsto s_n \cdots s_0 . z. \nonumber
\end{align}
We set the anchor maps for the topological correspondence \eqref{eq: chain map correspondences} to be $\rho^{(n)} \colon  Z(F_n,\dots,F_0) \to X(F_n,\dots,F_0)$ and $p^{(n)} \times \sigma^{(n)} \colon Z(F_n,\dots,F_0) \to Z \times_Y Y(F_n,\dots,F_0)$. Since $Z(F_n,\dots,F_0)$ is a $G$-equivariant proper topological correspondence, it induces a morphism $C_0(X(F_n,\dotsc,F_0)) \to \Ind_{Z(S)} C_0(Y(F_n,\dotsc,F_0))$ in $\KK^G$, and summing over $(F_n,\dotsc,F_0)$ we arrive at a morphism $P_n \to \Ind_{Z(S)} Q_n$. Our aim is to show that this defines a chain map $P_\bullet \to \Ind_{Z(S)} Q_\bullet$ over $f_{Z(S)} \colon C_0(X) \to \Ind_{Z(S)} C_0(Y)$.

The boundary maps of the projective resolutions $P_\bullet = (I_G R_G)^{\bullet +1} C_0(X) \to C_0(X)$ and $Q_\bullet = (I_H R_H)^{\bullet +1} C_0(Y) \to C_0(Y)$ are constructed from the counits of the adjunctions $I_G \dashv R_G$ and $I_H \dashv R_H$; we shall make use of an explicit form for these counits. In the following lemma we note that a $G$-equivariant local homeomorphism $V \to W$ may be considered as a $G$-equivariant topological correspondence $V \colon V \to W$ with left anchor map the identity.

\begin{lemma}\label{counit model}
Let $G$ be a second countable Hausdorff étale groupoid, let $K \subseteq G$ be a proper open subgroupoid and let $W$ be a second countable locally compact Hausdorff $G$-space with anchor map $\tau \colon W \to G^0$. Then after identifying $\Ind^G_K \Res^K_G C_0(W)$ with $C_0(G \times_K W)$, the counit 
\[\epsilon_{C_0(W)} \colon \Ind^G_K \Res^K_G C_0(W) \to C_0(W)\]
of the adjunction $\Ind^G_K \dashv \Res^K_G$ is induced by the $G$-equivariant local homeomorphism
\begin{align*}
G \times_K W & \to W \\
[g,w] & \mapsto g . w,
\end{align*}
for $g \in G$ and $w \in W$ with $s(g) = \tau(w) \in K^0$.
\end{lemma}
\setlength{\parindent}{0cm} \setlength{\parskip}{0cm}

\begin{proof}
Let us write $\widetilde \Ind^G_K \colon \KK^K \to \KK^G$ for the induction functor constructed in \cite{BP24}. On the $K$-$\cs$-algebra $\Res^K_G C_0(W)$, this is given by the crossed product $C_0(G \times_{K^0} W) \rtimes K$, where $K$ acts on $G \times_{K^0} W = \{ (g,w) \in G \times W \mid s(g) = \tau(w) \in K^0 \}$ via $(g,w) . k = (gk,k^{-1} . w)$. Let $G \overset{K}{\times} W$ denote the transformation groupoid of this action, so that $\widetilde \Ind^G_K \Res^K_G C_0(W) = \cs(G \overset{K}{\times} W)$. The $G$-equivariant Morita equivalence $\Ind^G_K \Res^K_G C_0(W) \to \widetilde \Ind^G_K \Res^K_G C_0(W)$ of $\cs$-algebras mentioned in \cite[Remark 2.1]{BP24} is induced by the $G$-equivariant Morita equivalence (of groupoids)
\[ G \times_K W \acts G \times_{K^0} W \rightacts G \overset{K}{\times} W \]
via Proposition \ref{prop:G equivariant correspondence functor}. 
The counit constructed in \cite[Theorem 2.3]{BP24} at $C_0(W)$ is the composition of a $G$-equivariant $*$-homomorphism $\kappa_{C_0(W)} \colon \widetilde \Ind^G_K \Res^K_G C_0(W) \to \widetilde \Ind^G_G \Res^G_G C_0(W)$ and a $G$-equivariant Morita equivalence $X^G_{C_0(W)} \colon \widetilde \Ind^G_G \Res^G_G C_0(W) \to C_0(W)$. The map $\kappa_{C_0(W)}$ is induced on the groupoid level by the canonical open inclusion $G \overset{K}{\times} W \hookrightarrow G \overset{G}{\times} W$, while $X^G_{C_0(W)}$ is induced by the $G$-equivariant Morita equivalence $G \overset{G}{\times} W \to W$ of groupoids with bispace $G \times_{G^0} W$. The composition of the $G$-equivariant étale correspondences
\[\begin{tikzcd}
	{G \times_K W} & {G \overset{K}{\times} W} & {G \overset{G}{\times}W} & W
	\arrow["{G \times_{K^0} W}", from=1-1, to=1-2]
	\arrow["{G \overset{G}{\times} W}", from=1-2, to=1-3]
	\arrow["{G \times_{G^0} W}", from=1-3, to=1-4]
\end{tikzcd}\]
has $G \times_K W$-$W$ bispace
\begin{align*}
(G \times_{K^0} W) \times_{G \overset{K}{\times} W} (G \overset{G}{\times} W) \times_{G \overset{G}{\times}W} (G \times_{G^0} W) & \cong (G \times_{K^0} W) \times_{G \overset{K}{\times} W} (G \times_{G^0} W) \\
& \cong G \times_K W,
\end{align*}
where the final $G \times_K W$ is the $G$-equivariant topological correspondence with left anchor map the identity and right anchor map $[g,w] \mapsto g . w$. It follows by Proposition \ref{prop:G equivariant correspondence functor} that our counit $\epsilon_{C_0(W)}$ is induced by this $G$-equivariant local homeomorphism.
\end{proof}
\setlength{\parindent}{0cm} \setlength{\parskip}{0.5cm}

\blemma\label{we actually get a chain map}
Let $S$ be a countable inverse semigroup and let $Z \colon X \to Y$ be a proper $S$-equivariant topological correspondence with $X$, $Y$ and $Z$ second countable and $G = S \ltimes X$ and $H = S \ltimes Y$ Hausdorff.
Consider the projective resolutions $P_\bullet \to C_0(X)$, $Q_\bullet \to C_0(Y)$ as in \eqref{eq:projective resolution}. Then the collection of $G$-equivariant topological correspondences $Z(F_n,\dotsc,F_0) \colon X(F_n,\dotsc,F_0) \to Y(F_n,\dotsc,F_0)$ described in \eqref{eq: chain map correspondences} induce a chain map $P_\bullet \to \Ind_{Z(S)} Q_\bullet$ over $f_{Z(S)} \colon C_0(X) \to \Ind_{Z(S)} C_0(Y)$.
\elemma
\setlength{\parindent}{0cm} \setlength{\parskip}{0cm}

\bproof
It suffices to show that the maps in $\KK^G$ induced by the $G$-equivariant topological correspondences $Z(F_n,\dotsc,F_0) \colon X(F_n,\dotsc,F_0) \to Z \times_Y Y(F_n,\dotsc,F_0)$ are compatible with the face maps appearing in the projective resolutions. Setting $P_{-1} = C_0(X)$, then for $n \geq 0$, the map $\delta_n \colon P_n \to P_{n-1}$ is given by $\sum_{i=0}^n (-I_G R_G)^i \epsilon_{(I_G R_G)^{n-i} C_0(X)}$. The face map $(I_G R_G)^{n-i} \epsilon_{(I_G R_G)^{i} C_0(X)}$ maps the $(F_n,\dotsc,F_0)$-component of $P_n$ into the $(F_n,\dotsc,\widehat F_{i}, \dotsc, F_0)$-component of $P_{n-1}$ (and maps the $F_0$-component of $P_0$ to $P_{-1}$). By Lemma \ref{counit model} this is induced by the $G$-equivariant local homeomorphism
\begin{equation*}
\epsilon \colon [s,x] \mapsto s . x \colon X(F_0) \to X
\end{equation*}
for $n = 0$ and the $G$-equivariant local homeomorphism
\begin{align*}
\epsilon_{n,i} \colon X(F_n,\dotsc,F_0) & \to X(F_n,\dotsc,\widehat F_{i}, \dotsc, F_0) \\
[s_n,\dotsc,s_0,x] & \mapsto    \begin{cases}  
                                [s_n,\dotsc,s_1,s_0 . x] & i = 0 \\
                                [s_n,\dotsc, s_{i} s_{i-1},\dotsc, s_0,x] & 0 < i \leq n  
                                \end{cases}
\end{align*}
for $n > 0$. The analogous summand of $\Ind_{Z(S)} \delta_n \colon \Ind_{Z(S)} Q_n \to \Ind_{Z(S)} Q_{n-1}$ (with $Q_{-1} = C_0(Y)$) is induced by the $G$-equivariant local homeomorphism
\begin{equation*}
\epsilon' \colon (z,[s,y]) \mapsto z \colon Z \times_Y Y(F_0) \to Z 
\end{equation*}
for $n = 0$ and the $G$-equivariant local homeomorphism
\begin{align*}
\epsilon'_{n,i} \colon Z \times_Y Y(F_n,\dotsc,F_0) & \to Z \times_Y Y(F_n,\dotsc,\widehat F_{i}, \dotsc, F_0) \\
(z,[s_n,\dotsc,s_0,y]) & \mapsto \begin{cases}  
                                    (z,[s_n,\dotsc,s_1,s_0 . y]) & i = 0 \\
                                    (z,[s_n,\dotsc, s_{i} s_{i-1},\dotsc, s_0,y]) & 0 < i \leq n 
                                    \end{cases}
\end{align*}
for $n > 0$. We have broken our maps down into $G$-equivariant topological correspondences; by Proposition \ref{prop:G equivariant correspondence functor} it suffices to prove that the following diagrams commute up to isomorphism of $G$-equivariant topological correspondences:
\begin{equation}\label{chain map diagram one}
\begin{tikzcd}
	{X(F_0)} & X \\
	{Z \times_Y Y(F_0)} & {Z}
	\arrow["\epsilon"', from=1-1, to=1-2]
	\arrow["{Z(F_0)}", from=1-1, to=2-1]
	\arrow["{Z}", from=1-2, to=2-2]
	\arrow["{\epsilon'}"', from=2-1, to=2-2]
\end{tikzcd}
\end{equation}
\begin{equation}\label{chain map diagram two}
\begin{tikzcd}
	{X(F_n,\dots,F_0)} & {X(F_n,\dots,\widehat F_{i},\dots,F_0)} \\
	{Z \times_Y Y(F_n,\dots,F_0)} & {Z \times_Y Y(F_n,\dots,\widehat F_{i},\dots,F_0)}
	\arrow["{\epsilon_{n,i}}"', from=1-1, to=1-2]
	\arrow["{Z(F_n,\dots,F_0)}", from=1-1, to=2-1]
	\arrow["{Z(F_n,\dots,\widehat F_{i},\dots,F_0)}", from=1-2, to=2-2]
	\arrow["{\epsilon'_{n,i}}"', from=2-1, to=2-2]
\end{tikzcd}
\end{equation}
Here we understand $Z \colon X \to Z$ as the $G$-equivariant topological correspondence with right anchor the identity and left anchor $\rho \colon Z \to X$, which induces $f_{Z(S)} \colon C_0(X) \to \Ind_\Omega C_0(Y)$ in $\KK^G$. For Diagram \eqref{chain map diagram one}, the clockwise composition is the $G$-equivariant topological correspondence 
\[\begin{tikzcd}[row sep = 0.2em]
	{X(F_0)} & {X(F_0) \times_{\epsilon,X,\rho} Z} & {Z} \\
	{[s_0,x]} & {([s_0,x], z)} & {z}
	\arrow[from=1-2, to=1-1]
	\arrow[from=1-2, to=1-3]
	\arrow[maps to, from=2-2, to=2-1]
	\arrow[maps to, from=2-2, to=2-3]
\end{tikzcd}\]
with the diagonal action $g . ([s_0,x],z) = (g . [s_0,x], g . z)$.
The anticlockwise composition is 
\[\begin{tikzcd}[row sep = 0.2em]
	{X(F_0)} & {Z(F_0)} & {Z} \\
	{[s_0,\rho(z)]} & {[s_0,z]} & {s_0 . z,}
	\arrow[from=1-2, to=1-1]
	\arrow[from=1-2, to=1-3]
	\arrow[maps to, from=2-2, to=2-1]
	\arrow[maps to, from=2-2, to=2-3]
\end{tikzcd}\]
with the usual $G$-action on $Z(F_0)$.
These compositions are isomorphic via the $G$-equivariant homeomorphism
\begin{align*}
Z(F_0) & \to X(F_0) \times_{\epsilon,X,\rho} Z  \\
[s_0,z] & \mapsto ([s_0,\rho(z)], s_0 . z).
\end{align*}
For Diagram \ref{chain map diagram two}, the clockwise composition has right anchor
\begin{align*}
    & X(F_n, \dotsc, F_0) \times_{X(F_n, \dotsc, \widehat F_i, \dotsc, F_0)} Z(F_n, \dotsc, \widehat F_i, \dotsc, F_0)\\
    \to \ & Z \times_Y Y(F_n,\dots,\widehat F_{i},\dots,F_0)\\
        & \begin{cases} ([s_n, \dotsc, s_0, x], [s_n,\dotsc, s_1,  z] ) & i = 0 \\
                    ([s_n, \dotsc, s_0, x], [s_n, \dotsc, s_{i+1}, s_i s_{i-1}, \dotsc, s_0, z]) & 0 < i \leq n \end{cases}\\
    \ma \ & \begin{cases}   (s_n \cdots s_1 . z,[s_n, s_{n-1}, \dots, s_1, \sigma(z)]] & i = 0  \\
                            (s_n \cdots s_0 . z, [s_n, s_{n-1},\dots, s_{i+1},s_i s_{i-1}, \dots, s_0, \sigma(z)]) & 0 < i \leq n 
                            \end{cases}
\end{align*}
with left anchor the left projection map, and $G$-action the diagonal one. The anticlockwise composition has underlying space $Z(F_n,\dots,F_0)$. The left anchor is inherited from $Z(F_n,\dots,F_0) \colon X(F_n,\dots,F_0) \to Y(F_n,\dots,F_0)$, and the right anchor is given by
\begin{align*}
    & Z(F_n, \dotsc, F_0) \\
    \to \ & Z \times_Y Y(F_n,\dots,\widehat F_{i},\dots,F_0)\\
    &  [s_n, \dotsc, s_0, z]  \\
    \ma \ & \begin{cases}   (s_n \cdots s_0 . z,[s_n, s_{n-1},\dots, s_1, s_0 . \sigma(z)]) & i = 0  \\
                            (s_n \cdots s_0 . z,[s_n,s_{n-1},\dots,s_{i+1},s_i s_{i-1}, \dots, s_0, \sigma(z)]) & 0 < i \leq n. 
                            \end{cases}
\end{align*}
It is then straightforward to verify that the following is an isomorphism of $G$-equivariant correspondences, noting that $\sigma(s_0 . z) = s_0 . \sigma(z)$.
\begin{align*}
    & Z(F_n, \dots, F_0)\\
    \to \ & X(F_n, \dotsc, F_0) \times_{X(F_n, \dotsc, \widehat F_i, \dotsc, F_0)} Z(F_n, \dotsc, \widehat F_i, \dotsc, F_0) \\
          & [s_n, \dotsc, s_0, z] \\
    \ma \ & \begin{cases} ([s_n, \dotsc, s_0, \rho(z)], [s_n, \dots, s_1, s_0 . z]) & i = 0 \\
                    ([s_n, \dotsc, s_0, \rho(z)], [s_n, \dotsc, s_{i+1}, s_i s_{i-1}, \dotsc, s_0, z]) & 0 < i \leq n  \end{cases} \qedhere
\end{align*}
\eproof
\setlength{\parindent}{0cm} \setlength{\parskip}{0.5cm}

The induction correspondence $Z(S) \ltimes Q_n \colon G \ltimes \Ind_{Z(S)} Q_n \to H \ltimes Q_n$ is given at the groupoid level by the étale correspondences (summing over $(F_n,\dotsc,F_0)$)
\begin{equation}\label{eq:correspondence A}
(Z \times_Y Y(F_n,\dotsc,F_0))(S) \colon S \ltimes (Z \times_Y Y(F_n,\dotsc,F_0)) \to S \ltimes Y(F_n,\dotsc,F_0) 
\end{equation}
induced by the $S$-equivariant local homeomorphisms $Z \times_Y Y(F_n,\dotsc,F_0) \to Y(F_n,\dotsc,F_0)$ as described in \eqref{crossed product model}. 
Moreover, by Proposition \ref{crossed product correspondence identification} the crossed product of the $G$-equivariant chain map $P_\bullet \to \Ind_{Z(S)} Q_\bullet$ from Lemma \ref{we actually get a chain map} is modelled at the groupoid level by the étale correspondences 
\begin{equation}\label{eq:correspondence B}
G \ltimes {Z}(F_n,\dots,F_0) \colon G \ltimes X(F_n,\dots,F_0) \to G \ltimes (Z \times_Y Y(F_n,\dots,F_0)) 
\end{equation}
induced by the $G$-equivariant topological correspondences $Z(F_n,\dots,F_0) \colon X(F_n,\dots,F_0) \to Z \times_Y Y(F_n,\dots,F_0)$.

\begin{lemma}\label{lem:G and S equivariant correspondences}
Let $G = S \ltimes W$ be an \'etale groupoid arising from the action of an inverse semigroup $S$ on a locally compact Hausdorff space $W$. Then any $G$-equivariant topological correspondence $Z \colon X \to Y$ is an $S$-equivariant topological correspondence, with the actions of $S$ on $X$, $Y$ and $Z$ induced by the actions of $G$. Moreover, the induced \'etale correspondence $Z(S) \colon S \ltimes X \to S \ltimes Y$ is isomorphic to $G \ltimes Z \colon G \ltimes X \to G \ltimes Y$.
\end{lemma}

Altogether, the composition of the above correspondences is induced by the $S$-equivariant proper topological correspondence 
\begin{equation}\label{labelled correspondence 1}
    Z(F_n, \dotsc, F_0)  \colon X(F_n, \dotsc, F_0) \to Y(F_n, \dotsc, F_0)
\end{equation}
with anchor maps $\rho^{(n)}$ and $\sigma^{(n)}$ as given in \eqref{anchor map definitions}:

\blemma\label{chain map composed model}
Let $S$ be an inverse semigroup and let $Z \colon X \to Y$ be an $S$-equivariant topological correspondence. Let $e_0, \dots, e_n \in E(S)$ be idempotents and let $F_i \subseteq S_{e_i}$ be finite subgroups. Then the composition of the \'etale correspondences   \eqref{eq:correspondence A} and \eqref{eq:correspondence B}
\[
\scalebox{0.9}{
$
   G \ltimes X(F_n, \dotsc, F_0) \xrightarrow{G \ltimes Z(F_n, \dotsc, F_0)} {G \ltimes Z \times_Y Y(F_n, \dotsc, F_0)} \xrightarrow{(Z \times_Y Y(F_n, \dotsc, F_0))(S)} H \ltimes Y(F_n, \dotsc, F_0)
$
}
\]
is isomorphic to the étale correspondence induced by the $S$-equivariant topological correspondence \eqref{labelled correspondence 1}
\[ Z(F_n, \dotsc, F_0)  \colon X(F_n, \dotsc, F_0) \to Y(F_n, \dotsc, F_0). \]
\elemma
\setlength{\parindent}{0cm} \setlength{\parskip}{0cm}

\begin{proof}
The \'etale correspondence \[G \ltimes Z(F_n,\dots,F_0) \colon G \ltimes X(F_n,\dots,F_0) \to G \ltimes Z \times_Y Y(F_n,\dots,F_0)\] is induced by the $S$-equivariant topological correspondence $Z(F_n,\dots,F_0) \colon X(F_n,\dots,F_0) \to Z \times_Y Y(F_n,\dots,F_0)$ by Lemma \ref{lem:G and S equivariant correspondences}. This composes with the $S$-equivariant local homeomorphism $Z \times_Y Y(F_n,\dots,F_0) \to Y(F_n,\dots,F_0)$ to give the claimed $S$-equivariant topological correspondence.
\end{proof}
\setlength{\parindent}{0cm} \setlength{\parskip}{0.5cm}

We shall now give a simpler description of the étale correspondence $G \ltimes X(F_n, \dotsc, F_0) \to H \ltimes Y(F_n,\dotsc,F_0)$ given in \eqref{labelled correspondence 1} up to Morita equivalence. Set $X_0 = \dom_X (e_0)$, $Y_0 = \dom_Y (e_0)$, $Z_0 = \dom_Z (e_0)$, and for $n \geq 1$, set
\begin{align}\label{definition of X_n Y_n Omega_n}
    X_n & \coloneq {}_{G_n^0} G \times_{G_{n-1}} \cdots \times_{G_1} G \times_{G_0} X, \nonumber \\
    Y_n & \coloneq {}_{H_n^0} H \times_{H_{n-1}} \cdots \times_{H_1} H \times_{H_0} Y,\\
    Z_n & \coloneq {}_{G_n^{0}} G \times_{G_{n-1}} \cdots \times_{G_1} G \times_{G_0} Z. \nonumber   
\end{align}
For each $n \geq 0$, $G_n$ acts on the left of $X_n$ and $Z_n$ by left multiplication on the leftmost factor, and $H_n$ acts similarly on the left of $Y_n$. Under the identifications $G_n = F_{n,\down} \ltimes X$ and $H_n = F_{n,\down} \ltimes Y$ this endows $X_n$, $Y_n$ and $Z_n$ with actions of the inverse semigroup $F_{n,\down}$. For $n \geq 1$ set 
\begin{align}\label{definition of rho_n sigma_n}
\rho_n \colon Z_n & \to X_n \nonumber \\
[s_{n-1},\dots,s_0,z] & \mapsto [s_{n-1}, \dots, s_0,\rho(z)]\\
\sigma_n \colon Z_n & \to Y_n \nonumber \\
[s_{n-1},\dots,s_0,z] & \mapsto [s_{n-1}, \dotsc, s_0, \sigma(z)] \nonumber
\end{align}
and set $\rho_0 \colon Z_0 \to X_0$ and $\sigma_0 \colon Z_0 \to Y_0$ to be the restrictions of $\rho$ and $\sigma$. These define for each $n \geq 0$ an $F_{n,\down}$-equivariant topological correspondence $Z_n \colon X_n \to Y_n$.  We obtain an \'etale correspondence 
\begin{equation}\label{labelled correspondence 2}
Z_n(F_{n,\down}) \colon F_{n,\down} \ltimes X_n \to F_{n,\down} \ltimes Y_n.
\end{equation}
By construction, $X(F_n,\dotsc,F_0) = G \times_{G_n} X_n$ and $Y(F_n,\dotsc,F_0) = H \times_{H_n} Y_n$, so we obtain Morita equivalences $G \times_X X_n \colon G \ltimes X(F_n,\dotsc,F_0) \to G_n \ltimes X_n = F_{n,\down} \ltimes X_n$ and $H \times_Y Y_n \colon H \ltimes Y(F_n,\dotsc,F_0) \to H_n \ltimes Y_n = F_{n,\down} \ltimes Y_n$. 

For $n \geq 1$ we set
\begin{align*}
p_n \colon Z_n & \to Z \\
[s_{n-1},\dots,s_0,z] & \mapsto s_{n-1} \dotsm s_0.z
\end{align*}
and we set $p_0 \colon Z_0 \to Z$ to be the inclusion.

\blemma\label{Morita identification lemma}
Let $S$ be an inverse semigroup and let $Z \colon X \to Y$ be a proper $S$-equivariant topological correspondence. Let $n \geq 0$, let $e_0, \dots, e_n \in E(S)$ be idempotents and let $F_i \subseteq S_{e_i}$ be finite subgroups for $0 \leq i \leq n$. Set $G = S \ltimes X$ and $H = S \ltimes Y$.
Then the diagram of étale correspondences below commutes up to isomorphism
\[\begin{tikzcd}[column sep = 7cm]
	{S \ltimes X(F_n,\dotsc,F_0)} & {S \ltimes Y(F_n,\dotsc,F_0)} \\
	{F_{n,\down} \ltimes X_n} & {F_{n,\down} \ltimes Y_n,}
	\arrow["{Z(F_n, \dotsc, F_0)(S)}", from=1-1, to=1-2]
	\arrow["{G \times_X X_n}"', from=1-1, to=2-1]
	\arrow["{H \times_Y Y_n}", from=1-2, to=2-2]
	\arrow["{Z_n(F_{n,\down})}", from=2-1, to=2-2]
\end{tikzcd}\]
where the objects are described in \eqref{definition of X(F_n...F_0) Y(F_n...F_0)} and \eqref{definition of X_n Y_n Omega_n} and the horizontal maps are described at \eqref{labelled correspondence 1} and \eqref{labelled correspondence 2}. The vertical maps are Morita equivalences.
\elemma
\setlength{\parindent}{0cm} \setlength{\parskip}{0cm}

\bproof
We set $G_i = G(F_i)$ and $H_i = H(F_i)$ as in \eqref{definition of X(F_n...F_0) Y(F_n...F_0)}.
Under the identifications $X(F_n,\dots,F_0) = G \times_{G_n} X_n$, $Y(F_n,\dots,F_0) = H \times_{H_n} Y_n$ and $Z(F_n,\dots,F_0) = G \times_{G_n} Z_n$, the maps $\rho^{(n)} \colon Z(F_n,\dots,F_0) \to X(F_n,\dots,F_0)$ and $\sigma^{(n)} \colon Z(F_n,\dots,F_0) \to Y(F_n,\dots,F_0)$ are given in terms of $\rho_n$  and $\sigma_n$ (see \eqref{definition of rho_n sigma_n}) by $\rho^{(n)}([g,z_n]) = [g,\rho_n(z_n)]$ and  $\sigma^{(n)}([g,z_n]) = [g|_{p_n(z_n)},\sigma_n(z_n)]$. Moreover, $p^{(n)} \colon Z(F_n,\dots,F_0) \to Z$ is given by $[g,z_n] \mapsto g . p_n(z_n)$.
\setlength{\parindent}{0cm} \setlength{\parskip}{0.5cm}

The clockwise composition is isomorphic to $Z(F_n,\dots,F_0) \times_{Y(F_n,\dots,F_0)} (H \times_Y H_n)$, with right action given by right multiplication on the $H \times_Y Y_n$ factor, and the left action given by $(g,x^{(n)}) . (z^{(n)}, (h,y_n)) = (g . z^{(n)}, (g|_{p^{(n)}(z^{(n)})}h,y_n))$. Under the above identifications, the bispace becomes $(G \times_{G_n} Z_n) \times_{H \times_{H_n} Y_n} (H \times_Y Y_n)$, with the left action of $G \ltimes (G \times_{G_n} X_n)$ given by $(g,[g',x_n]) . ([g',z_n],(h,y_n)) = ([gg',z_n],(g|_{g'.p_n(z_n)}h,y_n))$.

The anticlockwise composition is  $(G \times_X X_n) \times_{F_{n,\down} \ltimes X_n} Z_n(F_{n,\down})$. Consider the map 
\begin{align*}
(G \times_X X_n) \times_{F_{n,\down} \ltimes X_n} Z_n(F_{n,\down}) & \mapsto (G \times_{G_n} Z_n) \times_{H \times_{H_n} Y_n} (H \times_Y Y_n) \\
[(g,x_n),(z_n,h_n)] & \mapsto ([g,z_n], (g|_{p_n(z_n)} h_n ,\sigma_n(z_n . h_n))).
\end{align*}
This has inverse 
\begin{align*}
(G \times_{G_n} Z_n) \times_{H \times_{H_n} Y_n} (H \times_Y Y_n) & \mapsto (G \times_X X_n) \times_{F_{n,\down} \ltimes X_n} Z_n(F_{n,\down})  \\
([g,z_n],(h,y_n)) & \mapsto [(g,\rho_n(z_n)),(z_n, g|_{p_n(z_n)}^{-1} h)].
\end{align*}
We note that $g|_{p_n(z_n)}^{-1} h$ is indeed an element of $H_n = F_{n,\down} \ltimes Y$ because via the fibre product structure we assume the equality $[g|_{p_n(z_n)},\sigma_n(z_n)] = [h,y_n] \in H \times_{H_n} Y_n$.
It is routine to check that the above map describes an isomorphism of bispaces, completing the proof of the lemma.
\eproof
\setlength{\parindent}{0cm} \setlength{\parskip}{0.5cm}

Let us summarize our findings.

\bcor
\label{cor:Ln=Iso?}
Let $S$ be a countable inverse semigroup with idempotent semilattice $E$ and let $Z \colon X \to Y$ be a proper $S$-equivariant topological correspondence with $X$, $Y$ and $Z$ second countable and $G = S \ltimes X$, $H = S \ltimes Y$ Hausdorff. Suppose that for any choice of non-zero idempotents $e_0,\dots,e_n \in E$ and finite subgroups $F_i \subseteq S_{e_i}$, the \'etale correspondence 
\begin{equation}
Z_n(F_{n,\down}) \colon F_{n,\down} \ltimes X_n \to F_{n,\down} \ltimes Y_n
\end{equation}
described in \eqref{labelled correspondence 2} induces a $\K_*$-isomorphism. Then $\Ktop_*(Z(S)) \colon \Ktop_*(G) \to \Ktop_*(H)$ is an isomorphism.
\ecor
\setlength{\parindent}{0cm} \setlength{\parskip}{0cm}

\begin{proof}
We apply Proposition \ref{isomorphism proposition before concrete models}, so consider the projective resolutions $P_\bullet \to C_0(X)$, $Q_\bullet \to C_0(Y)$ as in \eqref{eq:projective resolution}. Lemma \ref{we actually get a chain map} tells us that after taking a direct sum over all $(F_n,\dots,F_0)$ at each $n \geq 0$, the $G$-equivariant topological correspondences $Z(F_n,\dots,F_0) \colon X(F_n,\dots,F_0) \to Z \times_Y Y(F_n,\dots,F_0)$ described in \eqref{eq: chain map correspondences} induce a chain map $\tilde f \colon P_\bullet \to \Ind_{Z(S)} Q_\bullet$ over $f_{Z(S)} \colon C_0(X) \to \Ind_{Z(S)} C_0(Y)$. By Proposition \ref{crossed product correspondence identification}, $G \ltimes \tilde f \colon G \ltimes P_n \to G \ltimes \Ind_{Z(S)} Q_n$ is modelled at the groupoid level by the \'etale correspondences $G \ltimes Z(F_n,\dots,F_0) \colon G \ltimes X(F_n,\dots,F_0) \to G \ltimes (Z \times_Y Y(F_n,\dots,F_0))$. Moreover, the induction correspondence ${Z(S)} \ltimes Q_n \colon G \ltimes \Ind_{Z(S)} Q_n \to H \ltimes Q_n$ is modelled by the $S$-equivariant local homeomorphism $[z,y] \mapsto y \colon Z \times_Y Y(F_n,\dotsc,F_0) \to Y(F_n,\dotsc,F_0)$ described in \eqref{crossed product model}. By Lemma \ref{chain map composed model} the composition $G \ltimes P_n \to H \ltimes Q_n$ is modelled by the $S$-equivariant topological correspondence $Z(F_n, \dotsc, F_0) \colon X(F_n, \dotsc, F_0) \to Y(F_n, \dotsc, F_0)$ described in \eqref{labelled correspondence 1}. Finally, by Lemma \ref{Morita identification lemma} the resulting map in K-theory \eqref{chain map crucial} $\K_*(G \ltimes P_n) \to \K_*(H \ltimes Q_n)$ is isomorphic to the map in K-theory induced by the $F_{n,\down}$-equivariant topological correspondence $Z_n \colon X_n \to Y_n$ described in \eqref{labelled correspondence 2}. If this is an isomorphism for each $(F_0,\dots,F_n)$, Proposition \ref{isomorphism proposition before concrete models} implies that $\Ktop_*({Z(S)}) \colon \Ktop_*(G) \to \Ktop_*(H)$ is an isomorphism.
\end{proof}
\setlength{\parindent}{0cm} \setlength{\parskip}{0.5cm}

\subsection{Isomorphism at the level of ABC spectral sequences}\label{isomorphism subsection}

Let us for the rest of this subsection fix a countable locally finite weak semilattice $L$ with an action of a countable inverse semigroup $S$ with idempotent semilattice $E$ such that $S \ltimes \widehat L$ is Hausdorff. Let $G = S \ltimes L\reg$ and $H = S \ltimes \widehat L$ with $X = L\reg$ and $Y = \widehat L$, and let $\Omega = \Omega^0(S) \colon G \to H$ be the discretisation correspondence. Let $e_0, \dots, e_n \in E$ be idempotents and let $F_i \subseteq S_{e_i}$ be finite subgroups. In this setting, we constructed an $F_{n,\down}$-equivariant topological correspondence (see \eqref{definition of X_n Y_n Omega_n} and \eqref{labelled correspondence 2}) 
\[ Z_n \colon X_n \to Y_n. \]
Our goal now is to prove the following.
\btheo
\label{thm:Omega=KIso}
The proper étale correspondence $Z_n (F_{n,\down}) \colon F_{n,\down} \ltimes X_n \to F_{n,\down} \ltimes Y_n$ induces a $\K_*$-isomorphism. 
\etheo

For the proof, let us first analyse the space
\begin{equation*}
    \Y_n \coloneq { }_{H_n^{0}} H \times_{H_{n-1}^{0}} \dotso \times_{H_1^{0}} H \times_{H_0^{0}} H^{0}.
\end{equation*}
Consider 
\begin{equation*}
    \crL_n \coloneq \menge{(s_{n-1}, \dotsc, s_0, \widehat{L}(k)) \subseteq \Y_n}{ s_m \in e_{m+1} S e_m \text{ for $0 \leq m \leq n-1$ },  k \in \dom_{L} (s_{n-1} \cdots s_0)},
\end{equation*}
where
\begin{equation*}
    (s_{n-1}, \dotsc, s_0, \widehat{L}(k)) \coloneq \menge{[s_{n-1}, \dotsc, s_0, \chi] \in \Y_n}{\chi \in \widehat{L}(k)}.
\end{equation*}
We note that by construction each member $(s_{n-1}, \dotsc, s_0, \widehat{L}(k)) \in \crL_n$ is compact and open, and homeomorphic to its projection $\widehat{L}(k)$ onto $H^0$. 

\blemma
\label{lem:crLLFWS}
The family $\crL_n$ is a generating independent locally finite weak semilattice in $\Y_n$ in the sense of Definition \ref{def:GenIndLFWS}.

In particular, we have an isomorphism $\cs(\crL_n) \cong C_0(\Y_n), \, e \ma 1_e$, which induces the identification $\Y_n \cong \widehat{\crL_n}$.
\elemma
\setlength{\parindent}{0cm} \setlength{\parskip}{0cm}

\bproof
To see that $\crL_n$ is a weak semilattice, take $e_1, e_2 \in \crL_n$. Then $e_1 \cap e_2$ is a compact subset of $\Y_n$ (this uses Hausdorffness). Moreover, it is easy to see that every point of $\Y_n$ has a neighbourhood given by an element of $\crL_n$ which is completely contained in $e_1 \cap e_2$. It follows that $e_1 \cap e_2$ can be written as a finite union $e_1 \cap e_2 = \bigcup_i f_i$ for some $f_i \in \crL_n$. Now take $d \in \crL_n$ with $d = (s_{n-1}, \dotsc, s_0, \widehat{L}(k))$ such that $d \subseteq e_1 \cap e_2 = \bigcup_i f_i$. As $[s_{n-1}, \dotsc, s_0, \chi_k]$ lies in $d$, this element has to lie in some $f_i$. Suppose that $f_i = (t_{n-1}, \dotsc, t_0, \widehat{L}(l))$. Then $\chi_k(l) = 1$, so that $k \leq l$ and $\widehat{L}(k) \subseteq \widehat{L}(l)$. Moreover, $[s_{n-1}, \dotsc, s_0, \chi_k] = [t_{n-1}, \dotsc, t_0, \chi_k]$ implies that there are idempotents $e_m \in E^\times$ with $s_m e_m = t_m e_m$ and $s_{m-1} \dotsm s_0.k = t_{m-1} \dotsm t_0.k \in \dom_{L^\times} (e_m)$ for all $0 \leq m \leq n-1$. It follows that $[s_{n-1}, \dotsc, s_0, \chi] = [t_{n-1}, \dotsc, t_0, \chi]$ for all $\chi \in \widehat{L}(k)$, i.e., $d \subseteq f_i$. 
\setlength{\parindent}{0cm} \setlength{\parskip}{0.5cm}

To see that $\crL_n$ is locally finite, take a finite subset $F \subseteq \crL_n$. Proceeding inductively on $\# F$, we want to show that there exists a $\down$-closed finite subset $\widebar{F}$ of $\crL_n$ containing $F$. So suppose we have found such a $\widebar{F}$ for $F$. Now take an additional element $f \in \crL_n$, say $f = (s_{n-1}, \dotsc, s_0, \widehat{L}(k))$. The down-set $\{ g \in \crL_n \mid g \leq f \}$ carries the weak semilattice structure of $\{ l \in L \mid l \leq k \}$. In particular, since $L$ is locally finite, we may a finite subset $A \subseteq \{ g \in \crL_n \mid g \leq f \}$ which is $\down$-closed and contains the finite set $\{ f \} \cup (\widebar{F} \down f)$. Now we claim that $\widebar{F} \cup A$ is $\down$-closed. Indeed, given $e \in \widebar{F}$ and $g \in A$, we have $e \down g = e \down (f \down g) = (e \down f) \down g$, which lies in $A$ because $e \down f \subseteq A$ and $A$ is $\down$-closed.

The family $\crL_n$ satisfies (GEN) because for each fixed $(s_{n-1},\dots,s_0,\widehat L(k))$, the family \[\left\{(s_{n-1},\dots,s_0,\widehat L(l)) \suchthat l \leq k \right\}\] satisfies (GEN) for $(s_{n-1},\dots,s_0,\widehat L(k))$. Moreover, $\crL_n$ satisfies (IND). Indeed, given elements $e,e_1,\dots,e_m \in \crL_n$ with $e = \bigcup_{i=1}^m e_i$ where $e = (s_{n-1},\dots,s_0,\widehat L(k))$, then there is some $1 \leq i \leq m$ with $[s_{n-1},\dots,s_0,\chi_k] \in e_i \subseteq e$, from which it follows that $e_i = e$.
\eproof
\setlength{\parindent}{0cm} \setlength{\parskip}{0.5cm}

Given $(s_{n-1}, \dotsc, s_0, \widehat{L}(k)) \in \crL_n$, we let $[s_{n-1}, \dotsc, s_0, \widehat{L}(k)]$ denote its image in $Y_n$ with respect to the canonical projection $\Y_n \onto Y_n$. We set 
\begin{equation*}
    L_n \coloneq \menge{[s_{n-1}, \dotsc, s_0, \widehat{L}(k)] \subseteq Y_n}{(s_{n-1}, \dotsc, s_0, \widehat{L}(k)) \in \crL_n}.
\end{equation*}
We identify $X_n$ with $L_n\reg$ via $[s_{n-1},\dots,s_0,k] \mapsto [s_{n-1}, \dotsc, s_0, \widehat{L}(k)] \colon X_n \to L_n\reg$ for elements $s_0,\dots,s_{n-1} \in S$ with $s_m \in e_{m+1} S e_m$ for $0 \leq m \leq n-1$, and $k \in \dom_{L\reg} (s_{n-1} \cdots s_0)$.
The following is a consequence of Lemma~\ref{lem:crLLFWS}.
\blemma
The family $L_n$ is an $F_{n,\down}$-invariant generating independent locally finite weak semilattice in $Y_n$ in the sense of Definition \ref{def:GenIndLFWS}. In particular, we have an isomorphism $\cs(L_n) \cong C_0(Y_n), \, e \ma 1_e$, which induces the identification $Y_n \cong \widehat{L_n}$.
\elemma
\setlength{\parindent}{0cm} \setlength{\parskip}{0cm}

\begin{proof}
We view $Y_n$ as the orbit space of $\Y_n$ under the action of the finite group $F = F_{n-1} \times \cdots \times F_0$. Explicitly, for $f = (f_{n-1},\dots,f_0) \in F$ and $y = [s_{n-1},\dots,s_0,\chi] \in \Y_n$, we write $y^f = [s_{n-1} f_{n-1},f_{n-1}^{-1} s_{n-2} f_{n-2},\dots,f_1^{-1} s_0 f_0, f_0^{-1} . \chi]$ for the (right) action of $f$ on $y$. This induces an action of $F$ on $\crL_n$, which we will denote by $l^f$ for $l \in \crL_n$ and $f \in F$. Then, for $l_1,l_2 \in \crL_n$, we compute $[l_1] \cap [l_2] = [ \bigcup_{f \in F} l_1^f \cap l_2 ] = \bigcup_{f \in F} [l_1^f \cap l_2] = \bigcup_{f \in F, k \in l_1^f \down l_2}[k]$, giving us a finite union. 
\setlength{\parindent}{0cm} \setlength{\parskip}{0.5cm}

Suppose that $[d] \subseteq [l]$ for $d,l \in \crL_n$. The principal filter in $\widehat \crL_n \cong \Y_n$ associated to $d$ lies in $l^f$ for some $f \in F$, from which it follows that $d \subseteq l^f$. Therefore, if $[d] \subseteq [l_1] \cap [l_2]$, then there are $f_1,f_2 \in F$ with $d \subseteq l_1^{f_1} \cap l_2^{f_2}$, and hence $d^{f_2^{-1}} \subseteq l_1^{f_1 f_2^{-1}} \cap l_2$. Thus we can find $k \in l_1^{f_1 f_2^{-1}} \down l_2$ with $d^{f_2^{-1}} \subseteq k$, ergo $[d] \subseteq [k]$.

It follows that $L_n$ taken with the inclusion order is a weak semilattice. It is moreover locally finite, because if $A \subseteq L_n$ is a finite subset, the preimage $B = \{ b \in \crL_n \mid [b] \in A \}$ is finite, and its closure $\widebar B$ under $\down$ remains finite. Then $\widebar A = \{ [b] \mid b \in \widebar B \}$ is finite, contains $A$ and is closed under $\down$. 

To see that $L_n$ satisfies (IND), suppose $d,d_1,\dots,d_m \in \crL_n$ with $[d] = \bigcup_{i=1}^m [d_i]$. Then the principal filter in $\widehat \crL_n \cong \Y_n$ associated to $d$ lies in $d_i^f$ for some $i$ and some $f \in F$. It follows that $[d] \subseteq [d_i^f] = [d_i] \subseteq [d]$ and hence $[d]=[d_i]$.

The compact open subsets of $Y_n$ correspond directly to $F$-invariant compact open subsets of $\Y_n$. Finite unions and relative complements of such compact open subsets in $\Y_n$ become finite unions and relative complements in $Y_n$. It follows that $L_n$ inherits (GEN) from $\crL_n$.
\end{proof}
\setlength{\parindent}{0cm} \setlength{\parskip}{0.5cm}

\blemma\label{identification with discretisation}
The $F_{n,\down}$-equivariant topological correspondence $Z_n \colon X_n \to Y_n$ is isomorphic to the discretisation correspondence $\Omega_n \colon L_n\reg \to \widehat{L_n}$.
\elemma
\setlength{\parindent}{0cm} \setlength{\parskip}{0cm}

\begin{proof}
We have already identified $Y_n$ with $\widehat{L_n}$ and $X_n$ with $L_n\reg$. Recall that $\Omega^0 = \coprod_{k \in L\reg} \widehat L(k)$ and
\begin{align*}
Z_n & = {}_{G_n^{0}} G \times_{G_{n-1}} \cdots \times_{G_1} G \times_{G_0} \Omega^0 \\
 & = \left\{[s_{n-1},\dots,s_0,k,\chi] \suchthat \begin{array}{l} s_m \in e_{m+1} S e_m \text{ for $0 \leq m \leq n-1$,} \\ k \in \dom_{L\reg} (s_{n-1}\cdots s_0), \, \chi \in \widehat L(k) \end{array} \right\},
\end{align*}
and that $\Omega_n = \coprod_{e \in L_n \reg} \widehat{L_n}(e)$. The desired homeomorphism is given by 
\begin{align*}
Z_n & \to \Omega_n \\
[s_{n-1},\dots,s_0,k,\chi] & \mapsto ([s_{n-1},\dots,s_0,\widehat{L}(k)], [s_{n-1},\dots,s_0,\chi]). \qedhere
\end{align*}
\end{proof}
\setlength{\parindent}{0cm} \setlength{\parskip}{0.5cm}

\btheo\label{finite ivs discretisation}
Let $L$ be a locally finite weak semilattice equipped with the action of a finite inverse semigroup $F$. The discretisation correspondence $\Omega \colon F \ltimes L\reg \to F \ltimes \widehat L$ induces an isomorphism $\K_*(\cs(F \ltimes L\reg)) \isom \K_*(\cs(F \ltimes \widehat L))$.
\etheo
\setlength{\parindent}{0cm} \setlength{\parskip}{0cm}

\bproof
Since $F$ is a finite inverse semigroup, we can find a directed family of finite subsets $L_i \subseteq L$ containing $0$ which are $F$-invariant and $\down$-closed. 
Then, via Lemma \ref{lem:C*L}, we obtain $F$-equivariant inductive limit descriptions
\begin{align}
\label{e:C0X=lim}
    C_0(L\reg) & \cong \ilim_i C(L_i\reg) \\
\label{e:C0Y=lim}
    C_0(\widehat L) & \cong \ilim_i C(\widehat{L_i}).
\end{align}
Set $Y_i \coloneq \bigcup_{e \in L_i\reg} \widehat L(e) \subseteq \widehat L$. We have a canonical projection $\pi_i \colon Y_i \onto \widehat{L_i}$ given by $\pi_i(\chi)(e) = 1$ if and only if $\chi \in e$. Together with the open inclusion $Y_i \hookrightarrow \widehat L$, this defines an $F$-invariant topological correspondence $Y_i \colon \widehat{L_i} \to \widehat L$. The only thing to check is that the $F$-invariant map $\pi_i \colon Y_i \onto \widehat{L_i}$ satisfies $\pi_i^{-1}(\dom_{\widehat{L_i}} (s)) \subseteq \dom_{Y_i} (s)$ for each $s \in F$. Let $s \in F$ and $y \in \pi_i^{-1}(\dom_{\widehat{L_i}} (s))$. Then $\pi_i(y)(e) = 1$ for some $e \in \dom_{L_i} (s)$. Viewing $e \in L_i$ now as an element of $L$, this tells us that $y(e) = 1$. From $e \in \dom_{L} (s)$ it follows that $y \in \dom_{\widehat L} (s)$ and ultimately $y \in \dom_{Y_i} (s)$. Similarly, when $i \leq j$, we have an $F$-invariant topological correspondence $Y_{i,j} \coloneq \bigcup_{e \in L_i\reg} \widehat{L_j}(e) \colon \widehat{L_i} \to \widehat{L_j}$.
\setlength{\parindent}{0cm} \setlength{\parskip}{0.5cm}

Consider the discretisation correspondences $\Omega^0 = \coprod_{e \in L\reg} \widehat L(e) \colon L\reg \to \widehat L$ and, for each $i$, $\Omega^0_i = \coprod_{e \in L_i\reg} \widehat{L_i}(e) \colon L_i\reg \to \widehat{L_i}$. We claim that, for $i \leq j$, the following diagrams of $F$-equivariant topological correspondences commute.
\begin{equation}\label{inductive limit compatibility}
    \begin{tikzcd}
     L_i\reg \arrow[r, "\Omega^0_i"] \arrow[d, hook] & \widehat{L_i} \arrow[d, "Y_i"] & L_i\reg \arrow[r, "\Omega^0_i"] \arrow[d, hook] & \widehat{L_i} \arrow[d, "Y_{i,j}"]   \\
     L\reg \arrow[r, "\Omega^0"] & \widehat{L} & L_j\reg \arrow[r, "\Omega^0_j"] & \widehat{L_j}
    \end{tikzcd}
\end{equation}
We consider the first diagram. The composition along the left and bottom is given by the topological correspondence $\coprod_{e \in L_i\reg} \widehat{L}(e)$, while along the top and right we have $\coprod_{e \in L_i\reg} (\widehat{L_i}(e) \times_{\widehat{L_i}} Y_i)$. These are isomorphic because for each $e \in L_i\reg$, the space $\widehat{L_i}(e) \times_{\widehat{L_i}} Y_i$ is homeomorphic via the second factor projection to the subspace $\pi_i^{-1}(\widehat{L_i}(e)) = \widehat{L}(e)$ of $\widehat{L}$.

By continuity of $\K_*(F \ltimes -)$ on $F$-equivariant $*$-homomorphisms, we have inductive limit descriptions 
\begin{align*}
\K_*(F \ltimes C_0(L\reg)) & \cong \ilim_i \K_*(F \ltimes C(L_i\reg)) \\
\K_*(F \ltimes C_0(\widehat L)) & \cong \ilim_i \K_*(F \ltimes C(\widehat{L_i})).
\end{align*}
Combining this with \eqref{inductive limit compatibility} it suffices to prove that $\Omega^0_i \colon L_i\reg \to \widehat{L_i}$ induces an isomorphism $\K_*(\cs(F \ltimes L_i\reg)) \to \K_*(\cs(F \ltimes \widehat{L_i}))$ for each $i$. Setting
\begin{align*}
    \Omega_i^= & \coloneq \coprod_{e \in L_i\reg} \gekl{\chi_e} \\
    \Omega_i^< & \coloneq \coprod_{e \in L_i\reg} \menge{\chi_d}{d \lneq e},    
\end{align*}
we have the decomposition
\begin{equation*}
    \Omega^0_i = \Omega_i^= \amalg \Omega_i^<
\end{equation*}
of $F$-equivariant $L_i\reg$-$\widehat{L_i}$-correspondences.
It is now straightforward to see that, under the isomorphism $L_i\reg \cong \widehat{L_i}$ induced by the $F$-equivariant homeomorphism $L_i\reg \cong \widehat{L_i}, \, e \ma \chi_e$, $\Omega_i^=$ becomes isomorphic to the identity correspondence and $\Omega_i^<$ becomes a nilpotent correspondence (in the sense that a sufficiently high power of it becomes trivial). Hence $\Omega_i^0 \colon L_i\reg \to  \widehat{L_i}$ induces an isomorphism $\K_*(\cs(F \ltimes L_i\reg)) \to  \K_*(\cs(F \ltimes \widehat{L_i}))$, as required.
\eproof
\setlength{\parindent}{0cm} \setlength{\parskip}{0.5cm}

Theorem \ref{thm:Omega=KIso} follows immediately from Lemma \ref{identification with discretisation} and Theorem \ref{finite ivs discretisation}. We are now ready for the proof of Theorem \ref{thm:Disc_K}.

\bproof[Proof of Theorem \ref{thm:Disc_K}]
Corollary~\ref{cor:Ln=Iso?} and Theorem~\ref{thm:Omega=KIso} together imply that the discretisation correspondence $\Omega \colon S \ltimes L\reg \to S \ltimes \widehat L$ induces an isomorphism $\Ktop_*(\Omega) \colon \Ktop_*(S \ltimes L\reg) \to \Ktop_*(S \ltimes \widehat L)$. Explicitly, $\Ktop_*(\Omega)$ is given by $\mathbb L(\Ind_\Omega,\alpha_\Omega,f_\Omega) \colon \Lz^{\cF_{\rm fin}^X}\K_*(G \ltimes C_0(X)) \to \Lz^{\cF_{\rm fin}^Y}\K_*(H \ltimes C_0(Y))$ under the identification between these localisations and topological K-theory, where $X = L\reg$, $Y = \widehat L$ and $G = S \ltimes X$, $H = S \ltimes Y$. As discussed in \cite[Example 5.11]{Mil3}, there is a commutative diagram
\begin{equation}
\label{e:CD_LL_Disc_BC}
\begin{tikzcd}
	{\Ktop_*(S \ltimes L\reg)} & {\K_*(\cs(S \ltimes L\reg))} & {\K_*(\csr(S \ltimes L\reg))} \\
	{\Ktop_*(S \ltimes \widehat L)} & {\K_*(\cs(S \ltimes \widehat L))} & {\K_*(\csr(S \ltimes \widehat L))},
	\arrow[from=1-1, to=1-2]
	\arrow["{\mu_{S \ltimes L\reg}}", curve={height=-18pt}, from=1-1, to=1-3]
	\arrow["{\Ktop_*(\Omega)}", from=1-1, to=2-1]
	\arrow[from=1-2, to=1-3]
	\arrow["{\K_*(\cs(\Omega))}", from=1-2, to=2-2]
	\arrow["{\K_*(\csr(\Omega))}", from=1-3, to=2-3]
	\arrow[from=2-1, to=2-2]
	\arrow["{\mu_{S \ltimes \widehat L}}"', curve={height=18pt}, from=2-1, to=2-3]
	\arrow[from=2-2, to=2-3]
\end{tikzcd}
\end{equation}
where $\mu_{S \ltimes L\reg}$ and $\mu_{S \ltimes \widehat L}$ are the Baum--Connes assembly maps. Thus, under the Baum--Connes conjecture, $\Ktop_*(\Omega)$ is identified with $\K_*(\csr(\Omega))$.
\eproof

\bremark
\label{rem:BC=Disc}
Let $L$ be a countable locally finite weak semilattice equipped with the action of a countable inverse semigroup $S$ and suppose that $S \ltimes \widehat L$ is Hausdorff. Assume that, for every unit $l \in L\reg = (S \ltimes L\reg)^{(0)}$, the isotropy group $(S \ltimes L\reg)_l^l$ satisfies the Baum--Connes conjecture. Then $S \ltimes \widehat L$ satisfies the Baum--Connes conjecture if and only if the discretisation correspondence $\Omega$ induces an isomorphism
\[ \K_*(\csr(\Omega)) \colon \K_*(\csr(S \ltimes L\reg) ) \isom \K_*(\csr(S \ltimes \widehat L)).\]
Indeed, our assumptions imply that $S \ltimes L\reg$ satisfies the Baum--Connes conjecture by Remark~\ref{rem:BC_discGpd}. Hence the implication \an{$\Rarr$} follows from Theorem~\ref{thm:Disc_K}. For \an{$\Larr$}, note that, in the diagram \eqref{e:CD_LL_Disc_BC}, the map $\mu_{S \ltimes L\reg}$ is an isomorphism by our assumptions, and that $\Ktop_*(\Omega)$ is an isomorphism by Theorem~\ref{thm:Disc_K}. Hence, as \eqref{e:CD_LL_Disc_BC} commutes, $\mu_{S \ltimes \widehat L}$ is an isomorphism if and only if $\K_*(\csr(\Omega))$ is an isomorphism, as claimed.
\eremark

\section{Applications and examples}
\label{s:AppEx}

Let us now address the question how independent resolutions help to compute or understand groupoid homology and K-theory for reduced groupoid $\cs$-algebras. We then apply our results to concrete examples.

\subsection{Computing groupoid homology using independent resolutions}

Suppose that $G$ is an ample groupoid, and $G_i$ ($i=-1, 0, 1, \dotsc$) is an independent resolution of $G$ in the sense of \S~\ref{ss:IndResDef}. In the following, we use the same notation as in \S~\ref{ss:IndResDef}. The decompositions $G_i^{(0)} = C_i \amalg O_i$, together with the isomorphisms $G_i \vert_{C_i} \cong G_{i-1} \vert_{O_{i-1}}$, induce long exact sequences in groupoid homology:
\begin{equation}
\label{e:LongExSeq_H}
 \dotso \to \Hlgy_{*+1}(G_{i-1} \vert_{O_{i-1}}) \to \Hlgy_*(G_i \vert_{O_i}) \to \Hlgy_*(G_i) \to \Hlgy_*(G_{i-1} \vert_{O_{i-1}}) \to \Hlgy_{*-1}(G_i \vert_{O_i}) \to \dotso
\end{equation}

For non-negative integers $p, q$, set $D_{pq} \coloneq \Hlgy_q(G_p \vert_{O_p})$ and $E_{pq} \coloneq \Hlgy_q(G_p)$. Then we obtain an exact couple
\[
 \xymatrix{
 D \ar[rr]^i & & D \ar[dl]^j \\
 & E \ar[ul]^k &   
 }
\]
Here $i \colon D_{pq} = \Hlgy_q(G_p \vert_{O_p}) \to \Hlgy_{q-1}(G_{p+1} \vert_{O_{p+1}}) = D_{p+1,q-1}$ is the boundary map from the long exact sequence \eqref{e:LongExSeq_H}, $j \colon D_{pq} = \Hlgy_q(G_p \vert_{O_p}) \to \Hlgy_q(G_p) = E_{pq}$ is induced by the canonical inclusion $G_p \vert_{O_p} \into G_p$, and $k \colon E_{pq} = \Hlgy_q(G_p) \to \Hlgy_q(G_{p-1} \vert_{O_{p-1}}) = D_{p-1,q}$ is given by the composition
\[
 \Hlgy_q(G_p) \to \Hlgy_q(G_p \vert_{C_p}) \to \Hlgy_q(G_{p-1} \vert_{O_{p-1}}),
\]
where the first map is induced by restriction, and the second map is induced by the isomorphism $G_p \vert_{C_p} \cong G_{p-1} \vert_{O_{p-1}}$ which is part of the data of an independent resolution (see \S~\ref{ss:IndResDef}).

This exact couple gives rise to the spectral sequence $E^r_{pq}$, where $E^0_{pq} = E_{pq} = \Hlgy_q(G_p)$ and $E_{pq}^1$ is given by the $p$-th homology of the chain complex
\[
 \dotso \to \Hlgy_q(G_2) \to \Hlgy_q(G_1) \to \Hlgy_q(G_0) \to \Hlgy_q(G) \to 0,
\]
where the chain maps are given by $j \circ k$. The construction of the spectral sequence can be found for instance in \cite[\S~5.9]{Wei} (see in particular \cite[Application~5.9.8]{Wei}).

Since groupoid homology vanishes in negative degrees, this spectral sequence converges to $0$, i.e., for all $p, q$ there exists $r_{pq}$ such that $E_{pq}^r = 0$ for all $r \geq r_{pq}$. Then the spectral sequence determines $\Hlgy_*(G) = \Hlgy_*(G_{-1})$ up to extension problems, assuming that we understand $\Hlgy_*(G_p)$ and the maps $\Hlgy_*(G_{p+1}) \to \Hlgy_*(G_p)$ for all $p \geq 0$.

Here is an example application.
\bcor
\label{cor:IndResGPDH_0}
Suppose that we are given an independent resolution for $G = G_{-1}$, consisting of groupoid $G_i$ which satisfy $\Hlgy_*(G_i) = 0$ for all $*>0$ and $i \geq 0$. Then $\Hlgy_q(G)$ is given by the $q$-th homology of the chain complex
\[
 \dotso \to \Hlgy_0(G_{p+1}) \to \Hlgy_0(G_p) \to \Hlgy_0(G_{p-1}) \to \dotso \to \Hlgy_0(G_1) \to \Hlgy_0(G_0) \to 0.
\]
\ecor
\setlength{\parindent}{0cm} \setlength{\parskip}{0cm}

\bproof
This follows immediately from an analysis of the spectral sequence $E_{pq}^r$, using that many terms vanish and that it converges to $0$.
\eproof
\setlength{\parindent}{0cm} \setlength{\parskip}{0.5cm}

\subsection{Applications of independent resolutions in K-theory}
\label{ss:AppIndRes_K}

Let us now turn to $\cs$-algebras and K-theory. Consider the following general setting: Let 
\begin{equation}
\label{e:AAA}
 \dotso \to A_2 \to A_1 \to A_0 \to A \to 0 \to 0 \to \dotso
\end{equation}
be a long exact sequence of $\cs$-algebras. Let us split it into short exact sequences, i.e., consider the commutative diagram
\[
\scalebox{0.8}{
 \xymatrix{
 \dotso \ar[rr] & & A_2 \ar[dr] \ar[rr] & & A_1 \ar[dr] \ar[rr] & & A_0 \ar[dr] \ar[rr] & & A \ar[dr] \ar[rr] & & 0 \ar[dr] \ar[rr] & & \dotso\\
 &  \dotso \ar[ur] & & I_1 \ar[ur] & & I_0 \ar[ur] & & A \ar[ur] & & 0 \ar[ur] & & \dotso &   
 }
}
\]
where we have short exact sequences 
\begin{equation}
\label{e:ShortExSeq}
0 \to I_p \to A_p \to I_{p-1} \to 0,
\end{equation}
and we set $I_{-1} \coloneq A$, $A_{-1} \coloneq A$, $I_p \coloneq 0$ and $A_p \coloneq 0$ for all $p < -1$. 

Applying K-theory produces a long exact sequence
\begin{equation}
\label{e:LongExSeq}    
 \dotso \to \K_{*+1}(I_{p-1}) \to \K_*(I_p) \to \K_*(A_p) \to \K_*(I_{p-1}) \to \K_{*-1}(I_p) \to \dotso
\end{equation}
for each short exact sequence $0 \to I_p \to A_p \to I_{p-1} \to 0$.

Set $D_{pq} \coloneq \K_q(I_p)$ and $E_{pq} \coloneq \K_q(A_p)$. Then we obtain an exact couple
\[
 \xymatrix{
 D \ar[rr]^i & & D \ar[dl]^j \\
 & E \ar[ul]^k &   
 }
\]
Here $i \colon D_{pq} = \K_q(I_p) \to \K_{q-1}(I_{p+1}) = D_{p+1,q-1}$ is the boundary map from the long exact sequence \eqref{e:LongExSeq}, $j \colon D_{pq} = \K_q(I_p) \to \K_q(A_p) = E_{pq}$ is induced by the map $I_p \to A_p$ from the short exact sequence \eqref{e:ShortExSeq}, and $k \colon E_{pq} = \K_q(A_p) \to \K_q(I_{p-1}) = D_{p-1,q}$ is induced by the map $A_p \to I_{p-1}$ from the short exact sequence \eqref{e:ShortExSeq}. 

This exact couple gives rise to the spectral sequence $E^r_{pq}$, where $E^0_{pq} = E_{pq} = \K_q(A_p)$ and $E_{pq}^1$ is given by the $p$-th homology of the chain complex
\[
 \dotso \to \K_q(A_2) \to \K_q(A_1) \to \K_q(A_0) \to \K_q(A) \to 0
\]
The construction of the spectral sequence can be found for instance in \cite[\S~5.9]{Wei} (see in particular \cite[Application~5.9.8]{Wei}).

We are interested in the question when this spectral sequence converges to $0$, i.e., for all $p, q$ there exists $r_{pq}$ such that $E_{pq}^r = 0$ for all $r \geq r_{pq}$. Then the spectral sequence determines $\K_*(A)$ up to extension problems, assuming that we understand $\K_*(A_p)$ and the maps $\K_*(A_{p+1}) \to \K_*(A_p)$ for all $p \geq 1$. It follows from \cite[Example~5.9.1]{Wei} that $E_{pq}^r$ converges to $0$ if and only if the map $\K_q(I_p) \to \K_{q-1}(I_{p+1}) \to \dotso \to \K_{q-r+1}(I_{p+r-1}) \to \K_{q-r}(I_{p+r})$ vanishes for sufficiently big $r$ (where \an{sufficiently big} may depend on $p$ and $q$). For example, convergence to $0$ is guaranteed if our initial long exact sequence \eqref{e:AAA} is of finite length, i.e., $A_k = 0$ for sufficiently big $k$. In that case we have $I_p = 0$ for sufficiently big $p$.

Assume now that we are given an independent resolution $G_i = S_i \ltimes \widehat{L_i}$ of an ample groupoid $G$. If our groupoids $G_i$ are exact, then the independent resolution yields a long exact sequence of reduced groupoid $\cs$-algebras, and the above discussion applies to $A_i = \csr(G_i)$. If furthermore all $S_i$ and $L_i$ are countable, all $G_i$ are Hausdorff, and $G_i$ and $S_i \ltimes L\reg_i$ satisfy the Baum--Connes conjecture, then Corollary~\ref{cor:OSF_K} allows to compute K-theory of $\csr(G_i)$ in terms of K-groups of isotropy groups, which are typically easier to handle. Here is a first direct consequence.

\bcor
\label{cor:FLIR->BC}
Assume that we are given a finite length independent resolution $G_i = S_i \ltimes \widehat{L_i}$ of an ample groupoid $G$, and suppose that for every $i$, the groupoids $G_i$ are Hausdorff and exact, and $S_i$ as well as $L_i$ are countable. Further assume that for every $i$ and every unit $l \in L_i\reg = (S_i \ltimes L_i\reg)^{(0)}$, the isotropy group $(S_i \ltimes L_i\reg)_l^l$ satisfies the Baum--Connes conjecture, and that for every $i$, the discretisation correspondence $\Omega_i \colon S_i \ltimes L_i\reg \to S_i \ltimes \widehat{L_i}$ from \S~\ref{s:Disc} induces an isomorphism
$\K_*(\csr(\Omega_i)) \colon \K_*(\csr(S_i \ltimes L_i\reg) ) \isom \K_*(\csr(S_i \ltimes \widehat L_i))$. Then $G$ satisfies the Baum--Connes conjecture.
\ecor
\setlength{\parindent}{0cm} \setlength{\parskip}{0cm}

\bproof
Since our groupoids $G_i$ are exact, our finite length independent resolution yields a long exact sequence of reduced groupoid $\cs$-algebras, which splits into finitely many short exact sequences, which in turn induce finitely many long exact sequences in K-theory. We obtain analogous long exact sequences for the left-hand side of the Baum--Connes conjecture, i.e., in $\Ktop_*$ (see \cite[\S~7.1]{Tu_1999a} and \cite[\S~5]{Tu}). By naturality of these long exact sequences, induction on the length of our independent resolution shows that if $G_i$ satisfy the Baum--Connes conjecture for all $i$, then so does $G$. By Remark~\ref{rem:BC=Disc}, our assumptions imply that $G_i$ satisfy the Baum--Connes conjecture for all $i$, which completes the proof.
\eproof
\setlength{\parindent}{0cm} \setlength{\parskip}{0.5cm}

In the following, let us develop a concrete application. We need the following terminology: Given a countable discrete group $\Gamma$, we say that the discretisation formula with arbitrary coefficients is valid for $\Gamma$ if for every countable semilattice $E$, every $\Gamma$-action $\Gamma \acts E$ by order automorphisms and every separable $\Gamma$-$\cs$-algebra $A$, we have an isomorphism
\begin{equation}
\label{e:DiscForm_K_Gamma}
 \K_*(\Gamma \ltimes_r (A \otimes C_0(E\reg))) \cong \K_*(\Gamma \ltimes_r (A \otimes C_0(\widehat{E})) ).
\end{equation}
Here $\Gamma$ acts via diagonal actions on $A \otimes C_0(E\reg)$ and $A \otimes C_0(\widehat{E})$.
\bcor
\label{cor:EquiBC}
Let $\Gamma$ be a countable discrete group. Assume that $\Gamma$ is exact. Suppose that there exists a finite length independent resolution for $G_{-1} = \Gamma$ consisting of groupoids $G_i$ of the form $G_i = \Gamma \ltimes \widehat{E}_i$ for countable semilattices $E_i$ with the property that, for all $e \in E\reg_i$, the stabilizer group $\Gamma_e$ is finite. Then $\Gamma$ satisfies the Baum--Connes conjecture with coefficients if and only if the discretisation formula with arbitrary coefficients is valid for $\Gamma$.
\ecor
\setlength{\parindent}{0cm} \setlength{\parskip}{0cm}

\bproof
The implication \an{$\Rarr$} follows from \cite[Theorem~3.12]{CEL2}. For the converse, we use the following equivalent formulation of the Baum--Connes conjecture with coefficients. As pointed out in \cite[\S~5.3]{CMR}, a countable discrete group $\Gamma$ satisfies the Baum--Connes conjecture with coefficients if and only if the following statement is true: For every separable $\Gamma$-$\cs$-algebra $A$, if $\K_*(F \ltimes_r A) = 0$ for every finite subgroup $F \subseteq \Gamma$ (here the $F$-action on $A$ is given by restricting the $\Gamma$-action on $A$), then $\K_*(\Gamma \ltimes_r A) = 0$. Now suppose that $A$ is a separable $G$-$\cs$-algebra such that $\K_*(F \ltimes_r A) = 0$ for every finite subgroup $F \subseteq \Gamma$. Since $\Gamma$ is exact, our finite length independent resolution for $\Gamma$ produces an exact sequence
\begin{equation}
\label{e:ses_AxGamma}
 0 \to A_k \to \dotso \to A_2 \to A_1 \to A_0 \to \Gamma \ltimes_r A \to 0,
\end{equation}
where each $A_i$ is of the form $A_i = \Gamma \ltimes_r (A \otimes C_0(\widehat{E_i})) $. Using \eqref{e:DiscForm_K_Gamma} and the analogue of Corollary~\ref{cor:OSF_K}, we deduce that
\[
 \K_*(\Gamma \ltimes_r (A \otimes C_0(\widehat{E_i}))) \cong \bigoplus_{[e] \in \Gamma \backslash E\reg_i} \K_*(\Gamma_e \ltimes_r A).
\]
Since $\Gamma_e$ is a finite subgroup of $\Gamma$ by assumption, it follows that $\K_*(\Gamma \ltimes_r (A \otimes C_0(\widehat{E_i}))) = 0$ for all $0 \leq i \leq k$. Exactness of \eqref{e:ses_AxGamma} then implies that $\K_*(\Gamma \ltimes_r A) = 0$, as desired.
\eproof
\setlength{\parindent}{0cm} \setlength{\parskip}{0.5cm}

\bremark
\label{rem:FLIndResSubgroups}
Let $\Gamma$ be a countable discrete group. Suppose that there exists a finite length independent resolution for $G_{-1} = \Gamma$ consisting of groupoids $G_i = \Gamma \ltimes \widehat{E}_i$ as in Corollary~\ref{cor:EquiBC} with the property that, for all $e \in E\reg_i$, the stabilizer group $\Gamma_e$ is finite. Then the same is true for any subgroup $\Gamma'$ of $\Gamma$. Indeed, we just have to restrict the $\Gamma$-action on $E_i$ to $\Gamma'$ and observe that $G'_i \coloneq \Gamma' \ltimes \widehat{E}_i$ is an independent resolution for $G'_{-1} = \Gamma'$ ($\Gamma'_e$ is finite because $\Gamma'_e \subseteq \Gamma_e$).
\eremark

Here is an alternative approach to Corollary~\ref{cor:EquiBC}. We first need the following.
\bremark
\label{rem:xFromCEL} 
Suppose that $\Gamma$ is a countable discrete group acting on a countable semilattice $E$. In this situation, an element $\bmx \in \KK^{\Gamma}(C_0(E\reg),C_0(\widehat{E}))$ was constructed in \cite[\S~3]{CEL2} such that the induced map $\K_*(\Gamma \ltimes_r \bmx) \colon \K_*(\Gamma \ltimes_r C_0(E\reg)) \to \K_*(\Gamma \ltimes_r C_0(\widehat{E}))$ in K-theory  can be identified with the map $\K_*(\csr(\Omega))$ induced by the discretisation correspondence $\Omega \colon \Gamma \ltimes E\reg \to \Gamma \ltimes \widehat{E}$ from \S~\ref{s:Disc}.
\eremark
With this preparation, we can now state the following applications.
\bcor
Let $\Gamma$ be a countable discrete group. Assume that $\Gamma$ is exact. Suppose that there exists a finite length independent resolution for $G_{-1} = \Gamma$ consisting of groupoids $G_i$ of the form $G_i = \Gamma \ltimes \widehat{E}_i$ for countable semilattices $E_i$.
\setlength{\parindent}{0cm} \setlength{\parskip}{0cm}

\begin{enumerate}[label=(\roman*)]
    \item Suppose that, for every $i$ and every $e \in E\reg_i$, the stabilizer group $\Gamma_e$ satisfies the Baum--Connes conjecture, and that for every $i$, the element $\bmx_i \in \KK^{\Gamma}(C_0(E_i\reg),C_0(\widehat{E_i}))$ as in Remark~\ref{rem:xFromCEL} induces an isomorphism $\K_*(\Gamma \ltimes_r \bmx_i) \colon \K_*( \Gamma \ltimes_r C_0(E_i\reg)) \isom \K_*(\Gamma \ltimes_r C_0(\widehat{E_i}) )$. Then $\Gamma$ satisfies the Baum--Connes conjecture.
    \item Let $A$ be a separable $\Gamma$-$\cs$-algebra. Suppose that, for every $i$ and every $e \in E\reg_i$, the stabilizer group $\Gamma_e$ satisfies the Baum--Connes conjecture with coefficient $A$, and that for every $i$, the element $\bmx_i \in \KK^{\Gamma}(C_0(E_i\reg),C_0(\widehat{E_i}))$ as in Remark~\ref{rem:xFromCEL} induces an isomorphism $\K_*(\Gamma \ltimes_r ([\id_A] \otimes_{\Cz} \bmx_i)) \colon \K_*(\Gamma \ltimes_r (A \otimes C_0(E_i\reg))) \isom \K_*( \Gamma \ltimes_r (A \otimes C_0(\widehat{E_i})) )$. Then $\Gamma$ satisfies the Baum--Connes conjecture with coefficient $A$.
    \item The following are equivalent:
    \begin{enumerate}
        \item[(a)] For every $i$ and every $e \in E\reg_i$, the stabilizer group $\Gamma_e$ satisfies the Baum--Connes conjecture with arbitrary coefficients, and for every $i$ and every separable $\Gamma$-$\cs$-algebra $A$, the element $\bmx_i \in \KK^{\Gamma}(C_0(E_i\reg),C_0(\widehat{E_i}))$ as in Remark~\ref{rem:xFromCEL} induces an isomorphism $\K_*(\Gamma \ltimes_r ([\id_A] \otimes_{\Cz} \bmx_i)) \colon \K_*(\Gamma \ltimes_r (A \otimes C_0(E_i\reg))) \isom \K_*( \Gamma \ltimes_r (A \otimes C_0(\widehat{E_i})) )$. 
        \item[(b)] $\Gamma$ satisfies the Baum--Connes conjecture with arbitrary coefficients.
    \end{enumerate}
\end{enumerate}
\ecor
\setlength{\parindent}{0cm} \setlength{\parskip}{0cm}

\bproof
The proof of (i) is analogous to the one of Corollary~\ref{cor:FLIR->BC}. The proof of (ii) is also similar, adding the coefficient algebra $A$. For (iii), (a) $\Rarr$ (b) follows from (ii), and (b) $\Rarr$ (a) follows from \cite[Theorem~3.12]{CEL2} because the Baum--Connes conjecture with arbitrary coefficients passes to subgroups by \cite[Theorem~2.5]{CE_2001}.
\eproof
\setlength{\parindent}{0cm} \setlength{\parskip}{0.5cm}

\subsection{Examples}
\label{s:Ex}

Garside categories provide a source for examples of groupoids which admit finite length independent resolutions. Let us first give a brief introduction to Garside categories, which in this article will mean small categories with Garside families with particular properties as specified below.\footnote{The notion of a Garside category is in general somewhat flexible; the emphasis is on studying Garside families with various desirable properties, but different situations merit different requirements on the Garside family.} We will follow the exposition in \cite{Li21b} and refer the reader to \cite{Deh15,Li21a,Li21b} and the references therein for more details.

Let $\fC$ be a small category. The set of objects $\fC^0$ in $\fC$ will be viewed as a subset of the set of morphisms by identifying an object with the corresponding identity morphism. Moreover, we will identify the category with its set of morphisms, again denoted by $\fC$. Let $\mfd: \: \fC \to \fC^0$ and $\mft: \: \fC \to \fC^0$ be the domain and target maps. $\fC^*$ denotes the set of invertible elements of $\fC$. We need to introduce some terminology.

\bdefin
\label{d:Garside1}
$ $
\setlength{\parindent}{0cm} \setlength{\parskip}{0cm}

\begin{itemize}
 \item $\fC$ is called left cancellative, if, for all $c, x, y \in \fC$ with $\mfd(c) = \mft(x) = \mft(y)$, $cx = cy$ implies $x = y$.
 \item $\fC$ is called finitely aligned if for all $a,b \in \fC$, there exists a finite subset $\gekl{c_i} \subseteq \fC$ such that $a \fC \cap b \fC = \bigcup_i c_i \fC$.
 \item Given $a, b \in \fC$, we write $a \preceq b$ if $a$ is a left divisor of $b$, i.e., $b \in a \fC$. We write $a \prec b$ if $b \fC \subsetneq a \fC$. We write $a \tipreceq b$ if $a$ is a right divisor of $b$, i.e., $b \in \fC a$. We write $a \tiprec b$ if $\fC b \subsetneq \fC a$. We write $a =^* b$ if $a \in b \fC^*$ (which is equivalent to $a \fC = b \fC$). 
 \item A subset $\fS \subseteq \fS$ is closed under right comultiples if for all $r, s \in \fS$ and $a \in \fC$ with $r \preceq a$, $s \preceq a$, there exists $t \in \fS$ with $r \preceq t$, $s \preceq t$ and $t \preceq a$.
\end{itemize}
\edefin
\setlength{\parindent}{0cm} \setlength{\parskip}{0.5cm}

Now we can introduce the concept of a Garside family.
\bdefin\label{d:Garside2}
Suppose $\fS \subseteq \fC$ is closed under right comultiples such that $\fS \cup \fC^*$ generates $\fC$ and $\fS^{\sharp} \coloneq \fS \fC^* \cup \fC^*$ is closed under right divisors. 
\setlength{\parindent}{0cm} \setlength{\parskip}{0cm}

\begin{itemize}
 \item A finite sequence $s_1, s_2, \dotsc$ in $\fC$ is called a \emph{path} if $\mfd(s_k) = \mft(s_{k+1})$ for all $k$. Such a path will be denoted by $s_1 s_2 \dotsm$. A path $s_1, \dotsc, s_l \in \fS^{\sharp}$ is called normal if for all $1 \leq k \leq l-1$ and $r \in \fS$, if $r \preceq s_k s_{k+1}$ then $r \preceq s_k$.
 \item For $a \in \fC$, a normal decomposition or normal form of $a$ is given by a normal path $s_1 \dotsc s_l$ in $\fS^{\sharp}$ with $a = s_1 \dotsm s_l$.
 \item $\fS$ is called a \emph{Garside family} if every element in $\fC$ admits a normal decomposition.
\end{itemize}
\edefin
\setlength{\parindent}{0cm} \setlength{\parskip}{0.5cm}

In the following, assume that $\fC$ is a finitely aligned left cancellative small category, and that $\fS \subseteq \fC$ is a Garside family. Further suppose that the following hold:
\setlength{\parindent}{0cm} \setlength{\parskip}{0cm}

\begin{itemize}
 \item $\fS$ is $=^*$-transverse, i.e., for all $s_1, s_2 \in \fS$, $s_1 =^* s_2$ implies $s_1 = s_2$. 
 \item $\fS \cap \fC^* = \emptyset$.
 \item $\fS$ is locally finite, i.e., $\# \mfv \fS < \infty$ for all $\mfv \in \fC^0$. 
 \item For all $\ell \geq 1$, $(\fS^{\leq \ell})^{\sharp}$ is closed under left divisors. Here $\fS^{\leq \ell} \coloneq \fS \cup \fS^2 \cup \dotso \cup \fS^{\ell}$.
\end{itemize}
\setlength{\parindent}{0cm} \setlength{\parskip}{0.5cm}

Note that $L \coloneq \{a \fC \mid a \in \fC \} \cup \gekl{0}$ is a weak semilattice (with respect to inclusion as partial order) because $\fC$ is finitely aligned. Moreover, $L$ is locally finite because our assumptions imply (see the discussion in \cite[\S~3]{Li21b}) that, for all $\ell \geq 1$, $\fS^{\leq \ell}$ is $\downarrow$-closed and locally finite (i.e. $\# \mfv \fS^{\leq \ell} < \infty$ for all $\mfv \in \fC^0$). Let $S$ be the left inverse hull of $\fC$ as in \cite{Li21a,Li21b}, which acts as an inverse semigroup on $L$ in the natural way. 
Let $X$ be a locally compact Hausdorff totally disconnected space with an action of $S$ and an $S$-equivariant generating representation $U \colon L \to \cO_c(X)$ of $L$. Note that via Lemma \ref{lem:X=whV}, the data of $(X,U)$ corresponds to a closed invariant subspace of $\widehat{L}$. For $a \in \fC$ let
\begin{equation}
\R(a \fC) := \left\{ F \subseteq a \fS \suchthat  \bigcup_{b \in F} U(b \fC)  = U(a \fC) \right\} \cup \{\{a \fC\}\}.
\end{equation}
It follows from \cite[Lemma~3.8]{Li21b} that the system $\R$ of concrete finite covers in $L$ is thorough. It is $S$-invariant by construction.

Uniform finiteness of $\R(a \fC)$ is ensured if $\sup_{v \in \fC^0} \# \, v \fS < \infty$. 
Hence the construction in \S~\ref{ss:IndResConst} produces a finite length independent resolution of $S \ltimes X$.

Examples~\ref{ex:HRG6}, \ref{ex:ssg} as well as the examples discussed in \ref{sss:Garside} below are special cases.

\bex[Higher rank graphs]
\label{ex:HRG6}

Let $\Lambda$ be a source-free row-finite higher rank graph as in Examples~\ref{ex:HRG1}, \ref{ex:HRG2}, \ref{ex:HRG3}, \ref{ex:HRG4} and \ref{ex:HRG5}, and $G_{\Lambda}$ its graph groupoid. Theorem \ref{t:resolution} yields a finite length independent resolution $G_i = S_{\Lambda} \ltimes \widehat{L_i}$ for $G_{\Lambda}$, as explained in Example~\ref{ex:HRG5}. Since the isotropy groups of the groupoids $S_{\Lambda} \ltimes L\reg_i$ are all trivial, we are in the situation of Corollary~\ref{cor:IndResGPDH_0}, which yields that groupoid homology $\Hlgy_*(G_{\Lambda})$ is given by the homology of the chain complex 
\[
 \dotso \to \bigoplus_{S_{\Lambda} \backslash L_{i+1}\reg} \Zz \to \bigoplus_{S_{\Lambda} \backslash L_i\reg} \Zz \to  \dotso
\]
where $L_i$ are as in Example~\ref{ex:HRG5}.
\setlength{\parindent}{0.5cm} \setlength{\parskip}{0cm}

Moreover, since $G_{\Lambda}$ is Hausdorff and amenable, all our groupoids satisfy the Baum--Connes conjecture and are exact. Corollary~\ref{cor:OSF_K} allows us to compute $\K_*(\csr(G_i)) \cong \bigoplus_{S_{\Lambda} \backslash L_i\reg} \K_*(\Cz)$ because all the isotropy groups are trivial. Hence the discussion in \S~\ref{ss:AppIndRes_K} shows that there is a spectral sequence $E_{pq}^r$, with $E_{pq}^0 = \bigoplus_{S_{\Lambda} \backslash L_p\reg} \Zz$ for even $q$ and $E_{pq}^0 = 0$ for odd $q$, such that $E_{pq}^r$ converges to $0$ and hence determines $\K_*(\csr(G_{\Lambda}))$ up to extension problems.
\eex
\setlength{\parindent}{0cm} \setlength{\parskip}{0.5cm}

\bex[Self-similar groups]
\label{ex:ssg}
A \emph{self-similar group action} $(G,X,\sigma)$ is the data of a group $G$, a finite set $X$ equipped with an action $G \times X \to X$ written $(g,x) \mapsto g(x)$, and a cocycle $\sigma \colon G \times X \to G$ written $(g,x) \mapsto g|_x$. This induces an action of $G$ on the tree $X^*$ of words via the recursive formula $g(xw) = g(x) g|_x(w)$ for $g \in G$, $x \in X$ and $w \in X^*$ and a cocycle on $X^*$ given by $g|_{x_1\cdots x_n} = g|_{x_1}\cdots|_{x_n}$. From this data one produces a unital inverse semigroup $S$ generated by $G$ and $X$ subject to the relations of $G$ and the further relations
\setlength{\parindent}{0cm} \setlength{\parskip}{0cm}

\begin{enumerate}
\item $x^*y = \begin{cases} 1 & x = y \\ 0 & x \neq y \end{cases} \,$ for $x, y \in X$ (the polycyclic relations in $X$),
\item $g x = g(x) g|_x$ for $g \in G $ and $x \in X$.
\end{enumerate}
The inverse semigroup $S$ acts on the tree $X^*$ of words in $X$ through the action of $G$ and left concatenation of $X$. We obtain an action of $S$ on the boundary $\partial X^* = X^{\mathbb N}$ satisfying $g.(wz) = g(w) g|_w(z)$ and $w.z = wz$ for $g \in G$, $w \in X^*$, $z \in X^{\mathbb N}$. The groupoid $\mathscr G_{G,X}$ associated to $(G,X,\sigma)$ is the transformation groupoid $S \ltimes X^{\mathbb N}$, and its full $\cs$-algebra is the Nekrashevych $\cs$-algebra $\mathcal O_{G,X}$ \cite{Nek09}.
\setlength{\parindent}{0cm} \setlength{\parskip}{0.5cm}

The idempotent semilattice $E$ of $S$ is identified with the set $X^* \cup \{0\}$, which is represented on $X^{\mathbb N}$ via
\[ U \colon E \to \cO_c(X^{\mathbb N}), \, w \mapsto wX^{\mathbb N}  \]
for $w \in X^*$. This is an $S$-equivariant generating representation of $E$ on $X^{\mathbb N}$, but it is not  independent because $wX^{\mathbb N} = \bigcup_{x \in X} wxX^{\mathbb N}$ for each $w \in X^*$. Note that $U \colon E \to \cO_c(X^{\mathbb N})$ is a special case of the representation considered for higher rank graphs, for the rank $1$ graph on a single vertex with edge set $X$. For each $w \in X^*$ consider the finite cover $R(w) := \{wxX^{\mathbb N} \mid x \in X \}$ of $wX^{\mathbb N}$ and set \[\mathscr R(wX^{\mathbb N}) := \{\{wX^{\mathbb N}\}, R(w)\}.\] 
Then $\R$ is a thorough $S$-invariant system of concrete finite covers in $E$ with respect to $U$. Each $\mathscr R(wX^{\mathbb N})$ contains exactly one non-trivial finite cover. Theorem \ref{t:resolution} therefore produces an independent resolution of length $1$.
The semilattice $L_0 = E$ is the idempotent semilattice $E$ of $S$, with $S$-action given by $s . e = s e s^*$ for $e \leq s^* s$. The semilattice $L_1 = \{0\} \cup \{(wX^{\mathbb N};R(w)) \mid w \in X^{\mathbb N} \}$ also bijects $S$-equivariantly with $E$, but satisfies $de = 0$ whenever $d \neq e$ in $L_1$, so that $\widehat{L_1} \cong X^*$. 
Discretisation identifies both $S \ltimes \widehat{L_0}$ and $S \ltimes \widehat{L_1}$ with $S \ltimes X^*$, which is itself Morita equivalent to $G$. Our independent resolution induces the long exact sequences constituting Theorems A and B in \cite{MS}, where the key extra step of identifying the maps $\Hlgy_*(G) \to \Hlgy_*(G)$ and $\K_*(\cs(G)) \to \K_*(\cs(G))$ is performed.\footnote{Note also that on the $\cs$-algebraic side, \cite{MS} uses the full $\cs$-algebra and requires no assumption that $S \ltimes X^{\mathbb N}$ is Hausdorff or satisfies the Baum--Connes conjecture.} One can therefore view the results of \cite{MS} as a demonstration of the kinds of computations one can do with independent resolutions.
\eex

\subsubsection{Garside monoids and their enveloping groups}
\label{sss:Garside}

Let us recall the definition of Garside monoids from \cite[Chapter~I, \S~2.1]{Deh15}.\footnote{We note that this particular notion of a Garside monoid is more specific than that of a monoid viewed as a small category equipped with a Garside family with the finiteness assumptions we take above. For example, monoids associated to self-similar group actions have non-trivial invertible elements, which are disallowed here by the length function $\lambda$.} A Garside monoid consists of a monoid $M$ which is cancellative, admits a function $\lambda \colon M \to \Nz$ such that $\lambda(xy) \geq \lambda(x) + \lambda(y)$ and $\lambda(x) > 0$ for all $x \in M \setminus\{1\}$, such that any two elements of $M$ have a left- and right-lcm and a left- and right-gcd, together with an element $\Delta \in M$ whose left- and right-divisors coincide and generate $M$, such that the set $D = \{ d \in M  \mid \Delta \in d M \} $ of divisors of $\Delta$ in $M$ is finite. Let $\Gamma$ be the enveloping group of $M$. Note that $\Gamma$ is given by the group of (left- or right-) fractions for $M$, i.e., $\Gamma = M^{-1} M = M M^{-1}$. 

Let $E_0$ be the semilattice $\Gamma \cup \gekl{0}$, where we identify an element $\gamma \in \Gamma$ with the subset $\gamma M$ of $\Gamma$ and the partial order is given by inclusion. To see that $E_0$ is indeed a semilattice, take $\gamma, \zeta \in \Gamma$, write $\gamma^{-1} \zeta = x^{-1} y$ for some $x, y \in M$ and let $z$ be the right-lcm of $x$ and $y$. Now we compute
\[
 (\gamma M) \cap (\zeta M) = \gamma (M \cap \gamma^{-1} \zeta M) = \gamma (M \cap x^{-1}y M) = \gamma x^{-1} ((x M) \cap (y M)) = \gamma x^{-1} z M \in E_0.
\]
$\Gamma$ acts on the semilattice $E_0$ by order automorphisms via left multiplication. Also consider the singleton space $\{*\}$ with trivial $\Gamma$-action. The trivial representation $U \colon E_0 \to \cO_c(\{*\})$ with $U(\gamma) = \{*\}$ for each $\gamma \in \Gamma$ is (trivially) $\Gamma$-equivariant and generating, but far from independent. However, $E_0$ admits the following thorough $\Gamma$-invariant system $\R$ of concrete finite covers with respect to $U$. Given $\gamma \in E_0$, the singleton $\{\gamma d\}$ is a finite cover of $e$ for each $d \in D$. We set 
\[ \R(\gamma) \coloneq \{ \{\gamma d\} \mid d \in D \} . \]
Let us first explain why condition (M) is satisfied. Given $\gamma \in \Gamma$, $d \in D$ and $x \in M$, we have to show that $(\gamma x M) \cap (\gamma d M) = \gamma x d' M$ for some $d' \in D$. Hence it suffices to show that for every $x \in M$ and $d \in D$, there exists $d' \in M$ such that $(xM) \cap (dM) = xd'M$. Every $x \in M$ is of the form $d_1 \dotsm d_n$ for some $d_m \in D$. Hence we can proceed inductively on $n$ . For $n=1$, as $D$ is closed under right-lcms, we know that $(d_1M) \cap (dM) = d_1 d'M$ for some $d' \in D$. For the induction step, assume that $x = y \ti{d}$. By induction hypothesis, $(yM) \cap (dM) = y d'' M$ for some $d'' \in D$. Moreover, by the $n=1$ case, we know that $(\ti{d}M) \cap (d''M) = \ti{d} d' M$ for some $d' \in D$. Thus, we conclude that $(xM) \cap (dM) = (y \ti{d} M) \cap (yM) \cap (dM) = (y \ti{d} M) \cap (y d'' M) = y( (\ti{d}M) \cap (d''M) ) = y \ti{d} d' M = x d' M$, as desired. 

To check that $\R$ is thorough, let $\gamma , \gamma_1, \dots , \gamma_n \in \Gamma$. For each $1 \leq i \leq n$, write $\gamma^{-1} \gamma_i = x_i y_i^{-1}$ for some $x_i, y_i \in M$. Let $x \in M$ be the right-lcm of $x_1,\dots,x_n$. Then we have, for each $1 \leq i \leq n$, $x M \subseteq x_i M \subseteq x_i y_i^{-1} M = \gamma^{-1} \gamma_i M$ and thus $\gamma x M \subseteq \gamma_i M$. As $D$ generates $M$, we may find $d_1,\dots,d_m \in D$ with $x = d_1 d_2 \dots d_m$. This determines a chain $\{\gamma d_1\} \in \R(\gamma), \{\gamma d_1 d_2 \} \in \R(\gamma d_1), \dots, \{\gamma x\} \in \R(\gamma d_1 \dots d_{m-1})$ witnessing thoroughness. 

Hence Theorem \ref{t:resolution} produces an independent resolution for $\Gamma$. Moreover, since there are no non-trivial invertible elements in $M$, all stabilizer groups $\Gamma_{e}$ are trivial for all $e \in E_0\reg$, so that the same is true for all other semilattices $E_j$ in our independent resolution. The semilattices $E_j$ look as follows:
\begin{align*}
 E_0 &= \Gamma \cup \gekl{0},\\
 E_1 &= \menge{(\gamma \mid D_1)}{\gamma \in \Gamma; \, \emptyset \ne D_1 \subseteq D} \cup \gekl{0},\\
 E_2 &= \menge{(\gamma \mid D_1 \mid D_2)}{\gamma \in \Gamma; \, \emptyset \ne D_1, D_2 \subseteq D \text{ disjoint}} \cup \gekl{0},\\
 &\dotso,\\
 E_j &= \menge{(\gamma \mid D_1 \mid \cdots \mid D_j)}{\gamma \in \Gamma; \, \emptyset \ne D_1, \dotsc, D_j \subseteq D \text{ disjoint}} \cup \gekl{0},\\
 &\dotso
\end{align*}
with a similar recursive definition as in Example~\ref{ex:HRG5}. Note that this is a finite length independent resolution, because the length is at most $\# \, D - 1$. The following is now a consequence of Corollary~\ref{cor:EquiBC} and Remark~\ref{rem:FLIndResSubgroups}.
\bcor
\label{cor:Garside_BC}
Let $\Gamma'$ be a subgroup of the enveloping group of a Garside monoid. Assume that $\Gamma'$ is exact. Then $\Gamma'$ satisfies the Baum--Connes conjecture with arbitrary coefficients if and only if the discretisation formula with arbitrary coefficients is valid for $\Gamma'$.
\ecor

Let us now present concrete examples. First, recall that we introduced Artin--Tits groups in \S~\ref{s:intro}.
\bex
\label{ex:Artin}
An Artin--Tits monoid is a Garside monoid in the sense of \cite[Chapter~I, \S~2.1]{Deh15} if and only it is spherical. The corresponding enveloping groups, spherical Artin--Tits groups, are special classes of Artin--Tits groups. Concrete examples are given by Braid monoids and Braid groups. Spherical Artin--Tits groups are linear by \cite{CW,Dig} and thus exact by \cite{GHW}. Hence Corollary~\ref{cor:Garside_BC} implies that for a subgroup $\Gamma'$ of a spherical Artin--Tits group, the Baum--Connes conjecture with coefficients is true for $\Gamma'$ if and only if the discretisation formula with coefficients is valid for $\Gamma'$.
\eex

\subsubsection{Groupoids corresponding to Stein's groups}
\label{sss:Stein}

This class of examples appears also in, for instance, \cite{Li15} (see also \cite{Tan}, \cite{Tanarxiv} and \cite[\S~2.2.7]{Li25}). Let $\Pi$ be a multiplicative subgroup of $\Rz_{>0}$ and $A_+$ be the monoid $A \cap [0,\infty)$ for some subgroup $A \subseteq (\Rz,+)$. Assume that $\Pi$ acts on $A_+$ by multiplication. Let $S$ be the left inverse hull of the $ax+b$-monoid $A_+ \rtimes \Pi$ and $E$ its semilattice of idempotents. We have a decomposition $\widehat{E} = \Omega \amalg \gekl{\infty}$, where $\Omega$ is an $S$-invariant open, non-compact subset such that $S \ltimes \Omega$ is minimal. We are interested in a restriction of the groupoid $S \ltimes \Omega$ to a non-empty compact open subset of $\Omega$; let us call it $G$. Its topological full group has been studied in \cite{Ste92} and can be viewed as generalizations of Thompson's group $V$, where we vary the conditions on slopes and breakpoints as well as the length of the underlying interval (choosing different lengths corresponds to restricting $S \ltimes \Omega$ to different compact open subsets). This description of Stein's groups as topological full groups has been worked out in \cite{Tan}, based on ideas in \cite{Li15}, and the reader may consult \cite[\S~7.2]{Tan} for details. The groupoid $G$ we are interested in is Morita equivalent to $S \ltimes \Omega$ because $S \ltimes \Omega$ is minimal. For this reason, if we are interested in groupoid homology of $G$, it suffices to study groupoid homology of $S \ltimes \Omega$. The sequence of groupoids $S \ltimes \Omega$, $S \ltimes \widehat{E}$ and $S \ltimes \gekl{\infty}$ can be viewed as an independent co-resolution for $S \ltimes \Omega$. Let us analyse the long exact sequence in groupoid homology arising from the decomposition $\widehat{E} = \Omega \amalg \gekl{\infty}$:
\begin{equation}
\label{e:LES1_Stein}
 \cdots \to \Hlgy_{*+1}(S \ltimes \gekl{\infty}) \to \Hlgy_*(S \ltimes \Omega) \to \Hlgy_*(S \ltimes \widehat{E}) \to \Hlgy_*(S \ltimes \gekl{\infty}) \to \Hlgy_{*-1}(S \ltimes \Omega) \to \cdots
\end{equation}
Groupoid homology of $\Hlgy_*(S \ltimes \widehat{E})$ is given by Corollary~\ref{cor:OSF_H} as follows:
\[
 \Hlgy_*(S \ltimes \widehat{E}) \cong \Hlgy_*(\Pi)
\]
because the action $S \acts E\reg$ only has one orbit, and the corresponding isotropy group is $\Pi$. The groupoid $S \ltimes \gekl{\infty}$ is already discrete, and we have an isomorphism $S \ltimes \gekl{\infty} \cong A \rtimes \Pi$. Plugging all this into \eqref{e:LES1_Stein}, and using the isomorphism $\Hlgy_*(G) \cong \Hlgy_*(S \ltimes \Omega)$ induced by the canonical inclusion $G \into S \ltimes \Omega$, we obtain the following long exact sequence:
\begin{equation}
\label{e:LES2_Stein}
 \cdots \to \Hlgy_{*+1}(A \rtimes \Pi) \to \Hlgy_*(G) \to \Hlgy_*(\Pi) \to \Hlgy_*(A \rtimes \Pi) \to \Hlgy_{*-1}(G) \to \cdots
\end{equation}
Moreover, it is straightforward to see that the map $\Hlgy_*(\Pi) \to \Hlgy_*(A \rtimes \Pi)$ appearing in \eqref{e:LES2_Stein} is induced by the canonical inclusion $\Pi \into A \rtimes \Pi$. Since the canonical projection $A \rtimes \Pi \onto \Pi$ is a split for the canonical inclusion $\Pi \into A \rtimes \Pi$, we obtain an isomorphism $\Hlgy_*(G) \cong \Hlgy_{*+1}(A \rtimes \Pi) / \Hlgy_{*+1}(\Pi)$. 
Now consider the Lyndon–Hochschild–Serre spectral sequence {(see e.g. \cite[\S~VII.6]{Brown})}
\[
 \Hlgy_i(\Pi,\Hlgy_j(A,\Zz)) \Rightarrow \Hlgy_{i+j}(A \rtimes \Pi).
\]
The canonical inclusion $\Pi \into A \rtimes \Pi$ and its split $A \rtimes \Pi \onto \Pi$ allow us to view $\Hlgy_*(\Pi)$ as a direct summand of $\Hlgy_*(A \rtimes \Pi)$, and functoriality of the Lyndon–Hochschild–Serre spectral sequence enables us to quotient by $\Hlgy_*(\Pi)$ at each level of the spectral sequence. All in all, we obtain a spectral sequence converging to $\Hlgy_*(G)$ of the form
\[
 \Hlgy_p(\Pi,\Hlgy_{q+1}(A,\Zz)) \Rightarrow \Hlgy_{p+q}(G).
\]
Therefore, in this example, our independent co-resolution allows us to compute the groupoid homology of $G$. Here is a direct consequence:
\bcor
\label{cor:ConcCompStein}
$ $
\setlength{\parindent}{0cm} \setlength{\parskip}{0cm}

\begin{enumerate}[label=(\roman*)]
    \item If $\Pi$ contains a number $\xi \in \Qz_{>0} \setminus \gekl{1}$, then $\Hlgy_*(G) \otimes_{\Zz} \Qz = 0$ for all $*$, and thus the corresponding topological full group $\bmF(G)$ is rationally acyclic.
    \item Assume that $\Pi$ contains a natural number $\nu > 1$, and for all $1 \leq r \leq \rk \, A \coloneq \dim_{\Qz} (A \otimes_{\Zz} \Qz)$, $A$ is $(\nu^r -1)$-divisible (i.e., multiplication by $(\nu^r -1)$ is an isomorphism on $A$), which is the case if $(\nu^r -1) \in \Pi$ for all $1 \leq r \leq \rk \, A$. Then $\Hlgy_*(G) = 0$ for all $*$, and thus the corresponding topological full group $\bmF(G)$ is integrally acyclic.
\end{enumerate}
\ecor
\setlength{\parindent}{0cm} \setlength{\parskip}{0cm}

\bproof
(i) It suffices to show that $\Hlgy_p(\Pi, \Hlgy_{q+1}(A,\Zz)) \otimes_{\Zz} \Qz = 0$ for all $p, q \geq 0$. As group homology commutes with filtered colimits, it suffices to treat the case where $\Pi$ is finitely generated. In that case, $\Pi$ is free abelian. By the Elementary Divisor Theorem, there exists $\pi \in \Pi$ such that $\Pi / \spkl{\pi}$ is free abelian and $\xi \in \spkl{\pi}$. Let $\xi_* \colon \Hlgy_{q+1}(A,\Zz) \otimes_{\Zz} \Qz \to \Hlgy_{q+1}(A,\Zz) \otimes_{\Zz} \Qz$ be the map induced by $\xi$. As $\xi \neq 1$, we know that $\id - \xi_*$ is an isomorphism on $\Hlgy_{q+1}(A,\Zz) \otimes_{\Zz} \Qz$. As $\xi \in \spkl{\pi}$, there exists an integer $\epsilon$, which we can assume to be positive (otherwise replace $\pi$ by $\pi^{-1}$), such that $\xi = \pi^{\epsilon}$. It follows that $\id - \xi_* = \id- (\pi_*)^{\epsilon} = (\id - \pi_*)((\pi_*)^0 + \dotso + (\pi_*)^{\epsilon - 1})$ on $\Hlgy_{q+1}(A,\Zz) \otimes_{\Zz} \Qz$. Since $\id - \xi_*$ is an isomorphism, it follows that $\id - \pi_*$ is an isomorphism. Hence $\Hlgy_n(\spkl{\pi}, \Hlgy_{q+1}(A,\Zz)) \otimes_{\Zz} \Qz = 0$ for all $n \geq 0$. Now the Lyndon–Hochschild–Serre spectral sequence 
\[
 \Hlgy_m(\Pi/\spkl{\pi}, \Hlgy_n(\spkl{\pi}, \Hlgy_{q+1}(A,\Zz))) \Rightarrow \Hlgy_{m+n}(\Pi, \Hlgy_{q+1}(A,\Zz))
\]
implies that $\Hlgy_p(\Pi, \Hlgy_{q+1}(A,\Zz)) \otimes_{\Zz} \Qz = 0$ for all $p, q \geq 0$, as desired. The statement about $\bmF(G)$ follows from \cite[Corollary~C]{Li25}.
\setlength{\parindent}{0cm} \setlength{\parskip}{0.5cm}

(ii) It suffices to show that $\Hlgy_p(\Pi, \Hlgy_{q+1}(A,\Zz)) = 0$ for all $p, q \geq 0$. The proof is similar to the one in (i). The only difference is that, if $\nu_* \colon \Hlgy_{q+1}(A,\Zz) \to \Hlgy_{q+1}(A,\Zz)$ is the map induced by $\nu$, then our assumptions on $A$ imply that $\id - \nu_*$ is an isomorphism on $\Hlgy_{q+1}(A,\Zz)$. In other words, there is no need to rationalize. The statement about $\bmF(G)$ follows from \cite[Corollary~D]{Li25} because $G$ is purely infinite minimal and the unit space of $G$ is homeomorphic to the Cantor space (see \cite[\S~7]{Tan}, in particular \cite[Lemmas~7.2.4 and 7.2.7]{Tan}).
\eproof
\setlength{\parindent}{0cm} \setlength{\parskip}{0.5cm}

\appendix

\section{\texorpdfstring{Equivariant étale correspondences to equivariant $\cs$-correspondences}{Equivariant étale correspondences to equivariant C*-correspondences}}
\label{correspondence functor}

For K-theoretic considerations we shall heavily utilise the construction of a $\cs$-correspondence $\cs(\Omega) \colon \cs(G) \to \cs(H)$ from an étale correspondence $\Omega \colon G \to H$. In particular, we need a version of this that works equivariantly with respect to actions of a fixed groupoid, for which we must check compatibility with semidirect product groupoids and crossed products.

We follow the treatment from \cite{AKM22}.\footnote{The reader should be warned that we use the opposite convention to \cite{AKM22} on the direction $G \to H$ of an étale correspondence.} Recall that for étale groupoids $G$ and $H$, an \emph{étale correspondence} $\Omega \colon G \to H$ is a $G$-$H$ bispace $\Omega$ such that the right action $\Omega \rightacts H$ is free, proper and étale. The étale correspondence is called \emph{proper} if the induced map $\Omega/H \to G^0$ is proper. The \emph{composition} $\Lambda \circ \Omega \colon H \to K$ of $\Omega \colon G \to H$ and $\Lambda \colon H \to K$ is given by the $G$-$K$ bispace $\Omega \times_H \Lambda$, which is the quotient of $\Omega \times_{H^0} \Lambda$ after identifying $(\omega . h,\lambda)$ with $(\omega, h . \lambda)$ whenever this makes sense. Composition is associative up to canonical isomorphisms and for each $G$, there is an identity correspondence whose $G$-$G$ biset is $G$ under left and right multiplication, which performs as the identity morphism up to canonical isomorphisms. All in all, this forms the structure of a bicategory. In this paper, we have no need for $2$-categorical considerations, as it suffices to consider the ordinary category whose morphisms are isomorphism classes of étale correspondences. For an étale correspondence $\Omega \colon G \to H$ with range and source anchors $\rho \colon \Omega \to G^0$ and $\sigma \colon \Omega \to H^0$, the $\cs$-correspondence $\cs(\Omega) \colon \cs(G) \to \cs(H)$ is constructed as a completion of $C_c(\Omega)$\footnote{For a locally compact locally Hausdorff space $X$, we abuse notation and write $C_c(X)$ for the set of functions $\mathop{\mathrm{span}} \{ f \colon X \to \mathbb C \mid \exists U \subseteq X \text{ open, Hausdorff with} f = 0 \text{ outside } U \text{ and } f|_U \in C_c(U)\}$, which need not be continuous.} with the following operations
\begin{align*}
C_c(G) \times C_c(\Omega) & \to C_c(\Omega) \\
(a,\xi) & \mapsto a . \xi \colon \omega \mapsto \sum_{g \in G_{\rho(\omega)}} a(g^{-1}) \xi(g . \omega),  \\
C_c(\Omega) \times C_c(H) & \to C_c(\Omega) \\
(\xi,b) & \mapsto \xi . b \colon \omega \mapsto \sum_{h \in H^{\sigma(\omega)}} \xi(\omega . h) b(h^{-1}),  \\
C_c(\Omega) \times C_c(\Omega) & \to C_c(H) \\
(\eta,\xi) & \mapsto \langle \eta, \xi \rangle \colon h \mapsto \sum_{\omega \in \Omega_{r(h)}} \overline{\eta(\omega)} \xi(\omega . h). 
\end{align*}
Proper étale correspondences induce proper $\cs$-correspondences, which induce maps in (K)K-theory. We will often use étale correspondences as models for these maps in (K)K-theory. We also get a $\cs$-correspondence at the reduced level when the left action is free, and therefore for any \'etale correspondence built from an inverse semigroup equivariant topological correspondence:

\bprop\label{free and reduced}
Let $\Omega \colon G \to H$ be an \'etale groupoid correspondence and suppose that $G \acts \Omega$ is free. Then $\cs(\Omega) \colon \cs(G) \to \cs(H)$ descends to the reduced level $\csr(\Omega) \colon \csr(G) \to \csr(H)$.
\eprop
\setlength{\parindent}{0cm} \setlength{\parskip}{0cm}

\begin{proof}
Let $\csr(\Omega)$ denote the Hilbert $\csr(H)$-module obtained from the Hilbert $\cs(H)$-module $\cs(\Omega)$, and let $\|-\|_\lambda$ denote the resulting seminorm on $\cs(\Omega)$. We must check that the operator norm of elements of $\cs(G)$ acting on $\csr(\Omega)$ is bounded by the norm in $\csr(G)$.
\setlength{\parindent}{0cm} \setlength{\parskip}{0.5cm}

Let $a \in C_c(G) \subseteq \cs(G)$ be positive, and let $\xi \in C_c(\Omega)$, $y \in H^0$ and $v \in C_c(H_y)$. We want to bound $\langle v, \langle \xi, a \xi \rangle v \rangle$ by $\|a\|_\lambda \|\xi\|_\lambda^2\|v\|_2^2$. Let $T \subseteq \Omega_y$ be a transversal for $G \backslash \Omega_y$, so that $G \times_{G^0} T \to \Omega_y \colon (g,t) \mapsto g . t$ is a bijection. For each $t \in T$ define a vector 
\begin{align*}
w_t & \in C_c(G_{\rho(t)}) \\
g & \mapsto \sum_{h \in H_y} \xi(g . t . h^{-1}) v(h).
\end{align*}
Then by construction we have 
\[ \langle v, \langle \xi, a \xi \rangle v \rangle = \sum_{t \in T} \langle w_t , a w_t \rangle \leq  \|a\|_\lambda \sum_{t \in T} \langle w_t , w_t \rangle = \|a\|_\lambda \langle v, \langle \xi , \xi \rangle v \rangle \leq \|a\|_\lambda \|\xi\|_\lambda^2\|v\|_2^2. \qedhere \]
\end{proof}
\setlength{\parindent}{0cm} \setlength{\parskip}{0.5cm}

We also need to utilise étale correspondences as models for maps in equivariant KK-theory, that is, morphisms in the equivariant Kasparov category $\KK^G$ for some étale groupoid $G$. For this, we need to consider étale groupoids and correspondences equipped with actions of $G$. We will show that the above construction, when applied to $G$-equivariant étale correspondences, produces $G$-equivariant $\cs$-correspondences.

\begin{definition}
Let $G$ and $K$ be étale groupoids. A left action of $G$ on $K$ as a groupoid is a left action of $G$ on the underlying space $K$ such that the unit space $K^0$ is $G$-invariant, the range and source maps $r,s \colon K \to K^0$ are $G$-equivariant, and the multiplication map $K^2 \to K$ is $G$-equivariant with respect to the diagonal action on $K^2$ (which makes sense because the range and source maps are $G$-equivariant). We say $K$ is a \emph{(left) $G$-groupoid}. If $J$ and $K$ are $G$-groupoids, a \emph{$G$-equivariant étale correspondence} $\Omega \colon J \to K$ is an étale correspondence $\Omega \colon J \to K$ equipped with an action of $G$ whose range and source maps $\rho \colon \Omega \to J^0$, $\sigma \colon \Omega \to K^0$ are $G$-equivariant, and such that the action maps $J \times_{J^0} \Omega \to \Omega$ and $\Omega \times_{K^0} K \to \Omega$ are $G$-equivariant with respect to the resulting diagonal actions on $J \times_{J^0} \Omega$ and $\Omega \times_{K^0} K$.
\end{definition}

The $\cs$-algebra $\cs(K)$ of an étale $G$-groupoid $K$ inherits the structure of a $G$-$\cs$-algebra as follows. Let $\tau \colon K \to G^0$ denote the anchor map of the action. There is an action of $C_0(G^0)$ on $\cs(K)$ via the composition 
\[ C_0(G^0) \to M(C_0(K^0)) \to M(\cs(K)) \]
which satisfies $f . \xi \in C_c(K)$ for $f \in C_0(G^0)$ and $\xi \in C_c(K)$ with  $(f . \xi)(k) = f(\tau(k)) \xi(k)$. The image is contained in the centre of $M(\cs(K))$ because $\tau \circ r = \tau = \tau \circ s \colon K \to G^0$, giving $\cs(K)$ the structure of a $C_0(G^0)$-algebra. For each $x \in G^0$, the preimage $\tau^{-1}(x) \subseteq K$ is the restriction of $K$ to the closed invariant subset $\tau^{-1}(x) \cap K^0$. The kernel of the quotient map $\cs(K) \to \cs(\tau^{-1}(x))$ induced by restriction of functions from $K$ to $\tau^{-1}(x)$ may be identified with $\cs(K \setminus \tau^{-1}(x))$. It follows that $\cs(\tau^{-1}(x))$ is the fibre of $\cs(K)$ at $x$ as a $C_0(G^0)$-algebra.

\begin{lemma}\label{continuity lemma}
Let $X$ be a locally compact Hausdorff space, let $\mathcal A \to X$ be a Banach bundle and let $Z$ be a space with two local homeomorphisms $r,s \colon Z \to X$. Suppose $\Gamma \subseteq \Gamma_0(X, \mathcal A)$ has dense span. Let $(z,a) \mapsto (z,z . a) \colon s^* \mathcal A \to r^* \mathcal A$ be a function which is linear and contractive on each fibre over $Z$. If for each $\xi \in \Gamma$, the map $g \mapsto g . \xi_{s(g)} \colon G \to \mathcal A$ is continuous, then $(z,a) \mapsto (z,z . a)$ is continuous.  
\end{lemma}
\setlength{\parindent}{0cm} \setlength{\parskip}{0cm}

\begin{proof}
Let us first reduce to the case with $Z$ locally compact and Hausdorff. For $z \in Z$ and $a \in A_{s(z)}$, pick some open neighbourhood $O$ of $z$ on which $r$ and $s$ are injective. Then $s|_O^* \mathcal A$ is an open set containing $(z,a) \in s^* \mathcal A$, and the image lands in the open set $r|_O^* \mathcal A \subseteq r^* \mathcal A$. Since $O$ is locally compact and Hausdorff, this reduces the statement, and thus we assume $Z$ to be locally compact and Hausdorff. Suppose we have a convergent net $(z_i,a_i) \to (z,a) \in s^* \mathcal A$ and let $\epsilon > 0$. Let $\xi \in \mathop{\mathrm{span}} \Gamma$ with $\lVert \xi_{s(z)} - a \rVert < \epsilon$. Then by assumption we have $z_i . \xi_{s(z_i)} \to z . \xi_{s(z)}$. By upper semicontinuity of the norm we have $\lVert \xi_{s(z_i)} - a_i \rVert < \epsilon$ for sufficiently large $i$, and hence also $\lVert z_i . \xi_{s(z_i)} - z_i . a_i \rVert < \epsilon$. It follows by \cite[Proposition C.20]{Wil07} that $z_i . a_i \to z . a$.
\end{proof}
\setlength{\parindent}{0cm} \setlength{\parskip}{0.5cm}

\begin{lemma}\label{continuity of action lemma}
Let $G$ be an étale groupoid and $K$ an étale $G$-groupoid with anchor $\tau \colon K \to G^0$. Let $A = \cs(K)$ be the $C_0(G^0)$-algebra as described above with associated bundle $\mathcal A \to G^0$. Consider for $g \in G$ the $*$-isomorphism $\alpha_g \colon A_{s(g)} \to A_{r(g)}$ induced by the groupoid isomorphism $k \mapsto g . k \colon \tau^{-1}(s(g)) \to \tau^{-1}(r(g))$. The resulting action map $s_G^* \mathcal A \to r_G^* \mathcal A$ gives $A$ the structure of a $G$-$\cs$-algebra.
\end{lemma}
\setlength{\parindent}{0cm} \setlength{\parskip}{0cm}

\begin{proof}
We need only check continuity of the action map. By Lemma \ref{continuity lemma} (see also \cite[Lemma 3.9]{Bön20}, \cite[Lemma 2.22]{Mil1}) it suffices to check that $g \mapsto \alpha_g(\xi_{s(g)}) \colon G \to \mathcal A$ is continuous for each $\xi \in C_c(K) \subseteq A$. The function $\alpha_g(\xi_{s(g)}) \in C_c(\tau^{-1}(r(g)))$ is given at $k \in \tau^{-1}(r(g))$ by $\xi(g^{-1} . k)$. Fix $g_0 \in G$ and let $U \subseteq G$ be an open bisection containing $g_0$. Pick $\gamma \in C_c(U)$ with $\gamma = 1$ on an open neighbourhood $V \subseteq U$ of $g_0$. The function
\begin{align*}
\gamma . \xi \colon K & \to \mathbb C \\
k & \mapsto \sum_{g \in G_{\rho(k)}} \gamma(g) \xi(g^{-1} . k)
\end{align*}
is continuous and compactly supported. For $g \in V$ and $k \in \tau^{-1}(r(g))$, we compute $(\gamma . \xi)(k) = \xi(g^{-1} . k)$, which means that $(\gamma . \xi)_{s(g)} = \alpha_g(\xi_{s(g)})$, and therefore $g \mapsto \alpha_g(\xi_{s(g)})$ is continuous on $V$.
\end{proof}
\setlength{\parindent}{0cm} \setlength{\parskip}{0.5cm}

Let $G$ be étale and let $\Omega \colon J \to K$ be a $G$-equivariant étale correspondence with anchor maps $\tau_J \colon J \to G^0$, $\tau_\Omega \colon \Omega \to G^0$ and $\tau_K \colon K \to G^0$. The $C_0(G^0)$-algebra structure on $\cs(K)$ induces the structure of a locally $C_0(X)$-convex Banach space on $\cs(\Omega)$, satisfying $f . \xi \in C_c(\Omega)$ for $f \in C_0(G^0)$ and $\xi \in C_c(\Omega)$, where $(f . \xi)(\omega) = f(\tau_\Omega(\omega)) \xi(\omega)$ at $\omega \in \Omega$. For $x \in G^0$, we note that $\tau_\Omega^{-1}(x) \colon \tau_J^{-1}(x) \to \tau_K^{-1}(x)$ is itself an étale correspondence, and the fibre of $\cs(\Omega)$ at $x$ is identified with the Hilbert $\cs(\tau_K^{-1}(x))$-module $\cs(\tau_\Omega^{-1}(x))$ via the restriction map $C_c(\Omega) \to C_c(\tau_\Omega^{-1}(x))$. For $g \in G$ the homeomorphism $\omega \mapsto g . \omega \colon \tau_{\Omega}^{-1}(s(g)) \to \tau_{\Omega}^{-1}(r(g))$ is compatible with the associated isomorphism $\tau_K^{-1}(s(g)) \to \tau_K^{-1}(r(g))$ and thus induces an isometric isomorphism $\beta_g \colon \cs(\tau_\Omega^{-1}(s(g))) \to \cs(\tau_\Omega^{-1}(r(g)))$. By exactly the same argument as in Lemma \ref{continuity of action lemma}, this gives a continuous action of $G$ on the Banach bundle associated to $\cs(\Omega)$, which by construction turns $\cs(\Omega)$ into a $G$-Hilbert $\cs(K)$-module. Moreover, it is clear that the action of $\cs(J)$ on $\cs(\Omega)$ is $G$-equivariant. 

In summary, we have constructed a $G$-equivariant $\cs$-correspondence $\cs(\Omega) \colon \cs(J) \to \cs(K)$ from the $G$-equivariant \'etale correspondence $\Omega \colon J \to K$. We note moreover that the canonical isomorphisms of $\cs$-correspondences witnessing compatibility of the $\Omega \mapsto \cs(\Omega)$ construction with composition and identities \cite[Theorem 7.13]{AKM22} are evidently $G$-equivariant. We have proved:

\begin{proposition}\label{prop:G equivariant correspondence functor}
The construction of a $G$-equivariant $\cs$-correspondence $\cs(\Omega) \colon \cs(J) \to \cs(K)$ from a $G$-equivariant étale correspondence $\Omega \colon J \to K$ described above respects composition and identities up to $G$-equivariant isomorphism.
\end{proposition}

We take the chance here to highlight a particular special case of the above construction. Let $G$ be an étale groupoid and let $X$ and $Y$ be locally compact Hausdorff $G$-spaces. A \emph{$G$-equivariant topological correspondence} $Z \colon X \to Y$ is a locally compact Hausdorff $G$-space $Z$ along with a $G$-equivariant local homeomorphism $\sigma \colon Z \to Y$ and a $G$-equivariant map $\rho \colon Z \to X$. It is proper if $\rho$ is proper. This is precisely a $G$-equivariant étale correspondence if we consider $X$ and $Y$ as $G$-groupoids. We obtain a $G$-equivariant $\cs$-correspondence $\cs(Z) \colon C_0(X) \to C_0(Y)$. In \cite{PY25}, $Z$ is called a \emph{span} and $Z \mapsto \cs(Z)$ \emph{Atiyah transfer}.

Let $G$ be an étale groupoid and let $K$ be an étale $G$-groupoid. Then $G \ltimes K$ carries the structure of an étale groupoid with unit space $K^0$, with $s(g,k) = s(k)$, $r(g,k) = r(g . k)$ and multiplication given by $(g_1.k_1)(g_2,k_2) = (g_1 g_2, (g_2^{-1} . k_1) k_2)$ for composable pairs.

\begin{proposition}\label{crossed product identification}
Let $G$ be an étale groupoid and let $K$ be an étale $G$-groupoid. Write $\mathcal A \to G^0$ for the $G$-Banach bundle associated to $\cs(K)$. Then there is a $*$-isomorphism $\Phi \colon \cs(G \ltimes K) \to G \ltimes \cs(K)$ given at $\xi \in C_c(G \ltimes K)$ by 
\begin{align*}
\Phi \colon C_c(G \ltimes K) & \to \Gamma_c(G,s^* \mathcal A) \\
\xi & \mapsto [g \mapsto [k \mapsto \xi(g,k)]]
\end{align*}
\end{proposition}
\setlength{\parindent}{0cm} \setlength{\parskip}{0cm}

\begin{proof}
We first argue that $\Phi$ as above is well-defined on $C_c(G \ltimes K)$ - by linearity it suffices to consider $\xi \in C_c(U;K) = C_c(\{(g,k) \in G \ltimes K \mid g \in U \})$ for each open bisection $U \subseteq G$. Let $\tau \colon K \to G^0$ be the anchor map for the $G$-space $K$. Then, for fixed $g \in U$, the function $k \mapsto \xi(g,k)$ is an element of $C_c(\tau^{-1}(s(g))) \subseteq A_{s(g)}$. To see that $g \mapsto [k \mapsto \xi(g,k)] \colon U \to s^* \mathcal A$ is continuous, fix $g_0 \in U$ and pick $\gamma \in C_c(U)$ with $\gamma = 1$ on an open neighbourhood $V \subseteq U$ of $g_0$. The function 
\begin{align*}
\eta \colon K & \to \mathbb C \\
k & \mapsto \sum_{g \in G_{\tau(k)}} \gamma(g) \xi(g, k).
\end{align*}
is continuous and compactly supported. We compute $\eta(k) = \xi(g,k)$ for $g \in V$ and $k \in \tau^{-1}(s(g))$. It follows that $g \mapsto [k \mapsto \xi(g,k)] = \eta_{s(g)}$ is continuous on $V$. We obtain a $*$-homomorphism $C_c(G \ltimes K) \to \Gamma_c(G, s^* \mathcal A)$ which by universality extends to a $*$-homomorphism $\Phi \colon \cs(G \ltimes K) \to G \ltimes \cs(K)$.
\setlength{\parindent}{0cm} \setlength{\parskip}{0.5cm}

For an open bisection $U \subseteq G$, the inclusion $\Phi \colon C_c(U;K) \hookrightarrow \Gamma_c(U, s^* \mathcal A)$ is isometric by the $\cs$-identity. Indeed, for $\xi \in C_c(U;K)$, we have $\Phi(\xi)^*\Phi(\xi) \in \Gamma_c(s(U),\mathcal A) \subseteq C_c(K) \subseteq \cs(K)$. The open inclusion $K \hookrightarrow G \ltimes K$ induces a $*$-homomorphism $\cs(K) \to \cs(G \ltimes K)$ which sends $\Phi(\xi)^* \Phi(\xi)$ to $\xi^* \xi$, and thus $\lVert \Phi(\xi) \rVert \geq \lVert \xi \rVert$. The image is dense because $\Gamma_c(U, s^* \mathcal A)$ carries the supremum norm over $U$ and the image in each fibre $(s^* \mathcal A)_g$ for $g \in U$ contains the dense subset $C_c(\tau^{-1}(s(g))$.

Suppose $\pi \colon C_c(G \ltimes K) \to \mathcal B(\mathcal H)$ is a $*$-representation. Then we get a unique extension $\pi_U \colon \Gamma_c(U,s^* \mathcal A) \to \mathcal B(\mathcal H)$ for each $U$. These subspaces span $\Gamma_c(G,s^*\mathcal A)$, and we claim the extension $\widetilde \pi \colon \sum_{i=1}^n \xi_i \mapsto \sum_{i=1}^n \pi_{U_i}(\xi_i) \colon \Gamma_c(G, s^* \mathcal A) \to \mathcal B(\mathcal H)$ of $\pi$ to $\Gamma_c(G, s^* \mathcal A)$ is well-defined, where $\xi_i \in \Gamma_c(U_i,s^* \mathcal A)$ for some open bisection $U_i$. Indeed, if $\sum_{i=1}^n \xi_i = 0$, consider $V = \{ g \in U_1 \mid f(g) \ne 0 \} \subseteq U_1$, which is open and Hausdorff and covered by $\{U_i \cap U_1\}_{i \geq 2}$. Via a partition of unity, we may find $\eta_i \in \Gamma_c(U_i \cap U_1, s^* \mathcal A)$ with $\xi_1 = \sum_{i=2}^n \eta_i$, so that $\sum_{i=2}^n (\xi_i + \eta_i) = 0$. By induction we can assume $\sum_{i=2}^n \pi_{U_i}(\xi_i + \eta_i) = 0$, but since $\pi_{U_i}(\eta_i) = \pi_{U_1}(\eta_i)$, we conclude $\sum_{i=1}^n \pi_{U_i}(\xi_i) = 0$. The extension $\widetilde \pi$ is by construction a $*$-representation, hence $\Phi$ is isometric and thus a $*$-isomorphism.
\end{proof}
\setlength{\parindent}{0cm} \setlength{\parskip}{0.5cm}

Let $\Omega \colon J \to K$ be a $G$-equivariant étale correspondence. Then there is an étale correspondence $G \ltimes \Omega \colon G \ltimes J \to G \ltimes K$ with underlying space $G \ltimes \Omega = G \times_{G^0} \Omega$, anchor maps given by $\rho(g,\omega) = \rho(g . \omega)$ and $\sigma(g,\omega) = \sigma(\omega)$ and actions given by $(g_1,j) . (g_2,\omega) = (g_1 g_2, (g_2^{-1} . j) . \omega)$ and $(g_1,\omega) . (g_2,k) = (g_1 g_2, (g_2^{-1} . \omega) . k)$. This construction is compatible with composition: given another $G$-equivariant étale correspondence $\Lambda \colon K \to L$, there is a canonical isomorphism $(G \ltimes \Omega) \times_{G \ltimes K} (G \ltimes \Lambda) \to G \ltimes (\Omega \times_K \Lambda)$ given by $[(g_1,\omega),(g_2,\lambda)] \mapsto (g_1 g_2, [g_2^{-1} . \omega, \lambda])$.

\begin{proposition}\label{crossed product correspondence identification}
Let $G$ be an étale groupoid and let $\Omega \colon J \to K$ be a $G$-equivariant topological correspondence. Then $G \ltimes \cs(\Omega) \colon G \ltimes \cs(J) \to G \ltimes \cs(K)$ is isomorphic to $\cs(G \ltimes \Omega) \colon \cs(G \ltimes J) \to \cs(G \ltimes K)$. 
\end{proposition}
\setlength{\parindent}{0cm} \setlength{\parskip}{0cm}

\begin{proof}
Consider the $G$-Banach bundle $\mathcal E \to G^0$ associated to the $G$-Hilbert $\cs(K)$-module $\cs(\Omega)$, whose fibre at $w \in G^0$ is a completion of $C_c(\tau_\Omega^{-1}(w))$. By the same argument as in Proposition \ref{crossed product identification}, the map 
\begin{align*}
\Phi_\Omega \colon C_c(G \ltimes \Omega) & \to \Gamma_c(G, s^* \mathcal E) \\
\xi & \mapsto [g \mapsto [\omega \mapsto \xi(g,\omega)]]
\end{align*}
is well-defined. Consider the $*$-isomorphisms $\Phi_J \colon \cs(G \ltimes J) \to G \ltimes \cs(J)$ and $\Phi_K \colon \cs(G \ltimes K) \to G \ltimes \cs(K)$ from Proposition \ref{crossed product identification}. Expanding terms shows that for $a \in C_c(G \ltimes J)$, $b \in C_c(G \ltimes K)$ and $\xi,\eta \in C_c(G \ltimes \Omega)$, we have 
\begin{align*}
\Phi_J(a) . \Phi_\Omega(\xi) & = \Phi_\Omega(a . \xi), \\
\Phi_\Omega(\xi) . \Phi_K(b) & = \Phi_\Omega(\xi . b), \\
\langle \Phi_\Omega(\eta), \Phi_\Omega(\xi) \rangle & = \Phi_K(\langle \eta, \xi \rangle).
\end{align*}
Moreover, $\Phi_\Omega$ has dense image by the same argument as in Proposition \ref{crossed product identification}; for each open bisection $U \subseteq G$ the image of $C_c(U;\Omega)$ in $\Gamma_c(U,s^* \mathcal E)$ is dense because the latter carries the supremum norm over $U$ and the image in each fibre is dense. We conclude that $\Phi_\Omega$ extends to an isomorphism of $\cs$-correspondences.
\end{proof}
\setlength{\parindent}{0cm} \setlength{\parskip}{0.5cm}

In summary, we have constructed the following diagram of functors, and shown that it commutes.
\[\begin{tikzcd}
	{{\mathrm{\acute{E}taleCorr}}^G } & {{\cs\mathrm{Corr}}^G} \\
	{{\mathrm{\acute{E}taleCorr}}} & {{\cs\mathrm{Corr}}.}
	\arrow["\cs", from=1-1, to=1-2]
	\arrow["{G \ltimes-}"', from=1-1, to=2-1]
	\arrow["{G \ltimes -}"', from=1-2, to=2-2]
	\arrow["\cs", from=2-1, to=2-2]
\end{tikzcd}\]
Here, ${\mathrm{\acute{E}taleCorr}}^G$ is the category of \'etale $G$-groupoids, with morphisms given by isomorphism classes of $G$-equivariant \'etale correspondences, and ${{\cs\mathrm{Corr}}^G}$ is the analogous category of $G$-$\cs$-algebras. Although there most likely exists a $2$-categorical refinement of this, we have no need for it in this paper.

\end{document}